\documentclass{amsart}
\usepackage{eurosym}
\usepackage{amssymb,amsmath,amsfonts}
\usepackage{tikz}
\usepackage[all]{xy}
\usepackage[unicode=true]{hyperref}

\newcommand*\circled[1]{\tikz[baseline=(char.base)]{
\node[shape=circle,draw,inner sep=0.5pt] (char) {#1};}}
\newtheorem{theorem}{Theorem}[section]
\newtheorem{corollary}[theorem]{Corollary}
\newtheorem{criterion}[theorem]{Criterion}
\newtheorem{definition}[theorem]{Definition}
\newtheorem{example}[theorem]{Example}
\newtheorem{fact}[theorem]{Fact}
\newtheorem{lemma}[theorem]{Lemma}
\newtheorem{remark}[theorem]{Remark}
\newtheorem{proposition}[theorem]{Proposition}

\def\Ker{\textup{Ker}}
\def\Tr{\textup{Tr}}

\def\red{\textup{red}}

\newcommand{\cref}[1]{eq. (\ref{#1})}

\begin{document}
\title[On matrix invertible extensions]{On matrix invertible
extensions over commutative rings}
\author{Grigore C\u{a}lug\u{a}reanu, Horia F. Pop, Adrian Vasiu}

\begin{abstract}
We introduce the class $E_{2}$ (resp.\ $SE_{2} $) of commutative rings $R$ with the property that each unimodular $2\times 2$ matrix with entries in $R$ extends to an invertible $3\times 3$ matrix (resp.\ invertible $3\times 3$ matrix whose $(3,3)$ entry is $0$). Among Dedekind domains of dimension 1, polynomial rings over $\mathbb Z$ or Hermite rings, only EDRs belong to the class. Using this, stable ranges, and units and projective modules interpretations, we reobtain (often refine) criteria due to Kaplansky, McGovern, Roitman, Shchedryk, and Wiegand and obtain new criteria for a Hermite ring to be an EDR. For instance, we characterize Hermite rings which are EDRs by (1) means of equations involving unimodular triples and (2) surjectivity of some maps involving units of factor rings by principal ideals. We use these criteria to show that each B\'{e}zout ring $R$ that is either a $J_{2,1}$ domain or an $(SU)_2$ ring (as introduced by Lorenzini) such that for each nonzero $a\in R$ there exists no nontrivial self-dual projective $R/Ra$-module of rank $1$ generated by $2$ elements (e.g., all its elements are squares), is an EDR, thus solving or partially solving problems raised by Lorenzini. Many other classes of rings are introduced and studied.
\end{abstract}

\maketitle

\section{Introduction}

\label{S1}

For a commutative ring $R$ with identity we denote by $U(R)$ its group of units, by $J(R)$ its Jacobson radical, and by $Pic(R)$ its Picard group. For $n\in\mathbb{N}=\{1,2,\ldots\}$ we denote by 
$\mathbb{M}_{n}(R)$ the ring of $n\times n$ matrices with entries in $R$, 
by $I_n$ its identity element, by ${GL}_{n}(R)$ its general linear group of units, and by ${SL}_{n}(R):=\{M\in GL_n(R)|\det(M)=1\}$ the special linear subgroup of ${GL}_{n}(R)$. 

For a free $R$-module $F $, let $Um(F)$ be the set of 
\textsl{unimodular} elements of $F$, i.e., of elements $v\in F$ for which
there exists an $R$-linear map $L:F\rightarrow R$ such that $L(v)=1$; we
have an identity $U(R)=Um(R)$ of sets.

For $A\in \mathbb{M}_{n}(R)$, let $A^{T}$ be its transpose and let $\chi_{A}(\lambda )\in R[\lambda ]$ be its (monic) characteristic polynomial. As we will be using pairs, triples and quadruples
extensively, the elements of $R^{n}$ will be denoted as $n$-tuples except
for the case of $R$-linear maps $L_{A}:R^{n}\rightarrow R^{n}$ defined 
by $A$, in which cases these will be viewed as $n\times 1$ columns. 
Let $K_{A}$ and $Q_{A}$ be the kernel and the image (respectively) of $L_{A}$. 

Recall that two matrices $A,B\in\mathbb M_n(R)$ are said to be equivalent 
if there exist matrices $M, N\in {GL}_n(R)$ such that $A=MBN$; if moreover we can 
choose $N=M^{-1}$, then $A$ and $B$ are said to be similar. We say that $A$ and $B$ are congruent modulo an ideal $I$ 
of $R$ if all entries of $A-B$ belong to $I$. By the reduction of $A$ modulo $I$ we mean the image of $A$ in $\mathbb M_n(R/I)\cong\mathbb M_n(R)/\mathbb M_n(I)$.

For $Q\in {SL}_3(R)$, let $\Theta_{R}(Q)\in \mathbb{M}_{2}(R)$ be the matrix obtained by removing the last row and the last column. The reductions of $\Theta _{R}(Q)$ modulo maximal ideals
of $R$ are nonzero, hence $\Theta _{R}(Q)\in Um(\mathbb{M}_2(R))$. The rule $Q\rightarrow \Theta _{R}(Q)$ defines a
map 
\begin{equation*}
\Theta:=\Theta _{R}:{SL}_3(R)\rightarrow Um(\mathbb{M}_2(R))
\end{equation*}and in this paper we study the image of $\Theta $ and in particular
the class of rings $R$ for which $\Theta$ is surjective (i.e., it
has a right inverse).

If $M=\left[ 
\begin{array}{cc}
a & b \\ 
c & d\end{array}\right]\in {GL}_{2}(R)$, let $\sigma(M)=\sigma_R(M):=\left[ 
\begin{array}{ccc}
a & b & 0\\ 
c & d & 0\\
0 & 0 & \det(M)^{-1}
\end{array}\right]\in {SL}_3(R)$. Clearly 
$\Theta(\sigma(M))=M$, hence $GL_{2}(R)\subseteq {Im}(\Theta)$.

\begin{definition}\label{def1}
We say that a matrix $A\in \mathbb{M}_2(R)$ is $SL_3$-extendable if there exists $A^{+}\in {SL}_3(R)$ such that $A=\Theta(A^+)$, and we call $A^{+}$ an $SL_3$-extension of $A$. If we can choose $A^{+}$ such that its $(3,3)$ is $0$, then we say that $A$ is simply $SL_3$-extendable and that $A^{+}$ is a simple $SL_3$-extension of $A$.
\end{definition}

As in this paper we do not consider $SL_n$-extensions with $n\ge 4$, in all that follows it will be understood
that ``extendable" means $SL_3$-extendable.

The problem of deciding if $A\in U(\mathbb{M}_2(R))$ has a (simple) invertible extension $A^{+}\in {SL}_3(R)$ 
relates to classical studies of finitely generated stable free modules that aim to extend 
$m\times (m+n)$ matrices whose entries are in $R$ and whose $m\times m$ minors 
generate $R$, to invertible matrices in ${GL}_{m+n}(R)$ (e.g., see \cite{LL}), but it fits better
within the general problem of finding $(m+n)\times (m+n)$ matrices with
have prescribed entries and coefficients of characteristic polynomials. To exemplify this, for $B\in\mathbb M_n(R)$ let $\nu_B:=\chi_B^{\prime}(0)\in R$ and let $\Tr(B)\in R$ be its trace; e.g., for $n=3$ we have $\chi_B(\lambda)=\lambda^3-\Tr(B)\lambda^2+\nu_B\lambda-1$. If $A\in\mathbb M_2(R)$ has a
simple extension $A^{+}$, then $A^{+}$ has 5 prescribed entries (out of 9) and $\chi_{A^{+}}(\lambda )=\lambda ^{3}-\Tr(A)\lambda ^{2}+\nu_{A^{+}}\lambda -1$
has 2 prescribed coefficients (out of 3); the nonempty subsets 
\begin{equation*}
\nu_{A}:=\{\nu_{A^{+}}|A^{+}\;is\; a\; simple\; extension\; of\; A\}\subseteq R
\end{equation*} are sampled in Remark \ref{rem1}(5) and Subsection \ref{S51}. Over fields, such general problems have a long history 
and often very complete results (e.g., see \cite{FL} and \cite{her}).

We will use the shorthand `iff' for `if and only if' in all that follows.

The following lemma proved in Subsection \ref{S30} presents basic facts.

\begin{lemma}\label{L1}
The following hold:

\medskip \textbf{(1)} A matrix $A\in \mathbb{M}_2(R)$ is extendable iff its reduction modulo $R\det (A)$ is simply extendable. Thus, if $\det (A)=0$, then $A$ is extendable iff it is simply extendable.

\medskip \textbf{(2)}  For $M,N\in {GL}_2(R)$ and $Q\in {SL}_3(R)$ we have an identity 
\begin{equation*}
\Theta(\sigma(M)Q\sigma(N))=M\Theta(Q)N
\end{equation*}and the $(3,3)$ entries of $Q$ and $\sigma(M)Q\sigma(N)$ 
are equivalent (thus one such entry is $0$ iff
the other entry is $0$). Also, $\Theta(Q^{T})=\Theta(Q)^{T}$.

\medskip \textbf{(3)} The fact that a matrix $A\in \mathbb{M}_2(R)$ is
extendable (or simply extendable) depends only on the double coset $[A]\in {GL}_2(R)\backslash \mathbb{M}_2(R)/GL_2(R)$. Moreover, $A$ is extendable (or simply extendable) iff so is $A^{T}$. So ${Im}(\Theta)$ is stable under transposition and
equivalence.
\end{lemma}

In general $\Theta$ is not surjective. In fact, for each $n\in 
\mathbb{N}$, there exists a regular finitely generated $\mathbb{Z}$-algebra $R$ of dimension $n$ such that a suitable $A\in Um(\mathbb{M}_{2}(R))$ is not
extendable (see Theorem \ref{TH4}(4) and Examples \ref{EX11}, \ref{EX13} 
and \ref{EX16}); if $n\geq 2$, we can even assume that $R$ is a \textsl{UFD}
(unique factorization domain).

For a class $Z$ of rings, $SZ$ (resp.\ $WZ$) will often denote a subclass (resp.\ a larger class), where $S$ (resp.\ $W$) stands for `strongly' or `simply' (resp.\ `weakly'). If $R$ is a $Z$ ring and an integral domain, we will say that $R$ is a $Z$ domain.

In this paper we study thirteen classes of rings to be named, in the
spirit of \cite{GH1} and \cite{lor}, using indexed letters. Definition \ref{def1} leads naturally to the first five classes.

\begin{definition}\label{def2}
We say that $R$ is:

\medskip \textbf{(1)} a ${\Pi}_2$ ring, if each $A\in Um(\mathbb{M}_2(R))$ with $\det(A)=0$ is
extendable;

\smallskip \textbf{(2)} an ${E}_2$ ring, if each matrix in $Um(\mathbb{M}_{2}(R))$ is extendable (i.e., if $\Theta$ is
surjective);

\smallskip \textbf{(3)} an ${SE}_2$ ring, if each matrix in $Um(\mathbb{M}_{2}(R))$ is simply extendable;

\smallskip \textbf{(4)} an ${E}^{\triangle}_2$ ring, if each upper (or
lower) triangular matrix in $Um(\mathbb{M}_2(R))$ is extendable;

\smallskip \textbf{(5)} an ${SE}^{\triangle}_2$ ring, if each upper
(or lower) triangular matrix in $Um(\mathbb{M}_2(R))$ is
simply extendable.
\end{definition}

Clearly, each ${SE}_2$ ring is an ${SE}^{\triangle}_2$ ring. We speak about $\Pi_2$ domains, etc.

We recall that $R$ is called an \textsl{elementary divisor ring}, abbreviated as \textsl{EDR}, if for all $m,n\in \mathbb{N}$, each $m\times n$
matrix with entries in $R$ \textsl{admits diagonal reduction}, i.e., is
equivalent to a matrix whose off diagonal entries are $0$ and whose diagonal
entries $a_{1,1},\ldots ,a_{s,s}$, with $s:=\min \{m,n\}$, are such that $a_{i,i}$
divides $a_{i+1,i+1}$ for all $i\in \{1,\ldots ,s-1\}$. Equivalently, $R$ is
an \textsl{EDR} iff each matrix in $\mathbb{M}_2(R)$ is
equivalent to a diagonal matrix, iff each finitely presented $R $-module is a direct sum of cyclic $R$-modules (see \cite{LLS}, Cor. (3.7)
and Thm. (3.8); see also \cite{WW}, Thm. 2.1). An \textsl{EDR} domain will be abbreviated as \textsl{EDD}.

\begin{example}
\normalfont\label{EX1} Let $A\in Um(\mathbb{M}_2(R))$. If there
exist matrices $M,N\in {GL}_2(R)$ such that $MAN=\left[ 
\begin{array}{cc}
1 & 0\\ 
0 & \det(A)
\end{array}
\right]$ (e.g., this holds if $R$ is an \textsl{EDR}),
then the product $\sigma(M^{-1})\left[ 
\begin{array}{ccc}
1 & 0 & 0\\ 
0 & \det(A) & 1\\
0 & -1 & 0
\end{array}\right] \sigma(N^{-1})$
is a simple extension of $A$ by Lemma \ref{L1}(2). Thus each \textsl{EDR} is an $SE_2$ ring.
\end{example}

The following classification of ${\Pi}_2$ rings is proved in Subsection \ref{S41}.

\begin{theorem}
\label{TH1} For a ring $R$ the following statements are equivalent:

\medskip \textbf{(1)} The ring $R$ is a ${\Pi}_{2}$ ring.

\smallskip \textbf{(2)} Each matrix in $Um(\mathbb{M}_{2}(R))$ of zero
determinant is non-full, i.e., the product of two matrices of sizes $2\times
1$ and $1\times 2$ (equivalently, the $R$-linear map $L_A$ factors as a
composite $R$-linear map $R^2\rightarrow R\rightarrow R^2$).

\smallskip \textbf{(3)} For each matrix in $Um(\mathbb{M}_{2}(R))$ of zero
determinant, one (hence both) of the $R$-modules $K_A$ and $Q_A$ is
isomorphic to $R$.

\smallskip \textbf{(4)} Each projective $R$-module of rank $1$ generated by
two elements is isomorphic to $R$.
\end{theorem}

For pre-Schreier domains see Subsection \ref{S22}. Recall that an integral domain $R$ is a pre-Schreier domain iff each matrix in $\mathbb{M}_{2}(R)$ of zero determinant
is non-full (see \cite{CP}, Thm. 1 or \cite{MR}, Lem. 1). Thus each
pre-Schreier domain is a $\Pi _{2}$ domain. Similarly, if $Pic(R)$ 
is trivial, then $R$ is a ${\Pi }_{2}$ ring.

The following notions introduced by Bass (see \cite{bas}), Shchedryk (see 
\cite{shc1}), and McGovern (see \cite{mcg1}) (respectively) will be often
used in what follows.

\begin{definition}\label{def3}
Let $n\in\mathbb{N}$. Recall that $R$ has:

\medskip \noindent \textbf{(1)} \textsl{stable range} n and we write $sr(R)=n $ if $n$ is the smallest natural number with the property that each $(a_1,\ldots,a_n,b)\in Um(R^{n+1})$ is reducible, i.e., there exists $(r_1,\ldots,r_n)\in R^n$ such that $(a_1+br_1,\ldots,a_n+br_n)\in Um(R^n)$ 
(when there exists no $n\in\mathbb N$ such that $sr(R)=n$, then 
$sr(R):=\infty$ and the convention is $\infty>n$);

\smallskip \noindent \textbf{(2)} \textsl{(fractional) stable range 1.5} and 
we write $fsr(R)=1.5$ if for each $(a,b,c)\in Um(R^{3})$ with $c\neq 0$ 
there exists $r\in R$ such that $(a+br,c)\in Um(R^{2})$;

\smallskip\noindent \textbf{(3)} \textsl{almost stable range 1} and we
write $asr(R)=1$ if for each ideal $I$ of $R$ not contained in $J(R)$, $sr(R/I)=1$.
\end{definition}

Stable range type of conditions on all suitable unimodular tuples 
of a ring date back at least to Kaplansky. For instance, the essence of 
\cite{kap}, Thm. 5.2 can be formulated in the language of this paper 
as follows: a ring $R$ is an $SE_2^{\triangle}$ ring iff for all 
$(a,b,c)\in Um(R^{3})$, there exists $(e,f)\in R^2$ such that 
$(ae,be+cf)\in Um(R^2)$. Theorem \ref{TH8}, among many other things, 
proves the general (nontriangular) form of loc. cit:  a ring $R$ is an $SE_2$ 
ring iff for all $(a,b,c,d)\in Um(R^4)$, there exists $(e,f)\in R^2$ 
such that $(ae+cf,be+df)\in Um(R^2)$.

We recall that if $R$ has (Krull) dimension $d$, then $sr(R)\le d+1$ (see \cite{bas}, Thm. 11.1 for the noetherian case and see \cite{hei}, Cor. 2.3 for the general case). Also, if $R$ is a finitely generated algebra over a finite field of dimension $2$ 
or if $R$ is a polynomial algebra in 2 indeterminates over a field that is algebraic 
over a finite field, then $sr(R)\le 2$ (see \cite{VS}, 
Cors. 17.3 and 17.4).
 
Each ${SE}_{2}$ ring is an ${E}_2$ ring, but we do not know when
the converse is true. However, the converse holds if $sr(R)\le 2$ (see Corollary \ref{C8}(3)) and Example \ref{EX14} shows that there exist $2\times 2$ matrices that are extendable but are not simply extendable.

The next two classes of rings are motivated by Lemma \ref{L1}(1) and the class $\Pi_2$.

\begin{definition}\label{def2-}
We say that $R$ is:

\medskip \textbf{(1)} a ${Z}_2$ ring, if for each
matrix $A\in Um(\mathbb{M}_2(R))$ there exists a matrix $B\in
Um(\mathbb{M}_2(R))$ congruent to $A $ modulo $R\det (A)$ and $\det (B)=0$;

\smallskip \textbf{(2)} a ${WZ}_2$ ring, if for each
matrix $A\in Um(\mathbb{M}_2(R))$ there exists a matrix $B\in \mathbb{M}_2(R)$ congruent to $A $ modulo $R\det (A)$ and $\det (B)=0$.
\end{definition}

We have the following classification of ${Z}_2$ rings proved in Subsection \ref{S410}.

\begin{theorem}
\label{TH2} For a ring $R$ the following three statements are equivalent:

\medskip \textbf{(1)} The ring $R$ is a $Z_2$ ring. 

\smallskip \textbf{(2)} For each matrix $A\in
Um(\mathbb{M}_2(R))$ there exists a matrix $C\in Um(\mathbb{M}_2(R))$ such that $A+\det (A)C\in Um(\mathbb{M}_2(R))$ and moreover 
$\det(C)=\det (A+\det (A)C)=0$.

\smallskip \textbf{(3)} For each matrix $A\in \mathbb{M}_2(R)$ there exists a matrix $C\in Um(\mathbb{M}_2(R))$ such that $A+\det (A)C\in Um(\mathbb{M}_2(R))$ and moreover $\det(C)=\det (A+\det (A)C)=0$.

\medskip
Moreover, these three statements are implied by the following fourth one:

\medskip \textbf{(4)} For each nonzero $a\in R$, the map of sets
$$\{A\in Um(\mathbb M_2(R))|\det(A)=0\}\rightarrow \{\bar{A}\in Um(\mathbb M_2(R/Ra))|\det(\bar{A})=0\},$$
defined by the reduction modulo $Ra$, is surjective.

\medskip
If $sr(R)\le 4$, then these four statements are equivalent.
\end{theorem}

If $sr(R)=1$ or $asr(R)=1$ or $fsr(R)=1$, there exist classical units interpretations (see Corollary \ref{C5}), so the last six classes of rings provide a framework for generalizing these interpretations to contexts such as Hermite rings and $sr(R)=2$.

\begin{definition}\label{def2+}
We say that $R$ is:

\medskip \textbf{(1)} a $U_2$ ring if for each $(a,b)\in Um(R^{2})$ and $c\in R$, the natural product
homomorphism 
\begin{equation*}
U(R/Rac)\times U(R/Rbc)\rightarrow U(R/Rc)
\end{equation*}is surjective, i.e., the commutative diagram

\begin{equation}\label{EQ1}
\xymatrix@R=10pt@C=21pt@L=2pt{
U(R/(Rac\cap Rbc)) \ar[r]^{\;\;\;\red} \ar[d]_{\red} & U(R/Rac)\ar[d]^{\red}\\
U(R/Rbc) \ar[r]^{\;\;\;\red} & U(R/Rc)\\}
\end{equation}
whose arrows are natural reductions, is not only a pullback but is also a
pushout;

\smallskip \textbf{(2)} a ${WU}_{2}$ ring if for each $(a,b)\in Um(R^2)$ and $c\in R$, the natural product
homomorphism of (1) contains $\{x^2|x\in U(R/Rc)\}$;

\smallskip \textbf{(3)} a $V_2$ ring if for each $(a,b,c)\in Um(R^{3})$ the
equation 
\begin{equation}
(1-ax)(1-cw)=y(b+acz)  \label{EQ2}
\end{equation}in the indeterminates $x,y,z,w$ has a solution in $R^{4}$ with $xw=yz$ (note
that for each such solution we have $(x,y,w),(1-ax,y,1-cw)\in Um(R^{3}))$;

\smallskip \textbf{(4)} a $WV_2$ ring if for each $(a,b,c)\in Um(R^{3})$,
there exist elements $x$, $y$, $z_1$, $z_2$, $w$, $a^{\prime}$, $c^{\prime}$, $a^{\prime\prime}$ and $c^{\prime\prime}$ of $R$ such that $(a^{\prime\prime}-ax,y,a^{\prime}c^{\prime}b+az_1+cz_2,c^{\prime\prime}-cw)\in Um(R^{4})$, $(a,a^{\prime}a^{\prime\prime}),(c,c^{\prime}c^{\prime\prime})\in U(R^2)$, and moreover 
\begin{equation}
(a^{\prime\prime}-ax)(c^{\prime\prime}-cw)=y(a^{\prime}c^{\prime}b+az_1+cz_2);  \label{EQ3}
\end{equation}

\smallskip \textbf{(5)} a ${W}_2$ ring if for each $(a,b)\in Um(R^2)$ and $c\in R$, the image of the homomorphism $U(R/Rabc))\rightarrow U(R/Rab)\cong
U(R/Ra)\times U(R/Rb)$ is the product of the images of the two homomorphisms $U(R/Rabc)\rightarrow U(R/Ra)$ and $U(R/Rabc)\rightarrow U(R/Rb)$;

\smallskip \textbf{(6)} a ${WW}_2$ ring if for each $(a,b,c)\in R^3$ with $(b,c)\in Um(R^2)$, the image $(b_0,b_1)$ of $b$ in the product group $U:=U(R/(Ra+Rc))\times
U(R/[R(1-a)+Rc])$ has the property that $(b_0,b_1^{-1})\in U$ belongs to the
product of the images of the natural two reduction homomorphisms $U(R/Ra)\times U(R/R(1-a))\rightarrow U$ and $U(R/Rc)\rightarrow U$.
\end{definition}

Each $V_2$ ring is a $WV_2$ ring as Equation (\ref{EQ3}) becomes Equation (\ref{EQ2}) if in Definition \ref{def2+}(4) one takes $a^{\prime}=c^{\prime}=a^{\prime\prime}=c^{\prime\prime}=1$, $z_2=0$ and $z_1=zc$.

\begin{remark}\label{rem0}
Referring to Definition \ref{def2+}, we note that $R$ is a $U_2$ ring (resp.\ a $WU_2$ ring) iff (1) (resp.\ (2)) holds only for triples $(a,b,c)\in R^3$ with $a+b=1$. Similarly, $R$ is a $WW_2$ ring iff (6) holds only for triples $(a,b,c)\in R^3$ with $c-1$ divisible by $b$.
\end{remark}

Paper's leitmotif is to identify equations that `encode' arithmetic and stable range properties of rings, and this explains the introduction of classes of rings such as $SE_2$, $V_2$ and $Z_2$ (see also the $R$-algebras of Subsection \ref{S31}).

The next theorem proved in Subsection \ref{S42} groups together important connections between the classes of rings we have introduced.

\begin{theorem}
\label{TH3} The following statements hold:

\medskip\textbf{(1)} Each $E_2$ ring is a ${WZ}_2$ ring and each ${SE}_2$
ring is a ${Z}_2$ ring. Moreover, if $sr(R)\le 2$, then each ${E}_2$ ring
is a ${Z}_2$ ring.

\smallskip \textbf{(2)} Let $R$ be a product of pre-Schreier domains. Then $R $ is an $E_2$ ring iff it is a ${WZ}_{2}$ ring.

\smallskip \textbf{(3)} If $R$ is a $\Pi_2$ ring, then $R $
is an ${SE}_{2}$ ring iff it is a ${Z}_{2} $ ring. 

\smallskip \textbf{(4)} If $R$ is an ${SE}^{\triangle}_2$ ring, then $R$ is a 
${V}_2$ ring. If $R$ is an ${SE}^{\triangle}_2$ domain, then $R$ is a $U_2$
ring. If $R$ is a $\Pi_2$ ring and a $WV_2$ ring, then $R$ is a $U_2$ ring.

\smallskip \textbf{(5)} Assume that $sr(R)\le 4$. If $R$ is an $E_{2}$ (resp.\ a $SE_2$) ring,
then $R/Ra$ is an $E_{2}$ (resp.\ a $SE_2$) ring for all $a\in R$.

\smallskip \textbf{(6)} If $R/Ra$ is a $\Pi_2$ ring for all $a\in R$, then $R $ is an $E_2$ ring.
\end{theorem}

For domains of dimension 1, based on Theorems \ref{TH1} and \ref{TH3}(4), in Subsection \ref{S43} we prove the following theorem.

\begin{theorem}
\label{TH4} Let $R$ be an integral domain of dimension 1. Then the following properties hold:

\medskip \textbf{(1)} Each matrix in $Um(\mathbb{M}_{2}(R))$ with nonzero determinant is
simply extendable.

\smallskip \textbf{(2)} The ring $R$ is an ${SE}^{\triangle}_2$ domain, 
a ${V}_2$ domain, and a $Z_2$ domain.

\smallskip \textbf{(3)} The ring $R$ is a ${\Pi}_2$ domain iff it is an ${SE}_2$ (or an $E_2$) domain and iff $Pic(R)$ is trivial. 

\smallskip \textbf{(4)} Assume $R$ is a Dedekind domain. The ring $R$ is a ${\Pi}_2$ domain iff it is a principal ideal domain (\textsl{PID}).
\end{theorem}

Recall that $R$ is called a \textsl{valuation ring} if for each $(a,b)\in R^2$, either $a$ divides 
$b$ or $b$ divides $a$ (see \cite{kap}, Sect. 10, Def.), equivalently, if the ideals 
of $R$ are totally ordered by set inclusion. Each reduced valuation ring is an 
integral domain\footnote{If $R$ is a reduced local ring which is not an integral 
domain, then there exists nonzero elements $a,b\in R$ such that $ab=0$, 
hence the intersection $Ra\cap Rb$, being nilpotent, is $0$, and it follows that 
the finitely generated ideal $Ra+Rb\cong Ra\oplus Rb$ is not principal.} and 
hence a valuation domain (this is already stated in \cite{kap}, Sect. 10).

Recall that $R$ is a \textsl{B\'{e}zout ring} if each finitely
generated ideal of it is principal. Each B\'{e}zout ring is an \textsl{arithmetical ring}, i.e., its lattice of ideals is distributive and all its localizations at 
prime ideals are valuation rings (see \cite{jen}, Thms. 1 and 2). It is known that $R$ is a B\'{e}zout ring iff each diagonal matrix with entries in $R$ admits diagonal
reduction (see \cite{LLS}, Thm. (3.1)). Also, recall that $R$ is a \textsl{Hermite} ring in the sense of Kaplansky, if $RUm(R^{2})=R^{2}$, equivalently
if each $1\times 2$ matrix with entries in $R$ admits diagonal reduction.
Thus each \textsl{EDR} is a Hermite ring. If $R$ is a Hermite ring, then $sr(R)\in \{1,2\}$ (see \cite{MM}, Prop. 8(i); see also \cite{zab1}, Thm.
2.1.2). Note that a Hermite ring is a B\'{e}zout ring and the converse does
not hold (see \cite{GH1}, Ex. 3.4, \cite{WW}, Ex. 3.3, or \cite{car}, Prop.
8), but a Hermite domain is the same as a B\'{e}zout domain. Also, $R$ is a
Hermite ring iff for all $p,q\in \mathbb{N}$ and each $p\times q$
matrix $B$ with entries in $R$, there exist $M\in {GL}_{p}(R)$ and $N\in \ {GL}_{q}(R)$ such that $MB$ and $BN$ are both lower
(equivalently, upper) triangular (see \cite{GH2}, Thm. 3). If $R$ is a 
B\'{e}zout domain, then it is a \textsl{GCD}
(greatest common divisors exist) domain and hence $Pic(R)$ is
trivial; thus $R$ is also a $\Pi_2$ domain. 

If $R$ is a \textsl{GCD} domain, for $A=\left[ 
\begin{array}{cc}
a & b \\ 
c & d\end{array}\right] \in \mathbb{M}_{2}(R)$ we have $A\in Um(\mathbb{M}_{2}(R))$ iff its entries are \textsl{collectively coprime}, i.e., $\gcd
(a,b,c,d)=1$.

\smallskip B\'{e}zout domains that have stable range 1 or 1.5 were also
studied in \cite{rus} and respectively \cite{shc3}, \cite{shc4} and \cite{BS}. Commutative PIDs, semilocal domains (in particular, valuation domains) and B\'{e}zout domains that have stable range 1, have stable range 1.5 and thus
also have almost stable range 1 (see Corollary \ref{C4}). For readers
convenience, details are given in Subsection \ref{S23} (see
Examples \ref{EX5} and \ref{EX6}).

Our applications to B\'{e}zout domains and more generally to Hermite rings
are grouped together in the following theorem proved in Subsection \ref{S45}.

\begin{theorem}
\label{TH5} Let $R$ be a Hermite ring. Then the following ten statements are
equivalent:

\medskip \textbf{(1)} The ring $R$ is an \textsl{EDR}.

\smallskip \textbf{(2)} The ring $R$ is an ${SE}_2$ (equivalently, an ${SE}^{\triangle}_2 $) ring.

\smallskip \textbf{(3)} The ring $R$ is a ${\Pi}_2$ ring and a ${Z}_2$ ring.

\smallskip \textbf{(4)} The ring $R$ is an ${E}_2$ ring.

\smallskip \textbf{(5)} The ring $R$ is a ${\Pi}_2$ ring and a ${V}_2$ ring.

\smallskip \textbf{(6)} The ring $R$ is a ${\Pi}_2$ ring and a ${WV}_2$ ring.

\smallskip \textbf{(7)} The ring $R$ is a $U_{2}$ ring.

\smallskip \textbf{(8)} Given two unimodular pairs $(a,b),(c,d)\in Um(R^2)$, there exists $t\in R$ such that we can factor $d+ct=d_1d_2$ with $(a,d_1),(b,d_2)\in Um(R^2)$.

\smallskip \textbf{(9)} Given $(a,d)\in R^2$, for $b:=1-a$ and $c\in 1+Rd$ 
there exists $t\in R$ such that we can factor $d+ct=d_1d_2$ with $(a,d_1),(b,d_2)\in Um(R^2)$.

\smallskip \textbf{(10)} For each $a\in R$, $R/Ra$ is a $\Pi_2$ ring
(equivalently, for each $a\in R$, every projective $R/Ra$-module of rank $1$
generated by two elements is free).

\medskip Moreover, these ten statements are implied by the following eleventh one:

\medskip \textbf{(11)} The ring $R$ is a ${W}_2$ domain.
\end{theorem}

\begin{example}
\normalfont\label{EX2} Let $R$ be a Hermite ring which is not an \textsl{EDR}
(see \cite{GH1}, Sect. 4 and Ex. 4.11). From Theorem \ref{TH5} it follows
that $R$ is neither an $E_{2}$ ring nor a $U_{2}$ ring and there exists $a\in R$ such that $R/Ra$ is not a $\Pi_2$ ring.
\end{example}

For the next applications we recall some of the classes of rings introduced
by Lorenzini in \cite{lor}, Lem. 4.1 and Def. 4.6, including their `weakly' generalizations.

\begin{definition}\label{def4}
Let $n\in\mathbb{N}$. We say that $R$ is:

\medskip \textbf{(1)} a ${J}_{2,1}$ (resp.\ ${WJ}_{2,1}$) ring if for each $(a,b,c,d,\alpha,\Delta)\in R^{6}$ such that the equation
\begin{equation}
ax+by+cz+dw=\alpha  \label{EQ5}
\end{equation}in the indeterminates $x,y,z,w$ has a solution in $R^{4}$ (resp.\ such that $(a,b,c,d)\in Um(R^{4})$), Equation (\ref{EQ5}) has a solution in $R^{4}$
with $xy-zw=\Delta $;

\smallskip \textbf{(2)} an ${(SU^{\prime})}_{n}$ (resp.\ a ${(WSU^{\prime})}_{n}$) ring if for each $A\in\mathbb{M}_n(R)$ (resp.\ for each $A\in Um(\mathbb{M}_n(R))$), there exists $N\in {SL}_n(R)$ such that $AN$ is symmetric
(equivalently, there exists $M\in {SL}_n(R)$ such that $MA$ is
symmetric, equivalently, there exist $M,N\in {SL}_n(R)$ such that $MAN
$ is symmetric);

\smallskip \textbf{(3)} an ${(SU)}_{n}$ (resp.\ a ${(WSU)}_{n}$) ring if for each 
$A\in\mathbb{M}_n(R)$ (resp.\ for each $A\in Um(\mathbb{M}_n(R))$), there
exists $N\in {GL}_n(R)$ such that $AN$ is symmetric (equivalently,
there exists $M\in {GL}_n(R)$ such that $MA$ is symmetric,
equivalently, there exist $M,N\in {GL}_n(R)$ such that $MAN$ is
symmetric);

\smallskip \textbf{(4)} an ${E}_2^{sym}$ (resp.\ ${SE}_2^{sym}$) ring if each symmetric $A\in Um(\mathbb{M}_2(R))$ is
extendable (resp.\ simply extendable).
\end{definition}

\medskip The equivalences of (2) and (3) above are a consequence of the fact that for each symmetric matrix $A\in\mathbb M_n(R)$ and every $M\in {GL}_n(R)$, $MAM^T$ is symmetric.

By taking $(a,b,d)\in Um(R^3)$ and $(c,\alpha,\Delta)=(0,1,0)$ in (1), it
follows that each ${WJ}_{2,1}$ ring is a ${V}_2$ ring. As each ${J}_{2,1}$ ring is a Hermite ring (see \cite{lor}, Prop. 4.11), every $J_{2,1}$ domain is a B\'{e}zout domain and hence a $\Pi_2$ domain. From this and the equivalence $(1)\Leftrightarrow (5)$ in Theorem \ref{TH5} it follows:

\begin{corollary}
\label{C1} Let $R$ be a ${J}_{2,1}$ ring. Then $R$ is an \textsl{EDR} iff it is a $\Pi_2$ ring (thus each ${J}_{2,1}$ domain is an \textsl{EDD}).
\end{corollary}

Therefore the constructions for an arbitrary commutative ring performed in \cite{lor}, Ex. 3.5 always produce rings which are either \textsl{EDR}s or are not $\Pi_2$ rings (in particular, the integral domains produced are \textsl{EDD}s).

Each $(SU^{\prime})_n$ ring is an $(SU)_n$ ring. Also, a $(WSU)_2$ ring is an $E_2$ (resp.\ $SE_2$ ring) iff it is an ${E}_2^{{sym}}$ (resp.\ ${SE}_2^{{sym}}$) ring. Moreover, each $(SU)_n$ ring is a Hermite ring (see \cite{lor}, Prop. 3.1). In Subsection \ref{S46} we prove that:

\begin{theorem}
\label{TH6} Let $R$ be a $(WSU)_2$ ring. Then the following properties
hold:

\medskip \textbf{(1)} The ring $R$ is a $WU_2$ ring.

\smallskip \textbf{(2)} Assume that each element of $R$ is the
square of an element of $R$ (e.g., this holds if $R$ is an integrally closed
domain with an algebraically closed field of fractions or is a perfect ring 
of characteristic $2$). Then $R$ is a $U_{2}$ ring. If moreover $R$ is an 
$(SU)_{2}$ ring, then $R$ is an \textsl{EDR}.
\end{theorem}

See Example \ref{EX12} for an interpretation of Theorem \ref{TH6}(1) in terms of elements of orders 1 or $2$ in Picard groups. Directly from it and Theorems \ref{TH5} and \ref{TH6} we get:

\begin{corollary}
\label{C2} Let $R$ be a Hermite ring such that for each $a\in R$, all self-dual projective $R/Ra$-modules of rank $1$ generated by $2$ elements are free. Then the following properties hold:

\medskip \textbf{(1)} The ring $R$ is a $WU_2$ ring iff it is an \textsl{EDR}. 

\smallskip \textbf{(2)} If $R$ is an $(SU)_2$ ring, then $R$ is an \textsl{EDR}. 
\end{corollary}

To make connection with Pell-type equations and to provide more examples of $(WSU)_2$ rings which are \textsl{EDR}s, we first prove in Subsection \ref{S47} the following non-full Pell-type criterion.

\begin{criterion}
\label{CR1} Let $A=\left[ 
\begin{array}{cc}
a & b \\ 
b & c\end{array}\right]\in Um(\mathbb{M}_2(R))$ be symmetric and with zero determinant. Then the
following properties hold:

\medskip \textbf{(1)} The matrix $A$ is simply extendable if there exists $(e,f)\in R^{2}$ such that $ae^{2}-cf^{2}\in U(R)$, and the converse holds if 
$R$ has characteristic $2$.

\smallskip \textbf{(2)} Assume that $R$ is a reduced Hermite ring of
characteristic $2$. If $b$ is not a zero divisor, then $A$ is simply
extendable.
\end{criterion}

By combining Theorem \ref{TH5} with Criterion \ref{CR1}, we obtain the
following Pell-type criterion proved in Subsection \ref{S48}.

\begin{criterion}
\label{CR2} Let $R$ be an $(SU)_2$ ring. Then $R$ is an \textsl{EDR} if for
all $(a,b,c)\in Um(R^3)$ there exists $(e,f)\in R^2$ such that $(ae^2-cf^2,ac-b^2)\in Um(R^2)$, and the converse holds if $R$ has
characteristic $2$. 
\end{criterion}

\begin{example}
\normalfont\label{EX3} Assume that $R$ has characteristic $2$. We recall
that its perfection $R_{{perf}}$ is the inductive limit of the
inductive system indexed by $n\in\mathbb{N}$ whose all transition
homomorphisms are the Frobenius endomorphism of $R$. If $R$ is a $WU_{2}$
ring, then $R_{{perf}}$ is a $U_{2}$ ring. From this and Theorem \ref{TH6}(2) it follows that the perfections of $(WSU)_{2}$ rings of characteristic $2$ are $U_{2}$ rings.
\end{example}

For almost stable range 1 we have the following applications:

\begin{corollary}
\label{C3} Assume that $asr(R)=1$. Then the following properties hold:

\medskip \textbf{(1)} The ring $R$ is an $SE_2^{\triangle}$ ring.

\smallskip \textbf{(2) (McGovern)} If $R$ is a Hermite ring, then $R$ is an 
\textsl{EDR}.
\end{corollary}

Corollary \ref{C3}(1) is proved in Subsection \ref{S34}. Corollary \ref{C3}(2) follows directly from Corollary \ref{C3}(1) and the equivalence $(1)\Leftrightarrow (2)$ of Theorem \ref{TH5}. If $asr(R)=1$, then $R$ is an $U_2$ ring (see Corollary \ref{C4}(4)), so Corollary \ref{C3}(2)
also follows from this and the equivalence $(1)\Leftrightarrow (7)$ of Theorem \ref{TH5}. A third proof of Corollary \ref{C3}(2) is presented in  Remark \ref{rem5}. Corollary \ref{C3}(2) was first obtained in \cite{mcg1},
Thm. 3.7.

Kaplansky proved that each B\'{e}zout domain of Krull dimension at most $1$
is an \textsl{EDD} (e.g., see \cite{sho}, Cor.). In \cite{gat}, Thm. 1 it is
claimed that the same holds for Hermite rings and \cite{gat}, Thm. 2 checks that 
if \cite{gat}, Thm. 1 holds, then each B\'{e}zout ring which is semihereditary 
and of Krull dimension $2$ is an \textsl{EDR}. As the proof of \cite{gat}, Thm. 1 
is incomplete, Examples \ref{EX17} and \ref{EX18} prove weaker results under 
suitable $\Pi_2$ ring assumptions.
Thus we still do not know if each B\'{e}zout domain of Krull dimension $2$
is an \textsl{EDD}. However, we have the following new criterion which is
proved in Subsection \ref{S49} and which is a direct application of Theorem \ref{TH5}:

\begin{criterion}
\label{CR3} For a B\'{e}zout domain $R$ the following statements are equivalent:

\medskip \textbf{(1)} The ring $R$ is an \textsl{EDD}.

\smallskip \textbf{(2)} For all triples $(a,b,s)\in R^3$, there exists a
pair $(q,r)\in R^2$ such that by defining $y:=r+s-asq-bqr$ and $t:=1+q-aq-br$
we have $t\in Ry+Rat$ (equivalently, there exists a product decomposition $y=y_1y_2$ such that $y_1$ divides $t$ and $(a,y_2)\in Um(R^2)$).

\smallskip \textbf{(3)} For each triple $(a,b,s)\in R^3$ there exists $(e,f)\in Um(R^2)$ such that $(a,e)$, $(be+af,1-bs-a)\in Um(R^2)$.
\end{criterion}

E.g., statement (3) holds if $(a,s)\in Um(R^2)$ as we can take $(e,f):=(s,1)$, if $(1-a,b)\in Um(R^2)$ as we can take $(e,f):=(1,0)$, or if there exists $q\in R$ such that $(b+aq,1-bs-a)\in Um(R^2)$ as we can take $(e,f):=(1-a,q+b)$. So, each B\'{e}zout domain $R$ with the property that for all $(a,b,s)\in R^3$ with $(a,s),(b,1-a)\notin Um(R^2)$ there exists $q\in R$ such that $(b+aq,1-bs-a)\in Um(R^2)$, is an \textsl{EDD}.

\subsection{More on literature and paper's structure}

\label{S15}

The implicit and explicit questions raised in the literature, such as,
\textquotedblleft Is a B\'{e}zout domain of finite Krull dimension [at least 
$2$] an \textsl{EDD}?" (see \cite{FS}, Ch. III, Probl. 5, p. 122), and,
`What classes of B\'{e}zout domains which are not \textsl{EDD}s exist?',
remain unanswered. However, Theorem \ref{TH5} reobtains or can be easily used to
reobtain multiple other criteria existing in the literature of when a Hermite 
ring is an \textsl{EDR}. For instance, the equivalence $(1)\Leftrightarrow (10)$ of Theorem \ref{TH5} can be viewed as a
strengthening of \cite{WW}, Thm. 2.1 or \cite{ZB}, Thms. 3 and 5, where one
requires many more finitely generated projective modules to be free, as well
as it can be used to get new proofs of these references.

The equivalence $(1)\Leftrightarrow (8)$ of Theorem \ref{TH5} was first proved for B\'{e}zout domains in \cite{shc2}, Thm. 5 and a variant of it was first proved for Hermite rings in \cite{rot}, Prop. 2.9. Equation (\ref{EQ2}) generalizes and refines the equation one would get based on \cite{rot}, Rm. 2.8. The last two references were reinterpreted in the form of neat stable range
1 in \cite{zab2}, Thms. 31 and 33 (see \cite{mcg2} for clean and neat rings
and see \cite{zab2}, Def. 21 for rings of neat stable range 1).

The extra terminology required in the proofs of the above results is presented in Section \ref{S2}. Section \ref{S3} presents basic
criteria of when a unimodular $2\times 2$ matrix is (simply) extendable.
The subsections of Section \ref{S4} prove the above results. Section \ref{S5} contains concrete extensions and computations over integral
domains.

The universal rings related to $E_2$, ${SE}_2$, ${Z}_2$ and ${WZ}_{2}$
rings are presented in Section \ref{S6}. Section \ref{S7} presents stability properties and shows that $R$ is an ${SE}_{2}$
(or an ${E}_2$ or ${U}_2$) ring iff $R[[x]]$ is so. Section \ref{S8} lists open problems/questions.

\section{Basic terminology and properties}

\label{S2}

Let $N(R)$ denote the nilradical of $R$ and let $Z(R)$ denote the set of
zero divisors of $R$. Let 
${Spec}R$ be the spectrum of $R$ and let ${Max}R$ and ${Min}R$ be the set of maximal and minimal
(respectively) prime ideals of $R$ endowed with the induced Zariski topology.
For a subset $J$ of $R$, let $V(J)$ be the closed subset of ${Spec}R$
formed by the prime ideals that contain $J$. Let ${char}(R)\in 
\mathbb{N}\sqcup \{0\}$ be the characteristic of $R$. For a projective $R$-module $P$ of rank $1$, let $[P]\in {Pic}(R)$ be its class in the
Picard group. Let ${Pic}_{2}(R)$ be the subgroup of ${Pic}(R)
$ generated by classes $[P]$ with $P$ generated by $2$ elements. Let ${Pic}_{2}(R)[2]$ be the subgroup of ${Pic}_{2}(R)$ generated
by classes $[P]$ with $P\oplus P\cong R^2$ (so $2[P]=[R]$ and $P$ is generated by $2$ elements).

For $(\alpha _{1},\ldots ,\alpha _{n})\in R^{n}$, let 
${Diag}(\alpha
_{1},\ldots ,\alpha _{n})\in \mathbb{M}_{n}(R)$ be the diagonal matrix 
whose $(i,i)$ entry is $\alpha _{i}$ for all $i\in \{1,\ldots ,n\}$.

Subsections \ref{S21} to \ref{S25} review terminology and basic
properties, and Subsection \ref{S26} introduces `universal test
matrices' over $\mathbb{Z}[x,y,z]$, a regular \textsl{UFD} of
dimension $4$, that will be used to study ${SE}^{\triangle}_{2}$ and 
${SE}_2$ rings (see Corollaries \ref{C10} and \ref{C13}).

\subsection{On unimodular vectors}

\label{S21} In what follows we will use without extra comments the following
four basic properties.

\textbf{(1)} For $(a,b,c)\in R^3$ we have $(a,bc)\in Um(R^2)$ iff 
$(a,b),(a,c)\in Um(R^2)$.

\textbf{(2)} If $(a,b)\in Um(R^2)$ and $c\in R$, then $a$ divides $bc$ iff $a$ divides $c$.

\textbf{(3)} If $I\subseteq J(R)$ is an ideal of $R$ and $(a_1,\ldots,a_n)\in R^n$, then $(a_1,\ldots,a_n)\in Um(R^n)$ iff $(a_1+I,\ldots,a_n+I)\in Um\bigl((R/I)^n\bigr)$.

\textbf{(4)} If an element of $Um(R^m)$ has entries that are sums of 
homogeneous polynomials in $a_1,\ldots,a_n\in R$ with coefficients in $R$, 
then $(a_1,\ldots,a_n)\in Um(R^n)$.

\subsection{On pre-Schreier rings}\label{S22}

A ring $R$ is called \textsl{pre-Schreier}, if every nonzero element $a\in R$
is \textsl{primal}, i.e., if $a$ divides a product $bc$ of elements of $R$,
there exists $(d,e)\in R^2$ such that $a=de$, $d$ divides $b$ and $e$
divides $c$. Pre-Schreier domains were introduced by Zafrullah in \cite{zaf}.
A pre-Schreier integrally closed domain was called a \textsl{Schreier}
domain by Cohn in \cite{coh}. Every \textsl{GCD} domain (in particular,
every B\'{e}zout domain) is Schreier (see \cite{coh}, Thm.
2.4). A product of pre-Schreier domains is a pre-Schreier ring but we do not
know other examples of pre-Schreier rings.

In an integral domain, an irreducible element is primal iff it is
a prime. Thus an integral domain that has irreducible elements which are not
prime, such as each noetherian domain which is not a UFD, is not
pre-Schreier.

The \textsl{inner rank} of an $m\times n$ matrix over a ring is defined as
the least positive integer $r$ such that it can be expressed as the product
of an $m\times r$ matrix and an $r\times n$ matrix; over fields, 
this notion coincides with the usual notion of rank. A square matrix
is called \textsl{full} if its inner rank equals its order, and \textsl{non-full} otherwise. A $2\times 2$ matrix is non-full iff its
inner rank is $1$, (i.e., it has a column-row decomposition).

\subsection{On stable range less than or equal to 2}\label{S23}

We consider the sets 
\begin{equation*}
T_{3}(R):=Um (R^{2})\times (R\setminus \{0\})\;\text{\textup{and}}\;J_{3}(R):=Um (R^{2})\times (R\setminus J(R)).
\end{equation*}

\begin{proposition}
\label{PR1} We have $fsr(R)=1.5$ iff for all $(a,b,c)\in T_{3}(R)$ there exists $r\in R$ such that $(a+br,c)\in Um(R^{2})$.
\end{proposition}

\begin{proof}
The `only if' part is clear. To check the `if' part, let $(a,b,c)\in
Um(R^{3})$ with $c\neq 0$ and let $x,y,z\in R$ be such that $ax+by+cz=1$.
Thus $(a,by+cz,c)\in T_{3}(R)$ and hence there exists $r\in R$ such that $(a+byr+czr,c)\in Um(R^{2})$. This implies $(a+byr,c)\in Um(R^{2})$, thus $fsr(R)=1.5$.
\end{proof}

It follows from Proposition \ref{PR1} that if $sr(R)=1$, then $fsr(R)=1.5$.

\begin{proposition}
\label{PR2} We have $asr(R)=1$ iff for all $(a,b,c)\in J_{3}(R)$
there exists $r\in R$ such that $(a+br,c)\in Um(R^{2})$.
\end{proposition}

\begin{proof}
See \cite{mcg1}, Thm. 3.6 for the `only if' part. For the `if' part, for $I$ an ideal of $R$ not contained in $J(R)$ we check that $sr(R/I)=1$. If $(a,b)\in R^2$ is such that $(a+I,b+I)\in Um\bigl((R/I)^{2}\bigr)$, let
$(d,e)\in R^2$ and $c\in I$ be such that $ad+be+c=1$. If $c\notin J(R)$,
then for $(f,g):=(c,c)\in \bigl(I\setminus J(R)\bigr)\times I$ we have $(a,be+g,f)\in J_{3}(R)$. If $c\in J(R)$, then $ad+be=1-c\in U(R)$, so for $(f,g)\in\bigl(I\setminus J(R)\bigr)\times \{0\}$ 
we have $(a,be+g,f)\in J_{3}(R)$. If $r\in R$ is such that $(a+(be+g)r,f)\in Um(R^{2})$, then $a+I+(b+I)(er+I)\in U(R/I)$; so $sr(R/I)=1$.
\end{proof}

\begin{corollary}
\label{C4} \textbf{(1)} If $fsr(R)=1.5$, then $asr(R)=1$.

\smallskip \textbf{(2)} If $asr(R)=1$, then $sr(R)\le 2$.
\end{corollary}

\begin{proof}
As $J_{3}(R)\subseteq T_{3}(R)$, part (1) follows from Propositions \ref{PR1} and \ref{PR2}.

To check part (2), it suffices to show that each $(a,b,c)\in Um(R^3)$ is
reducible. If $b\notin J(R)$, then $R/Rb$ has stable range $1$ and $(a+Rb,c+Rb)\in Um\bigl((R/Rb)^2\bigr)$; hence there exists $r\in R$ such that $a+rc+Rb\in U(R/Rb)$ and thus for $(r_1,r_2):=(r,0)$ we have $(a+r_1c,b+r_2c)\in Um(R^2)$. If $b\in J(R)$, then $(a,b,c)\in Um(R^3)$
implies that $(a,c),(a,b+c)\in Um(R^2)$ and thus for $(r_1,r_2):=(0,1)$ we
have $(a+r_1c,b+r_2c)\in Um(R^2)$. We conclude that $(a,b,c)$ is reducible.
\end{proof}

Corollary \ref{C4}(2) was first obtained in \cite{mcg1}, Thm. 3.6.

We have the following \textsl{units interpretations}.

\begin{corollary}
\label{C5} \textbf{(1)} We have $sr(R)=1$ iff for each $b\in R$,
the homomorphism $U(R)\rightarrow U(R/Rb)$ is surjective.

\smallskip \textbf{(2)} We have $fsr(R)=1$ iff for
all $(b,c)\in R^2$ with $c\neq 0$, the homomorphism $U(R/Rc)\rightarrow
U(R/(Rb+Rc))$ is surjective.

\smallskip \textbf{(3)} We have $asr(R)=1$ iff for all $(b,c)\in R^2$
with $c\notin J(R)$, the homomorphism $U(R/Rc)\rightarrow U(R/(Rb+Rc))$ is
surjective.

\smallskip \textbf{(4)} If $asr(R)=1$, then $R$ is a $U_2 $ ring.
\end{corollary}

\begin{proof}
Parts (1) and (2) follow from definitions.

We first check the `if' part of (3). If $(a,b,c)\in J_{3}(R)$, then $a+Rb+Rc$
is a unit of $R/(Rb+Rc)$ and thus is the image of a unit of $R/Rc$, which is
of the form $a+rb+Rc$ with $r\in R$. Hence $(a+rb,c)\in Um(R^{2})$. From
this and Proposition \ref{PR2} it follows that $asr(R)=1$.

To check the `only if' part of (3), let $a+Rb+Rc\in U(R/(Rb+Rc))$. Thus $(a,b,c)\in Um(R^{3})$. As $Rc\not\subseteq J(R)$ and we are assuming that $asr(R)=1$, it follows that $sr(R/Rc)=1$, hence there exists $r\in R$ such that 
$a+rb+Rc\in U(R/Rc)$ maps to $a+Rb+Rc\in U(R/(Rb+Rc))$. Thus the homomorphism $U(R/Rc)\rightarrow U(R/(Rb+Rc))$ is surjective.

For part (4), note that for $c\in J(R)$ the Diagram (\ref{EQ1}) is a pushout. 
If $c\notin J(R)$ and $(a,b)\in Um(R^2)$, then 
$U(R/Rbc)\rightarrow U(R/(Rbc+Rc))=U(R/Rc)$ and 
$U(R/Rac)\rightarrow U(R/(Rac+Rc))=U(R/Rc)$ are surjective homomorphisms 
(see part (3)),  hence Diagram (\ref{EQ1}) is again a pushout. 
Thus $R$ is a $U_2$ ring.
\end{proof}

\begin{example}
\normalfont
\label{EX4} Let $(a,b,c)\in Um(R^3)$ with $(b,c)\in Um(R^2)$. Writing $a-1=eb+fc$ with $e,f\in R$, for $r:=-e$ we have $(a+rb,c)\in Um(R^2)$.
\end{example}

\begin{example}
\normalfont\label{EX5} Suppose $R$ is a semilocal ring (e.g., a valuation domain). 
As for each $b\in R$, the homomorphism $U(R)\rightarrow U(R/Rb)$ is
surjective, from Corollary \ref{C5}(1), it follows that $sr(R)=1$.
\end{example}

\begin{example}
\normalfont\label{EX6} Let $R$ be a noetherian domain of dimension 1. For $(b,c)\in R^2$ with $c\neq 0$, the rings $R/Rc$ and $R/(Rb+Rc)$ are artinian, hence the homomorphism $U(R/Rc)\rightarrow U(R/(Rb+Rc))$ is surjective. From
this and Corollary \ref{C5}(2), it follows that $fsr(R)=1.5$;
thus $sr(R)\le 2$ (see Corollary \ref{C4}(2)).
\end{example}

An argument similar to the one of Example \ref{EX6} shows that each B\'{e}zout domain which is a filtered union of Dedekind domains has stable range
1.5.

\begin{example}
\normalfont\label{EX7} We have $sr(\mathbb{Z})=2$ and $fsr(\mathbb Z)=1.5$ 
and the ring $\mathbb{Z}[x]/(x^{2})$ has
almost stable range 1 but does not have stable range 1.5.
\end{example}

\subsection{On arbitrary stable range}

\label{S24}

If $n\geq 2$ we do not have a unit interpretation of the stable range $n$
similar to Corollary \ref{C5} but this is replaced by standard projective modules
considerations recalled here in the form required in the sequel. Firstly, if $P_1$ and $P_2$ are two projective $R$-modules of rank $1$ such that $P_1\oplus P_2\cong R^2$, then by taking determinants it follows that $P_1\otimes_R P_2\cong R$, hence $[P_1]=-[P_2]\in Pic(R)$. Secondly, from the definition of stable ranges we have:

\begin{fact}
\label{F1} If $n\in\mathbb{N}$ and $sr(R)\le n$, then for each $a\in R$ the
reduction modulo $Ra$ map of sets $Um(R^n)\rightarrow Um\bigl((R/Ra)^n\bigr)$ is
surjective.
\end{fact}

For $n\in\mathbb{N}$, let $(a_1,\ldots,a_n,b)\in Um(R^{n+1})$ with 
$b\notin Z(R)$. We consider short exact sequences of $R$-modules $0\rightarrow R\xrightarrow{b} R\xrightarrow{\pi} R/Rb\rightarrow 0$ and $0\rightarrow
Q\rightarrow R^n\xrightarrow{f} R/Rb\rightarrow 0$ where $f$ maps the
elements of the standard basis of $R^n$ to $a_1+Rb,\ldots,a_n+Rb$ and $Q:=Ker(f)$. If $Q^+\rightarrow R^n$ and $Q^+\rightarrow R$ define the pullback of $f$ and $\pi$, then we have short exact sequences $0\rightarrow Q\rightarrow
Q^+\rightarrow R$ and $0\rightarrow R\rightarrow Q^+\rightarrow
R^n\rightarrow 0$ which imply that $Q^+$ is a free $R$-module of rank $n+1$
and $Q$ is a projective $R$-module of rank $n$ generated by $n+1$ elements;
moreover, if $n=1$, then $Q\cong R$.

We consider $R$-linear maps $g:R^{n}\rightarrow R$ such that $\pi \circ g=f$. Then $(a_{1},\ldots ,a_{n},b)$ is reducible iff we can choose $g $ to be surjective. Thus, if $sr(R)\le n$, then we can choose $g$ to be
surjective.

If $A\in \mathbb{M}_{n}(R)$ is equivalent to ${Diag}(1,1,\ldots
,1,b) $ and $Q_A$ is the $R$-submodule $Q$ of $R^{n}$, then we
can choose $g$ to be surjective and hence $(a_{1},\ldots ,a_{n},b)$ is
reducible.

In this paragraph we assume that $(a_{1},\ldots ,a_{n},b)$ is reducible and
that $g$ is chosen to be surjective. Then ${Ker}(g)$ is a projective $R$-module of rank $n-1$ generated by $n$ elements, we have a short exact
sequence $0\rightarrow {Ker}(g)\rightarrow Q\rightarrow Rb\rightarrow 0$,
and $0\rightarrow Q\rightarrow R^{n}\xrightarrow{f}R/Rb\rightarrow 0$
is the direct sum of the two projective resolutions $0\rightarrow R\xrightarrow{b}R\xrightarrow{\pi}R/Rb\rightarrow 0$ and $0\rightarrow {Ker}(g)\rightarrow {Ker}(g)\rightarrow 0\rightarrow 0$. If $n=2$, then ${Ker}(g)\cong R$.
Thus if $n=2$ and the $R$-submodule $Q$ of $R^{2}$ is $Q_{A}$ for some $A\in 
\mathbb{M}_{2}(R)$, then $A$ is equivalent to ${Diag}(1,b)$.
Similarly, if $n\geq 3$ and ${Ker}(g)\cong R^{n-1}$, then $Q$ is a free $R$-module of rank $n$, so if the $R$-submodule $Q$ of $R^{n}$ is $Q_{A}$
for some $A\in \mathbb{M}_{n}(R)$, then $A$ is equivalent to ${Diag}(1,1,\ldots ,1,b)$.

\subsection{Projective modules}

\label{S25}

Let $A\in Um(\mathbb{M}_{2}(R))$ and $\bar{R}:=R/R\det(A)$. The cokernel of $L_A$, denoted by $E_A$, is annihilated by $\det(A)$ and hence we can view it as an $\bar{R}$-module isomorphic to $C_{\bar A}$, where $\bar A$ is the reduction of $A$ modulo $R\det(A)$. 

Locally in the Zariski
topology of ${Spec}\bar{R}$, one of the entries of $\bar{A}$ is a unit and
hence the matrices $\bar{A}$ and ${Diag}(1,0)$ are equivalent. Thus $Q_{\bar{A}}$
is a projective $\bar{R}$-module of rank $1$ generated by two elements and we have two short exact sequences of projective $\bar{R}$-modules $0\rightarrow Q_{\bar{A}}\rightarrow \bar{R}^2\rightarrow C_{\bar{A}}\rightarrow 0$ and $0\rightarrow K_{\bar{A}}\rightarrow
\bar{R}^{2}\rightarrow Q_{\bar{A}}\rightarrow 0$ that split. From the existence of $\bar{R}$-linear isomorphisms $\bar{R}^{2}\cong K_{\bar{A}}\oplus Q_{\bar{A}}\cong K_{\bar{A}}\oplus C_{\bar{A}}$ it follows that $K_{\bar{A}}\cong C_{\bar{A}}$ are the dual of $Q_{\bar{A}}$. Thus $K_{\bar{A}}\cong\bar{R}$ iff $Q_{\bar{A}}\cong\bar{R}$ (equivalently, $C_{\bar{A}}\cong\bar{R}$), hence $[K_{\bar{A}}]=[C_{\bar{A}}]=-[Q_{\bar{A}}]\in {Pic}_2(R)$.

\subsection{Universal test matrices}

\label{S26}

In this subsection we assume that $R$ is a Hermite ring, i.e., for every
pair $(p,q)\in R^{2}$ there exists $(r,s,t)\in R^3$ such that $p=rs$, $q=rt$ and $(s,t)\in Um(R^{2})$. If moreover $R$ is an integral domain (i.e., if $R$ is
a B\'{e}zout domain), then $r$ is unique up to a multiplication with a unit
of $R$ and is called the greatest common divisor of $p$ and $q$ and one
writes $r=\gcd (x,y)$. Our convention is $\gcd (0,0)=0$ (so that we can
still write $0=0\cdot 1$ with $\gcd (1,1)=1$).

Let $A\in Um(\mathbb{M}_{2}(R))$ and $M\in {SL}_2(R)$ be such that 
$B:=MA=\left[ 
\begin{array}{cc}
g & u \\ 
0 & h\end{array}\right] $ is upper triangular. We write $g=ac$ and $h=bc$ with $(a,b)\in
Um(R^{2})$ and $c\in R$. As $B\in Um(\mathbb{M}_{2}(R))$ and $Rg+Rh=Rc$, we
have $(c,u)\in Um(R^{2})$. Let $a^{\prime },b^{\prime },c^{\prime
},d^{\prime }\in R$ be such that $aa^{\prime }+bb^{\prime }=cc^{\prime
}+uu^{\prime }=1$. For $d\in R$, $A$ and $B$ are equivalent to 
$$C:=\left[ 
\begin{array}{cc}
1 & db^{\prime } \\ 
0 & 1\end{array}\right] B\left[ 
\begin{array}{cc}
1 & da^{\prime } \\ 
0 & 1\end{array}\right]  =\left[ 
\begin{array}{cc}
ac & u+cd(aa^{\prime }+bb^{\prime }) \\ 
0 & bc\end{array}\right] =\left[ 
\begin{array}{cc}
ac & u+cd \\ 
0 & bc\end{array}\right].$$ 
Here the role of $u+cd$ is that of an arbitrary element of $R$
whose reduction modulo $Rc$ is the `fixed' unit $u+Rc\in U(R/Rc)$. Each
unimodular matrix of the form 
\begin{equation*}
D:=\left[ 
\begin{array}{cc}
aa^{\prime }cc^{\prime } & u \\ 
0 & bb^{\prime }cc^{\prime }\end{array}\right]
\end{equation*}will be called a \emph{companion test matrix} associated to $A$.

Note that $D$ is the image of the `first universal test matrix for Hermite
rings' 
\begin{equation*}
\mathcal{D}:=\left[ 
\begin{array}{cc}
x(1-yz) & y \\ 
0 & (1-x)(1-yz)\end{array}\right] \in Um(\mathbb{M}_{2}(\mathbb{Z}[x,y,z]))
\end{equation*}via the homomorphism $\phi _{D}:\mathbb{Z}[x,y,z]\rightarrow R$ which maps $x $, $y$, and $z$ to $aa^{\prime }$, $u$, and $u^{\prime }$ (hence $1-x$
maps to $bb^{\prime }=1-aa^{\prime }$ and $1-yz$ maps to $1-uu^{\prime
}=cc^{\prime }$). Note that $\mathcal{D}$ is equivalent to the matrix 
\begin{equation*}
\mathcal{E}:=\left[ 
\begin{array}{cc}
1 & 0 \\ 
z(x-1)(1-yz) & 1\end{array}\right] \mathcal{D}\left[ 
\begin{array}{cc}
1 & 0 \\ 
xz & 1\end{array}\right] =\left[ 
\begin{array}{cc}
x & y \\ 
0 & (1-x)(1-yz)^{2}\end{array}\right]
\end{equation*}which is the image of the `second universal test matrix for Hermite rings'
\begin{equation*}
\mathcal{F}:=\left[ 
\begin{array}{cc}
x & y \\ 
0 & (1-x)(1-yz)\end{array}\right] \in Um(\mathbb{M}_{2}(\mathbb{Z}[x,y,z]))
\end{equation*}via the endomorphism of $\mathbb{Z}[x,y,z]$ that fixes $x$ and $y$ and maps $z$ to $2z-yz^{2}$. See Corollary \ref{C13} for the usage of universal test
matrix in this paragraph.

The `universal test matrix for all rings' is (see Corollary \ref{C10}) 
\begin{equation*}
\mathcal{G}:=\left[ 
\begin{array}{cc}
x & y \\ 
0 & 1-x-yz\end{array}\right] \in Um(\mathbb{M}_{2}(\mathbb{Z}[x,y,z])).
\end{equation*}

\section{Criteria on extending $2\times 2$ matrices}

\label{S3}

For an element $v$ of an $R$-algebra, $(v)$ will be the principal ideal generated by it and $\bar{v}$ will be its reductions via some given surjective ring homomorphisms. 

Subsection \ref{S30} proves Lemma \ref{L1}. Subsections \ref{S31} and \ref{S32} introduce the $R$-algebras and the statements (respectively) that are required to study when a unimodular matrix in $\mathbb M_2(R)$ is simply extendable and prove Theorems \ref{TH7} and \ref{TH8} (respectively). Direct applications of Theorems \ref{TH7} and \ref{TH8} to ${SE}_2$, ${E}_2$, and ${SE}_2^{\triangle}$ rings are included in Subsections \ref{S33} and \ref{S34}. 

\subsection{Proof of Lemma \ref{L1}}

\label{S30}

To check the `only if' part of part (1), let $A\in\mathbb M_2(R)$ be extendable, with $A^+\in SL_3(R)$ an extension of it. If $A^+_0$ is obtained from $A^+$ by replacing the $(3,3)$ entry with $0$, then the reductions of $A^+$ and $A^+_0$ modulo $R\det(A)$ have the same determinant $1$, and it follows that $A$ modulo $R\det(A)$ is simply extendable. To check the `if' part of (1), let $B\in\mathbb M_3(R)$ be such that its reduction modulo $R\det(A)$ is a simple extension of the reduction of $A$ modulo $R\det(A)$. Hence there exists $w\in R$ such that $\det(B)=1+w\det(A)$. If $A^+\in\mathbb M_3(R)$ is obtained from $B$ by subtracting $w$ from its $(3,3)$ entry, then $A^+$ is an extension of $A$ as $\det(A^+)=\det(B)-w\det(A)=1$. Thus part (1) holds. Part (2) is a simple computation, while part (3) follows directly from part (2). Thus Lemma \ref{L1} holds.

\subsection{Five $R$-algebras}

\label{S31}

For $\upsilon =(a,b,c,d)\in Um(R^{4})$ and its associated matrix $A:=\left[ 
\begin{array}{cc}
a & b \\ 
c & d\end{array}\right] $ we will consider five $R$-algebras as follows.

The $R$-algebra
\begin{equation*}
U_{\upsilon }:=R[x,y,z,w]/(1-ax-by-cz-dw)
\end{equation*}
represents (parametrizes) unimodular relations $ax+by+cz+dw=1$ in $R$-algebras.
The $R$-algebra
\begin{equation*}
E_{\upsilon }:=R[x,y,z,w,v]/\bigl(1-axw-bxz-cyw-dyz-v(ad-bc)\bigr),
\end{equation*}
represents extensions $\left[ 
\begin{array}{ccc}
a & b & y\\ 
c & d & -x\\
-z & w & v
\end{array}\right]$ of (images of) $A$ in rings of $2\times 2$ matrices with entries in $R$-algebras. Similarly, the $R$-algebra
\begin{equation*}
SE_{\upsilon }:=R[x,y,z,w]/(1-axw-bxz-cyw-dyz)=E_{\upsilon }/(\bar{v}),
\end{equation*}
represents simple extensions of (images of) $A$ in rings of $2\times 2$ matrices with entries in $R$-algebras. The last two $R$-algebras 
\begin{equation*}
WZ_{\upsilon }:=R[x,y,z,w]/\bigl(1-ax-by-cz-dw+(ad-bc)(xw-yz)\bigr)
\end{equation*}and
\begin{equation*}
Z_{\upsilon }:=R[x,y,z,w]/(1-ax-by-cz-dw,xw-yz)=WZ_{\upsilon }/(\bar{x}\bar{w}-\bar{y}\bar{z})=U_{\upsilon }/(\bar{x}\bar{w}-\bar{y}\bar{z})
\end{equation*}
represent matrices $B$ as in Definition \ref{def2-}(2) and (1) (respectively) for (images of) $A$ in rings of $2\times 2$ matrices with entries in $R$-algebras.

Note that the polynomial $1-ax-by-cz-dw+(ad-bc)(xw-yz)\in R[x,y,z,w]$ is the
determinant of a $2\times 2$ matrix and hence it has many decompositions of
the form $f_{1,1}f_{2,2}-f_{1,2}f_{2,1}$, e.g.,
\begin{equation*}
WZ_{\upsilon }=R[x,y,z,w]/\bigl((1-ax-cz)(1-by-dw)-(ay+cw)(bx+dz)\bigr).
\end{equation*}
For example, if $c=0$, then 
\begin{equation}  \label{EQ6}
\begin{split}
WZ_{\upsilon}=R[x,y,z,w]/\bigl((1-ax)(1-dw)-y(b+adz)\bigr) \\
=R[x,y,z,w]/\bigl((1-ax)(1-by-dw)-ay(bx+dz)\bigr),
\end{split}\end{equation}
and these two identities led us to ${V}_2$ rings and to Criterion \ref{CR3}
(respectively).

If $b=c$, then all these five $R$-algebras have an involution defined by
fixing $x$ and $w$ and by interchanging $y$ and $z$, and we have 
\begin{equation*}
WZ_{\upsilon}/(\bar y-\bar z)= R[x,y,w]/\bigl((1-ax-by)(1-by-dw)-(bx+dy)(ay+bw)\bigr),
\end{equation*}
\begin{equation*}
SE_{\upsilon}/(\bar y-\bar z)= R[x,y,w]/\bigl[
(1-axw)-y\bigl(b(x+w)+dy\bigr)\bigr].
\end{equation*}

We have an arrow diagram of $R$-algebra homomorphisms 

\begin{equation}\label{EQ7}
\xymatrix@R=10pt@C=21pt@L=2pt{
WZ_{\upsilon} \ar[r]^{\epsilon} & Z_{\upsilon}\ar[d]^{\rho} & U_{\upsilon}\ar[l]_{\varepsilon}\\
E_{\upsilon} \ar[r]^{\psi} & SE_{\upsilon}\\}
\end{equation}
defined as follows. The horizontal homomorphisms are epimorphisms defined by identifications 
$WZ_{\upsilon }/(\bar{x}\bar{w}-\bar{y}\bar{z})=Z_{\upsilon }=U_{\upsilon }/(\bar{x}\bar{w}-\bar{y}\bar{z})$
and $SE_{\upsilon }=E_{\upsilon }/(\bar{v})$, and $\rho $ is defined by mapping
$\bar{x}$, $\bar{y}$, $\bar{z}$, $\bar{w}$ to $\bar{x}\bar{w}$, $\bar{x}\bar{z}$, $\bar{y}\bar{w}$, $\bar{y}\bar{z}$ (respectively). A key part of the essence of Theorem \ref{TH3} is the existence of Diagram (\ref{EQ7}).

The $R$-module 
$$P_{\upsilon}:=\{(x,y,z,w)\in R^4|ax+by+cz+dw=0\}$$ 
is the 
kernel of the $R$-linear map $R^4\to R$ that maps  $[x,y,z,w]^T$ to 
$ax+by+cz+dw$. As $\upsilon\in Um(R^4)$, there exists $(a^{\prime},b^{\prime},c^{\prime},d^{\prime})\in R^4$ such that $aa^{\prime}+bb^{\prime}+cc^{\prime}+dd^{\prime}=1$ and hence the mentioned $R$-linear 
map is surjective. Thus $P_{\upsilon}$ and its dual $P^*_{\upsilon}$ are projective $R$-modules of rank $3$ with the property that $P_{\upsilon}\oplus R\cong P_{\upsilon}^*\oplus R\cong R^4$. From \cite{lam}, Ch. III, Sect. 6, Thm. 6.7 (1) it follows that $P_{\upsilon}\cong P_{\upsilon}^*$.

\begin{theorem}
\label{TH7} The following properties hold:

\medskip \textbf{(1)} For $f\in \{a,b,c,d\}$, 
the $R_{f}$-algebra $(U_{\upsilon })_{f}$ is a polynomial $R_{f}$-algebra in $3$ indeterminates. Also, the $R$-algebra $U_{\upsilon }$ is isomorphic to the 
symmetric $R$-algebra of $P_{\upsilon}$ (thus the homomorphism $R\rightarrow U_{\upsilon }$ is smooth of relative 
dimension $3$). Moreover, if $R$ is a Hermite ring, then $P_{\upsilon}\cong R^3$ and 
hence the $R$-algebra $U_{\upsilon }$ is a polynomial $R$-algebra in $3$ indeterminates.

\smallskip \textbf{(2)} The $R$-algebra $Z_{\upsilon }$ is smooth of
relative dimension $2$.

\smallskip \textbf{(3)} The $R$-algebra $E_{\upsilon}$ is smooth of relative
dimension $4$ and the $R$-algebra $(WZ_{\upsilon})_{1-(ad-bc)(xw-yz)}$ is smooth
of relative dimension $3$ (thus, if $ad-bc\in J(R)$, then the $R$-algebra $WZ_{\upsilon}$ is smooth of relative dimension $3$).

\smallskip \textbf{(4)} The morphism of schemes ${Spec}
(SE_{\upsilon})\rightarrow {Spec} (Z_{\upsilon})$ defined by $\rho$
is a $\mathbb{G}_{m,Z_{\upsilon}}$-torsor and hence smooth of relative
dimension $1$ (thus the $R$-algebra ${SE}_{\upsilon}$ is smooth of relative
dimension $3$).

\smallskip \textbf{(5)} Assume $ad=bc$. Then there exists a projective $R$-module $Q_{\upsilon}$ of rank $2$ such that the $R$-algebra $Z_{\upsilon }$ is isomorphic to the symmetric $R$-algebra of $Q_{\upsilon}$. Also $WZ_{\upsilon}=U_{\upsilon}$ is isomorphic to the 
symmetric $R$-algebra of $P_{\upsilon}$ and there exists an isomorphism $Q_{\upsilon}\oplus R\cong P_{\upsilon}$. Moreover, if $R$ is a Hermite ring, then $Q_{\upsilon}\cong R^2$, the $R$-algebra $Z_{\upsilon }$ is a polynomial $R$-algebra in $2$ indeterminates, and the $\mathbb{G}_{m,Z_{\upsilon}}$-torsor of part (4) is the pullback of a $\mathbb{G}_{m,R}$-torsor.

\smallskip \textbf{(6)} Assume that $c=b$ and $ad=b^2$ (i.e., $A$ is
symmetric with zero determinant). Then $(a,d)\in Um(R^2)$ and there exist
two canonical $R$-algebra identifications $(SE_{\upsilon})_a\cong
R_a[x^{\prime},y,z,w^{\prime}]/(1-x^{\prime}w^{\prime})$ and $(SE_{\upsilon})_d\cong
R_d[x,y^{\prime},z^{\prime},w]/(1-y^{\prime}z^{\prime})$, where $x^{\prime}:=ax+by$, $w^{\prime}:=w+\dfrac{b}{a}z$, $y^{\prime}:=y+\dfrac{b}{d}x$, and $z^{\prime}:=dz+bw$.
\end{theorem}

\begin{proof}
(1) As $\upsilon \in Um(R^{4})$, we have ${Spec} R=\cup _{f\in
\{a,b,c,d\}}{Spec}R_{f}$. To check that the $R_{f}$-algebra $(U_{\upsilon })_{f}$ is a polynomial $R_{f}$-algebra we can
assume that $f=a$ and by replacing $(R,x)$ with $(R_{f},a^{-1}x)$ we can
assume that $a=1$, in which case we have $U_{\upsilon }\cong R[y,z,w]$.
From \cite{BCW}, Thm. 4.4 it follows that $U_{\upsilon }$ is isomorphic to the 
symmetric $R$-algebra of a projective $R$-module $P$ of rank $3$. The fact that 
$P\cong P_{\upsilon}$ follows from the facts that $P_{\upsilon}\cong P^*_{\upsilon}$ 
and that under the substitution $(x^{\prime},y^{\prime},z^{\prime},w^{\prime})=(a^{\prime},b^{\prime},c^{\prime},d^{\prime})+(x,y,z,w)$, $U_{\upsilon}=R[x^{\prime},y^{\prime},z^{\prime},w^{\prime}]/(ax^{\prime}+by^{\prime}+cz^{\prime}+dw^{\prime})$.

If $R$ is a Hermite ring, then $P_{\upsilon}$, being stably free, is free (\cite{WW}, Cor. 3.2); hence $P_{\upsilon}\cong R^3$. Thus the $R$-algebra $U_{\upsilon }$ is a polynomial $R$-algebra in $3$ indeterminates.

(2) It suffices to show that if $\mathfrak{n}\in Max(R[x,y,z,w])$ 
contains $xw-yz$ and $1-ax-by-cz-dw$, then, denoting $\kappa
:=R[x,y,z,w]/\mathfrak{n}$, the $\kappa $-vector space 
\begin{equation*}
\kappa \delta x\oplus \kappa \delta y\oplus \kappa \delta x\oplus \kappa
\delta w/(\kappa \delta (ax+by+cz+dw)+\kappa \delta (xw-yz))
\end{equation*}has dimension $2$ (see \cite{SGA1}, Exp. II, Thm. 4.10); here $\delta $ is
the differential operator denoted in an unusual way in order to avoid
confusion with the element $d\in R$. To check this, it suffices to show that
the assumption that the reduction of the matrix 
\begin{equation*}
E:=\left[ 
\begin{array}{cccc}
-a & -b & -c & -d \\ 
w & -z & -y & x\end{array}\right]
\end{equation*}modulo $\mathfrak{n}$ does not have rank $2$, leads to a contradiction; as $E $ modulo $\mathfrak{n}$ has unimodular rows, there exists $\alpha \in U(\kappa )$ such that $(-a,-b,-c,-d)-\alpha
(w,-z,-y,x)\in \mathfrak{n}^{4}\subseteq R^{4}$. From this, as $xw-yz,1-ax-by-cz-dw\in\mathfrak{n}$, it follows that $2\alpha (xw-yz)$
is congruent to $0$ modulo $\mathfrak{n}$, which implies that $0$ and $1$
are congruent modulo $\mathfrak{n}$, a contradiction. Thus part (2) holds.

(3) For $E_{\upsilon}$, based on \cite{SGA1}, Exp. II, Thm. 4.10, it suffices to 
show that for 
$$F(x,y,z,w,v):=1-axw-bxz-cyw-dyz-v(ad-bc)\in R[x,y,z,w,v],$$ 
$F$ and its partial derivatives $F_x$, $F_y$, $F_z$, $F_w$ and $F_v$ generate 
$R[x,y,z,w,v]$, but this follows from the identity 
$$1=F-vF_v+xF_x+yF_y.$$ 

For $WZ_{\upsilon}$, based on loc. cit. it suffices to show that for 
$$G(x,y,z,w):=1-ax-by-cz-dw+(ad-bc)(xw-yz)\in R[x,y,z,w],$$ 
$G$ and its partial derivatives $G_x$, $G_y$, $G_z$ and $G_w$ generate an ideal of $R[x,y,z,w]$ which contains $1-(ad-bc)(xw-yz)$, but this follows from the identity
$$1-(ad-bc)(xw-yz)=G-xG_x-yG_y-zG_z-wG_w.$$

(4) The $\mathbb{G}_{m,Z_{\upsilon }}$-action 
\begin{equation*}
{Spec} (SE_{\upsilon }[u,u^{-1}])={Spec} (Z_{\upsilon
}[u,u^{-1}])\times _{{Spec}(Z_{\upsilon })}{Spec} (SE_{\upsilon
})\rightarrow {Spec} (SE_{\upsilon })
\end{equation*}is given by the $R$-algebra homomorphism $SE_{\upsilon }\rightarrow
SE_{\upsilon }[u,u^{-1}]$ that maps $\bar{x}$, $\bar{y}$, $\bar{z}$, and $\bar{w}$ to $u\bar{x}$, $u\bar{y}$, $u^{-1}\bar{z}$, and $u^{-1}\bar{w}$
(respectively); the fact that it makes the morphism ${Spec}
(SE_{\upsilon })\rightarrow {Spec} (Z_{\upsilon })$ a $\mathbb{G}_{m,Z_{\upsilon }}$-torsor is a standard exercise.

(5) Part (5) for $WZ_{\upsilon}$ follows from part (1) as $ad=bc$ implies
that $WZ_{\upsilon}=U_{\upsilon}$. Similar to part (1), to check part (5) for $Z_{\upsilon}$ it suffices to show that for each $f\in \{a,b,c,d\}$ the $R_{f}$-algebra $(Z_{\upsilon })_{f}$ is a polynomial $R_{f}$-algebra in $2$ indeterminates. To
check this we can assume that $f=a$, and so the $R_a$-algebra $(Z_{\upsilon })_{a}$ is isomorphic to 
\begin{equation*}
R_{a}[y,z,w]/\bigl(yz-a^{-1}w(1-by-cz-dw)\bigr),
\end{equation*}which via the change of indeterminates $(y_1,z_1,w_1):=(ay+cw,az+bw,aw)$ is isomorphic, as $ad=bc$, to $R_{f}[y_1,z_1,w_1]/(y_1z_1-w_1)\cong R_{f}[y_1,z_1]$.

As $Z_{\upsilon}$ is a symmetric $R$-algebra, there exists an $R$-algebra epimorphism (retraction or augmentation in the terminology of \cite{BCW}) $Z_{\upsilon}\rightarrow R$ and hence there exists a quadruple $\zeta:=(a^{\prime},b^{\prime},c^{\prime},d^{\prime})\in R^4$ such that $1-aa^{\prime}-bb^{\prime}-cc^{\prime}-dd^{\prime}=0=a^{\prime}d^{\prime}-b^{\prime}c^{\prime}$. With the substitution $(x^{\prime},y^{\prime},z^{\prime},w^{\prime})=(a^{\prime},b^{\prime},c^{\prime},d^{\prime})+(x,y,z,w)$, we identify
$$Z_{\upsilon}=R[x^{\prime},y^{\prime},z^{\prime},w^{\prime}]/(ax^{\prime}+by^{\prime}+cz^{\prime}+dw^{\prime},a^{\prime}w^{\prime}-c^{\prime}y^{\prime}-b^{\prime}z^{\prime}+d^{\prime}x^{\prime}-x^{\prime}w^{\prime}+y^{\prime}z^{\prime}).$$
Endowing $R^4$ with the standard inner product 
$$\langle(x_1,x_2,x_3,x_4)\cdot (y_1,y_2,y_3,y_4)\rangle:=\sum_{i=1}^4 x_iy_i,$$ we have $\zeta\perp (d^{\prime},-c^{\prime},-b^{\prime},a^{\prime})$ and $\langle \zeta,(a,b,c,d)\rangle=1$. Therefore the sum $W:=R(a,b,c,d)+ R(d^{\prime},-c^{\prime},-b^{\prime},a^{\prime})$ is a direct sum (so $W\cong R^2$) and a direct summand of $R^4$ (i.e., and $R^4/W$ is a projective $R$-module of rank $2$).
Let 
$$Z^{\prime}_{\upsilon}:=R[x^{\prime},y^{\prime},z^{\prime},w^{\prime}]/(ax^{\prime}+by^{\prime}+cz^{\prime}+dw^{\prime},a^{\prime}w^{\prime}-b^{\prime}z^{\prime}-c^{\prime}y^{\prime}+d^{\prime}x^{\prime}).$$
Note that ${Spec} R=\cup _{f\in
\{a,b,c,d\}}{Spec}R_{ff'}$ and for $f\in\{a,d\}$ (resp.\ $f\in\{b,c\}$), the $R_{ff'}$-algebra $(Z^{\prime}_{\upsilon})_{ff'}$ is isomorphic to $R_{ff'}[y',z']$ (resp.\ $R_{ff'}[x',w']$). Let $J_{\upsilon}$ and $J^{\prime}_{\upsilon}$ be the ideals of $Z_{\upsilon}$ and $Z^{\prime}_{\upsilon}$ (respectively) generated by the images of $x^{\prime}$, $y^{\prime}$, $z^{\prime}$, $w^{\prime}$. The $R$-modules $Q_{\upsilon}:=J_{\upsilon}/J^2_{\upsilon}$ and $Q_{\upsilon}:=J^{\prime}_{\upsilon}/(J^{\prime}_{\upsilon})^2$ are identified via the third isomorphism theorem with the $R$-module 
$$(x',y',z',w')/[(ax^{\prime}+by^{\prime}+cz^{\prime}+dw^{\prime},a^{\prime}w^{\prime}-b^{\prime}z^{\prime}-c^{\prime}y^{\prime}+d^{\prime}x^{\prime})+(x',y',z',w')]$$ 
which is a quotient of ideals of $R[x',y',z',w']$ isomorphic to $R^4/W$. From \cite{BCW}, Cor. 4.3 and the above part on isomorphisms of the localizations $(Z_{\upsilon })_{f}$ and $(Z^{\prime}_{\upsilon})_{ff'}$ that involve indeterminates that are linear (not necessary homogeneous) polynomials in $x',y',z,w'$, it follows that $Z_{\upsilon}$ and $Z^{\prime}_{\upsilon}$ are isomorphic to the 
symmetric $R$-algebras of $Q_{\upsilon}$ and $Q_{\upsilon}^{\prime}$ (respectively), and in particular they are isomorphic. As $W$ is a direct summand of $R^4$, there exists a split short exact sequence $0\rightarrow R\rightarrow P_{\upsilon}\rightarrow Q_{\upsilon}\rightarrow 0$ of $R$-modules. The second part of part (5) follows from the last two sentences. 

If $R$ is Hermite, then $P_{\upsilon}\cong R^3$ (see part (1)), therefore $Q_{\upsilon}$ is stably free, hence free by \cite{WW}, Cor. 3.2. Based on the equivalence between $\mathbb{G}_m$-torsors and line bundles, to complete the proof of part (5) it suffices to show that for each polynomial $R$-algebra $R_1$, the functorial homomorphism ${Pic}(R)\rightarrow {Pic}(R_1)$ is an isomorphism. To check this we can assume that $R$ is reduced. As $R$ is a reduced Hermite ring, for each $\mathfrak p\in {Spec}R$, the local ring $R_{\mathfrak p}$ is a valuation domain. Thus $R$ is a normal ring (i.e., all its localizations $R_{\mathfrak p}$ are integral domains that are integrally closed in their fields of fractions) and hence also a seminormal ring in the sense of \cite{GH+}. From this and \cite{GH+}, Thm. 1.5 it follows that ${Pic}(R)\rightarrow {Pic}(R_1)$ is an isomorphism.

(6) As $(a,b,c,d)\in Um(R^4)$ and $b=c$, we have $(a,b,d)\in Um(R^3)$ and
thus also $(a,b^2,d)=(a,ad,d)\in Um(R^2)$. This implies that $(a,d)\in
Um(R^2)$. The existence of the mentioned identifications is straightforward.
\end{proof}

\subsection{Ten statements}

\label{S32}

For $\upsilon =(a,b,c,d)\in Um(R^{4})$ and its associated matrix $A=\left[ 
\begin{array}{cc}
a & b \\ 
c & d\end{array}\right] $ we also consider the following ten statements:

\medskip \circled{\textup{1}} The matrix $A$ is equivalent to the diagonal
matrix ${Diag}(1,\det(A))$.

\smallskip \circled{\textup{2}} The matrix $A$ is simply extendable.

\smallskip \circled{\textup{3}} There exists $(e,f)\in R^2$ such that $(ae+cf,be+df)\in Um(R^{2})$ (note that $(e,f)\in Um(R^2)$).

\smallskip \circled{\textup{4}} There exists $(x,y,z,w)\in R^4$ such that $ax+by+cz+dw=1$ and the matrix $\left[ 
\begin{array}{cc}
x & y \\ 
z & w\end{array}\right] $ is non-full.

\smallskip \circled{\textup{5}} There exists $(x,y,z,w)\in R^4$ such that $ax+by+cz+dw=1$ and $xw-yz=0$.

\smallskip \circled{\textup{6}} There exists $B\in Um(\mathbb{M}_{2}(R)) $ such that $(ad-bc)B+A\in Um(\mathbb{M}_{2}(R))$ and both $B$ and $(ad-bc)B+A$ have zero determinant.

\smallskip \circled{\textup{7}} There exists $C\in Um(\mathbb{M}_2(R))$ congruent to $A$ modulo $R(ad-bc)$ and $\det(C)=0$.

\smallskip \circled{\textup{8}} There exists $B\in \mathbb{M}_2(R)$ such that both $B$ and $(ad-bc)B+A$ have zero determinant.

\smallskip \circled{\textup{9}} There exists $(x,y,z,w)\in R^4$ such that $ax+by+cz+dw-(ad-bc)(xw-yz)=1$.

\smallskip \circled{\textup{10}} There exists $C\in \mathbb{M}_2(R)$
congruent to $A$ modulo $R(ad-bc)$ and $\det(C)=0$.

\begin{theorem}
\label{TH8} The following implications hold: 
\begin{equation*}
\circled{\textup{1}}\Leftrightarrow \circled{\textup{2}}\Leftrightarrow \circled{\textup{3}}\Leftrightarrow \circled{\textup{4}}\Rightarrow \circled{\textup{5}}\Leftrightarrow\circled{\textup{6}}\Leftrightarrow\circled{\textup{7}}\Leftrightarrow\circled{\textup{8}}\Rightarrow 
\circled{\textup{9}}\Rightarrow\circled{\textup{10}}.
\end{equation*}
Moreover, if $R$ is reduced, then $\circled{\textup{9}}\Leftrightarrow\circled{\textup{10}}$.
\end{theorem}

\begin{proof}
See Example \ref{EX1} for the implication $\circled{\textup{1}}\Rightarrow \circled{\textup{2}}$. For $(e,f,s,t)\in R^4$ we have two identities 
\begin{equation}  \label{EQ8}
\det \left[ 
\begin{array}{ccc}
a & b & f \\ 
c & d & -e \\ 
-t & s & 0\end{array}\right] =(be+df)t+(ae+cf)s=a(es)+b(et)+c(fs)+d(ft).
\end{equation}The equivalence $\circled{\textup{2}}\Leftrightarrow \circled{\textup{3}}$
follows from the first identity. The equivalence $\circled{\textup{3}}\Leftrightarrow \circled{\textup{4}}$ follows from the second identity.
Clearly, $\circled{\textup{4}}\Rightarrow \circled{\textup{5}}$, $\circled{\textup{6}}\Rightarrow \circled{\textup{7}}\wedge\circled{\textup{8}}$, and $\circled{\textup{7}}\vee\circled{\textup{8}}\Rightarrow\circled{\textup{10}}$. 

To check $\circled{\textup{3}}\Rightarrow \circled{\textup{1}}$, let $(e,f)\in
R^2$ be such that there exists $(s,t)\in Um(R^{2})$ with $(ae+cf)s+(be+df)t=1$. Then $(e,f),(s,t)\in Um(R^{2})$ and thus there exists $M\in {SL}_{2}(R)$ whose first row is $[e\;f]$ and there exists $N\in {SL}_{2}(R)$ whose first column is $[s\;t]^{T}$. The matrix $MAN$ has
determinant $\det (A)$ and its $(1,1)$ entry is $1$, thus it is equivalent to
the matrix ${Diag}(1,\det (A))$ and statement $\circled{\textup{1}}$
holds.

The implication $\circled{\textup{5}}\Rightarrow \circled{\textup{6}}$
follows from the fact that if $x,y,z,w\in R$ are such that $ax+by+cz+dw=1$
and $xw-yz=0$, then for $\delta :=ad-bc$ and $B:=\left[ 
\begin{array}{cc}
-w & z \\ 
y & -x\end{array}\right] \in Um(\mathbb{M}_{2}(R))$, the determinant of the matrix 
$$C:=A+\delta
B=\left[ 
\begin{array}{cc}
a-(ad-bc)w & b+(ad-bc)z \\ 
c+(ad-bc)y & d-(ad-bc)x\end{array}\right] $$ is equal to 
\begin{equation}
\delta (1-ax-by-cz-dw)+\delta ^{2}(xw-zy)=0  \label{EQ9}
\end{equation}and the relation $C\in Um(\mathbb{M}_{2}(R))$ follows from the identity 
\begin{equation*}
x(a-\delta w)+y(b+\delta z)+z(c+\delta y)+w(d-\delta x)=ax+by+cz+dw-2\delta
(xw-yz)=1.
\end{equation*}

If $\delta\notin Z(R)$, then the arguments of the
previous paragraph can be reversed and hence it follows that $\circled{\textup{8}}\Rightarrow \circled{\textup{5}}$.

To check that $\circled{\textup{8}}\Rightarrow \circled{\textup{5}}$ in
general, by replacing $R$ with a finitely generated $\mathbb{Z}$-subalgebra $S$
of $R$ such that we have $A,B\in Um(\mathbb{M}_{2}(S))$, we can
assume that $R$ is noetherian. Thus the set ${Min}R=\{\mathfrak{p}
_{1},\ldots ,\mathfrak{p}_{j}\}$ has a finite number of elements $j\in\mathbb N$. Statement $\circled{\textup{5}}$ holds if and only of there exists an $R$-algebra epimorphism (retraction) $Z_{\upsilon}\rightarrow R$. As the homomorphism $R\rightarrow Z_{\upsilon }$ is smooth (see Theorem \ref{TH7}(2)), each $R$-algebra homomorphism $Z_{\upsilon }\rightarrow R/N(R)$ lifts to an $R$-algebra epimorphism $Z_{\upsilon }\rightarrow R$. Hence, by replacing $R$ with $R/N(R)$, we can assume that $N(R)=\cap
_{i=1}^{j}\mathfrak{p}_{i}=0$. Let $\delta _{i}$ be the image of $\delta $
in $R/\mathfrak{p}_{i}$. Based on Theorem \ref{TH7}(5) we can assume that $\delta \neq 0$, and hence there exists an $i\in \{1,\ldots ,j\}$ such that $\delta _{i}\neq 0$. We can assume that the indexing of the minimal prime
ideals is such that there exists $j^{\prime }\in \{1,\ldots ,j\}$ such that $\delta _{1},\ldots ,\delta _{j^{\prime }}$ are all nonzero and $\delta
_{j^{\prime }+1},\ldots \delta _{j}$ are all zero. If $\mathcal I_{1}:=\cap
_{i=1}^{j^{\prime }}\mathfrak{p}_{i}$ and 
$\mathcal I_{2}:=\cap _{i=j^{\prime }+1}^{j}\mathfrak{p}_{i}$, we have 
$\mathcal I_{1}\cap \mathcal I_{2}=0$ and $\delta\in \mathcal I_2$. As the 
image of $\delta $ in $R/\mathcal I_{1}$ is not a zero divisor, from the previous paragraph 
it follows that there 
exists an $R$-algebra homomorphism $h_{1}:Z_{\upsilon }\rightarrow R/\mathcal I_{1}$; 
let $h_{1,2}:Z_{\upsilon
}\rightarrow R/(\mathcal I_{1}+\mathcal I_{2})$ be induced by $h_{1}$. As $A$ modulo $\mathcal I_2$ has zero determinant, $Z_{\upsilon
}/\mathcal I_{2}Z_{\upsilon }$ is the symmetric algebra of a projective $R/\mathcal I_{2}$-module $Q_{2}$ of rank $2$ (see Theorem \ref{TH7}(5)), hence the $R$-algebra homomorphism $h_{1,2}$ is
uniquely determined by an $R/\mathcal I_{2}$-linear map $l_{1,2}:Q_{2}\rightarrow
R/(\mathcal I_{1}+\mathcal I_{2})$. If $l_{2}:Q_{2}\rightarrow R/\mathcal I_{2}$ is an $R/\mathcal I_{2}$-linear
map that lifts $l_{1,2}$ and if $h_{2}:Z_{\upsilon }\rightarrow R/\mathcal I_{2}$ is
the $R$-algebra homomorphism uniquely determined by $l_2$, then $h_{2}$
lifts $h_{1,2}$. As $\mathcal I_{1}\cap \mathcal I_{2}=0$, the natural diagram
\begin{equation*}
\xymatrix@R=10pt@C=16pt{
R \ar[r] \ar[d] & R/\mathcal I_1 \ar[d]\\
R/\mathcal I_2 \ar[r] & R/(\mathcal I_1+\mathcal I_2)\\}
\end{equation*} is a pullback, hence there
exists a unique $R$-algebra homomorphism $Z_{\upsilon }\rightarrow R$ that
lifts both $h_{1}$ and $h_{2}$. Thus statement $\circled{\textup{5}}$ holds.

We conclude that statements $\circled{\textup{5}}$, $\circled{\textup{6}}$ and $\circled{\textup{8}}$ are equivalent. Clearly, $\circled{\textup{5}}\Rightarrow\circled{\textup{9}}$, hence $\circled{\textup{8}}\Rightarrow\circled{\textup{9}}$.

The implication $\circled{\textup{9}}\Rightarrow\circled{\textup{10}}$
follows from the fact that if $(x,y,z,w)\in R^4$ is such that $1-ax-by-cz-dw+(ad-bc)(xw-yz)=0$, then Equation (\ref{EQ9}) holds, thus, with 
$C:=A+\delta \left[ 
\begin{array}{cc}
-w & z \\ 
y & -x\end{array}\right]$ as above, we have $\det(C)=0$. Note that the entries of $C$ satisfy the 
linear identity
\begin{equation}\label{EQ10}
\begin{split}
x(a-\delta w)+y(b+\delta z)+z(c+\delta
y)+w(d-\delta x) \\
=ax+by+cz+dw-2\delta(xw-yz)=1-\delta(xw-yz).
\end{split}
\end{equation}

If $R$ is reduced (i.e., $N(R)=0$), the proof that $\circled{\textup{10}}\Rightarrow\circled{\textup{9}}$ is similar to the proof of the implication $\circled{\textup{8}}\Rightarrow\circled{\textup{5}}$ in the case when $N(R)=0$.

We are left to prove that $\circled{\textup{7}}\Rightarrow\circled{\textup{6}}$. As $\circled{\textup{5}}\Leftrightarrow\circled{\textup{6}}$, as above we argue that it suffices to prove that $\circled{\textup{7}}\Rightarrow\circled{\textup{6}}$ holds when $R$ is noetherian and reduced.

Recall that $\delta=\det(A)$ and we write $C=A+\delta B=AD$, where 
$D:=I_2+A^*B$ with $A^*\in\mathbb M_2(R)$ being the adjugate of $A$. 

As $C=AD\in Um(\mathbb M_2(R))$ we have $D\in Um(\mathbb M_2(R))$. We check that the converse holds as well, i.e., if $E\in Um(\mathbb M_2(R))$ 
and $AE$ is congruent to $A$ modulo $R\delta$, then $AE\in Um(\mathbb M_2(R))$. 
Let $J$ be the ideal of $R$ generated by the entries of $AE$. To prove that $J=R$ 
it suffices to show that if $\mathfrak m\in {Max}R$, then $J\not\subseteq\mathfrak m$. 
Clearly this holds if $\delta\in\mathfrak m$. But if $\delta\notin\mathfrak m$, then $A$ modulo $\mathfrak m$ is invertible, hence $AE$ modulo $\mathfrak m$ is nonzero 
as this is so for $E$ modulo $\mathfrak m$ and it follows that $J\not\subseteq\mathfrak m$.

As $\circled{\textup{7}}\Rightarrow\circled{\textup{10}}$, and as the proof of $\circled{\textup{10}}\Rightarrow\circled{\textup{9}}$ is similar to the proof of the implication $\circled{\textup{8}}\Rightarrow\circled{\textup{5}}$, it follows that there exists $B^{\prime}\in\mathbb M_2(R)$ such that for $D^{\prime}:=I_2+A^*B^{\prime}$ we have $\det(D')=0$ and $B^{\prime}$ and $B$ are congruent modulo each $\mathfrak m\in {Max}R$ which does not contain $\delta$ (see the above part involving $\delta\in I_2$). Thus, if $J^{\prime}$ is the ideal of $R$ generated by the entries of $D^{\prime}$, then $J^{\prime}$ is not contained in a $\mathfrak m\in {Max}R$ which does not contain $\delta$, and from Equation (\ref{EQ10}) it follows that $J^{\prime}$ is not contained in any maximal ideal which does not contain $1-\delta\det(B^{\prime})$. Hence $J^{\prime}$ is not contained in any $\mathfrak m\in {Max}R$, therefore $J^{\prime}=R$, i.e., $D^{\prime}$ is unimodular.

From the last two paragraphs it follows that, by replacing $(B,D)$ with $(B^{\prime},D^{\prime})$, we can assume that $\det(D)=0$. As $D=I_2+A^*B$ has zero determinant, it follows that $B\in Um(\mathbb M_2(R))$. 

Thus to complete the proof that $\circled{\textup{7}}\Rightarrow\circled{\textup{6}}$, it suffices to show that we can replace $B$ by a matrix $B_1\in Um(\mathbb M_2(R))$ with $\det(B_1)=0$ and such that for $D_1:=I_2+A^*B_1$ we have $\det(D_1)=0$ and $D_1\in Um(\mathbb M_2(R))$ (so $C_1:=AD_1$ is congruent to $A$ modulo $R\delta$, is unimodular by the above converse, and $\det(C_1)=0$). Recall that $K_D$ and $Q_D$ are projective $R$-modules of rank $1$ and the short exact sequence $0\rightarrow K_D\rightarrow R^2\rightarrow Q_D\rightarrow 0$ splits, i.e., it has a section $s_D:Q_D\rightarrow R^2$ (see Subsection \ref{S25}). Let $B_1\in\mathbb M_2(R)$ be the unique matrix such that $K_{D}\subseteq K_{B_1-B}$ and $s_D(Q_D)\subseteq K_{B_1}$. As $K_D$ is a direct summand of $R^2$ of rank $1$ and as for all $t\in K_{D}$ we have $A^*B_1(t)=A^*B(t)=-t$, it follows firstly that $K_D\subseteq K_{D_1}$, secondly that
$$Q_{B_1}=B_1(K_D)=B(K_D)$$ is a direct summand of $R^2$ of rank $1$ isomorphic to $K_D$, and thirdly that $K_{B_1}$ is a projective $R$-module of rank $1$ which is a direct summand of $R^2$; as
$$Q_{D_1}=D_1(s_D(Q_D))=s_D(Q_D)\subseteq K_{B_1},$$
it follows that $K_{B_1}=s_D(Q_D)$. We conclude that $B_1,D_1\in Um(\mathbb M_2(R))$ and $\det(B_1)=\det(D_1)=0$.
\end{proof}

\begin{remark}\label{rem1}
\textbf{(1)} As $\circled{\textup{2}}\Leftrightarrow \circled{\textup{3}}$ and as a $2\times 2$ matrix has a
(simple) extension iff its transpose has it, it follows that
statement $\circled{\textup{3}}$ holds iff there exists $(e^{\prime
},f^{\prime })\in R^2$ such that $(ae^{\prime }+bf^{\prime },ce^{\prime
}+df^{\prime })\in Um(R^{2})$.

\smallskip \textbf{(2)} The matrix $A$ is extendable (simply extendable) iff the $R$-algebra homomorphism $R\rightarrow E_{\upsilon }$ (resp.\ $R\rightarrow SE_{\upsilon }$) has a retraction. Similarly, statement $\circled{\textup{9}}$ holds iff the $R$-algebra homomorphism $R\rightarrow WZ_{\upsilon }$ has a retraction. 

\smallskip \textbf{(3)} We do not know if the implication $\circled{\textup{10}}\Rightarrow\circled{\textup{9}}$ is true when $R$ is not reduced. Modulo this, Theorem \ref{TH8} is optimal. This is so, as in general 
statement \circled{\textup{7}} does not imply \circled{\textup{4}}
 (see Theorem \ref{TH4} for the case of Dedekind domains which are not 
 \textsl{PID}s) and statement \circled{\textup{9}} does not imply 
 \circled{\textup{7}} even if $R$ is regular (see Example \ref{EX14} that involves, based on Theorem \ref{TH9}(2), a regular ring $\mathcal E_2$ of dimension $9$). 

\smallskip \textbf{(4)} Assume that $R$ is a ${WJ}_{2,1}$ ring. As $(a,b,c,d)\in Um(R^{4})$, the equation $1-ax-by-cz-dw=0$ has
solutions in $R^{4}$, thus from the definition of a ${WJ}_{2,1}$ ring it follows 
that statement $\circled{\textup{5}}$ holds.

\smallskip \textbf{(5)} If $A^+=\left[ 
\begin{array}{ccc}
a & b & f \\ 
c & d & -e \\ 
-t & s & 0\end{array}\right]$ is a simple extension of $A=\left[ 
\begin{array}{cc}
a & b \\ 
c & d\end{array}\right]\in Um(\mathbb{M}_2(R))$, then the characteristic polynomial $\chi_{A^+}$ of $A^+$ is of the form 
\begin{equation*}
x^3-{Tr}(A)x^2+\nu_{A^+}x-1
\end{equation*}
(see Section \ref{S1} for ${Tr}(A)=a+d$ and $\nu_{A^+}=\det(A)+es+ft$).
Thus the set of characteristic polynomials of simple extensions of $A$ is
in bijection to the subset $\nu_A\subseteq R$ introduced in Section \ref{S1}
and we have 
\begin{equation*}
\nu_A=\{\det(A)+es+ft|(e,f,s,t)\in R^4,\;a(es)+b(et)+c(fs)+d(ft)=1\}.
\end{equation*}
We include simple examples to show that $\nu_A$ can be an affine subset that is or is not an ideal of $R$. If $R=\mathbb{Z}$, $A={Diag}(7,11)$, $\{(8+11k,-5-7k)|k\in\mathbb{Z}\}$ is the solution set 
of the equation $7es+11ft=1$ in 
$es$ and $ft$ and $\nu_A=4\mathbb{Z}$. Also, $\nu_{Diag(1,d)}=2+R(d-1)$. 

\smallskip \textbf{(6)} For each $\mathfrak m\in {Max}R$, the image of $A$ in $Um(\mathbb M_2(R_{\mathfrak m}))$ is simply extendable (for instance, see Corollary \ref{C6}(2)) below). Thus, if $ab=bc$ and $Z_{\upsilon}$ is a polynomial $R$-algebra in $2$ indeterminates, Quillen Patching Theorem (see \cite{qui}, Thm. 1) implies (cf. last part of the proof of Theorem \ref{TH7}(5)) that the $\mathbb{G}_{m,Z_{\upsilon}}$-torsor of Theorem \ref{TH7}(4) is the pullback of a $\mathbb{G}_{m,R}$-torsor.
\end{remark}

\begin{example}
\normalfont\label{EX8} We check elementarily that each B\'{e}zout domain is
a $\Pi _{2}$ domain. Let $A\in Um(\mathbb{M}_{2}(R))$ with $\det (A)=0$. As $A$
is equivalent to an upper triangular matrix, based on Lemma \ref{L1}(3) we can
assume that $A=\left[ 
\begin{array}{cc}
a & b \\ 
0 & d\end{array}\right] $ is upper triangular with $ad=0$ and we will only assume that a
certain divisor of $b$ is not a zero divisor. Let $p,q,r\in R$ be such that $1=ap+bq+dr$. We have $(1-ap)(1-dr)=bq$ and hence, if either $\gcd (b,1-ap)$
or $\gcd (b,1-dr)$ is not a zero divisor, there exists $(e,s,t,v)\in R^4$ such
that we can write $b=tv$, $1-ap=st$, $1-dr=ev$. Then $a=a(1-dr)=aev$ and $be=evt=t-dtr$, hence $be+dtr=t$. Thus the ideal $Rae+R(be+dtr)$ contains $t$
and $a$, hence it is $R$ as $1=ap+st$. From this and Theorem \ref{TH8} (applied to $f:=tr$) it follows that $A$ is simply extendable.
\end{example}

\subsection{Applications to ${SE}_2$ and ${E}_2$ rings}

\label{S33}

\begin{corollary}
\label{C6} \textbf{(1)} The ring $R$ is an ${SE}_2$ ring iff for
each $(a,b,c,d)\in Um(R^{4})$ the equivalent statements $\circled{\textup{1}}
$, $\circled{\textup{2}}$, $\circled{\textup{3}}$ and $\circled{\textup{4}}$
hold.

\smallskip \textbf{(2)} Each semilocal ring is an ${SE}_2$ ring.
\end{corollary}

\begin{proof}
Part (1) follows from definitions and Theorem \ref{TH8}. Part (2)
follows from part (1) and the fact that, as $R$ is semilocal, for each $(a,b,c,d)\in Um(R^{4})$ there exists $(e,f)\in R^2$ such that the
intersection $U(R)\cap \{ea+cf,be+df\}$ is nonempty.
\end{proof}

\begin{corollary}
\label{C7} Let $(a,b,c,d)\in Um(R^{4})$. The following are equivalent:

\medskip \textbf{(1)} The matrix $\left[ 
\begin{array}{cc}
a & b \\ 
c & d\end{array}\right] $ is extendable.

\smallskip \textbf{(2)} There exists $(e,f)\in R^2$ such that $(ae+cf,be+df,ad-bc)\in Um(R^3)$.

\medskip 
Similarly, (1) with simply extendable is equivalent to (2) with $(e,f)\in Um(R^2)$. Moreover, 
if $sr(R)\le 2$, then the extendable and simply extendable properties on a matrix in $\mathbb M_2(R)$ are equivalent.
\end{corollary}

\begin{proof}
The 'extendable' case and the implication $(1)\Rightarrow (2)$ in the 'simply extendable' case follow from Lemma \ref{L1}(1) and Theorem \ref{TH8}. We check that if there exists $(e,f)\in Um(R^2)$ such that $(ae+cf,be+df,ad-bc)\in Um(R^3)$, then $A$ is simply extendable. Based on Theorem \ref{TH8} it suffices to show
that we have $(ae+cf,be+df)\in Um(R^{2})$. Let $I:=R(ae+cf)+R(be+df)$ and $\mathfrak{m}\in {Max}R$. If $ad-bc\in \mathfrak{m}$, then $(ae+cf,be+df,ad-bc)\in Um(R^{3})$ implies that 
$I\not\subseteq \mathfrak{m}$. If $ad-bc\notin \mathfrak{m}$, then $A$ modulo $\mathfrak{m}$ is invertible, hence $I\subseteq \mathfrak{m}$ iff $Re+Rf\subseteq \mathfrak{m}$; from this and $(e,f)\in Um(R^{2})$ we infer that 
$I\not\subseteq \mathfrak{m}$. As $I$ is not contained in any maximal ideal of 
$R$, it follows that $(ae+cf,be+df)\in Um(R^{2})$.

To complete the proof it suffices to show that if $sr(R)\le 2$, each extendable $\left[ 
\begin{array}{cc}
a & b \\ 
c & d\end{array}\right]\in Um(\mathbb M_2(R))$ is simply extendable. 
As $(1)\Leftrightarrow (2)$, let $(e^{\prime},f^{\prime})\in R^2$ be such that 
$(ae^{\prime}+cf^{\prime},be^{\prime}+df^{\prime},ad-bc)\in Um(R^3)$. 
Thus $(e^{\prime},f^{\prime},ad-bc)\in Um(R^3)$. As $sr(R)\le 2$, there exists $(q,r)\in R^2$ such that 
$$(e,f):=\bigl(e^{\prime}+q(ad-bc),f^{\prime}+r(ad-bc)\bigr)\in
Um(R^2).$$ 
As the ideals of $R$ generated by $ae+cf,be+df,ad-bc$ and by $ae^{\prime}+cf^{\prime},be^{\prime}+df^{\prime},ad-bc$ are equal, 
we have $(ae+cf,be+df,ad-bc)\in Um(R^3)$. Thus from the previous paragraph it follows that $A$ is 
simply extendable.
\end{proof}

\begin{corollary}
\label{C8} The following properties hold:

\medskip \textbf{(1)} The ring $R$ is an ${E}_2$ ring iff for
each $(a,b,c,d)\in Um(R^{4})$ there exists $(e,f)\in R^2$ such that $(ae+cf,be+df,ad-bc)\in Um(R^3)$.

\smallskip \textbf{(2)} The ring $R$ is an ${SE}_2$ ring iff for
each $(a,b,c,d)\in Um(R^{4})$ there exists $(e,f)\in Um(R^2)$ such that $(ae+cf,be+df,ad-bc)\in Um(R^3)$.

\smallskip \textbf{(3)} If $sr(R)\leq 2$, then $R$ is an $SE_{2}$
ring iff $R$ is an $E_{2}$ ring.
\end{corollary}

\begin{proof}
All three parts follow directly from definitions and the corresponding three sentences of Corollary \ref{C7}.
\end{proof}

\begin{example}
\normalfont\label{EX9} Let $A=\left[ 
\begin{array}{cc}
a & b \\ 
c & d\end{array}\right]\in Um (\mathbb{M}_2(R))$. In many simple cases one can easily
prescribe $(e,f)\in R^{2}$ such that statement $\circled{\textup{3}}$ holds,
and hence conclude that $A$ is simply extendable. We include four such cases as follows.

\textbf{(1)} If $\{a,b,c,d\}\cap U(R)\neq \emptyset $, then we can take $(e,f)\in R^2$
such that $\{e,f\}=\{0,1\}$. E.g., if $a\in U(R)$, then $\left[ 
\begin{array}{ccc}
a & b & 0 \\ 
c & d & -1 \\ 
0 & a^{-1} & 0\end{array}\right] $ is a simple extension of $A$.

\textbf{(2)} If $\{(a,b),(a,c),(b,d),(c,d)\}\cap Um(R^{2})\neq \emptyset $,
then we can take $(e,f,e^{\prime },f^{\prime })\in R^4$ such that $1\in
\{ae+cf,be+df,ae^{\prime }+bf^{\prime },ce^{\prime }+df^{\prime }\}$. E.g., 
if $(a,b)\in Um(R^{2})$ and $s,t\in R$ are such that $as+bt=1$,
then $\left[ 
\begin{array}{ccc}
a & b & 0 \\ 
c & d & -1 \\ 
-t & s & 0\end{array}\right] $ is a simple extension of $A$.

\textbf{(3)} If at least two of the entries $a$, $b$, $c$ and $d$ are in $J(R)$ (e.g., they are $0$), then we can take $e,f$ such that $1\in
\{ae+bf,ce+df,ae+bf+ce+df\}$. E.g., if $c=d=0$ and $s,t\in R$ are
such that $as+bt=1$, then $\left[ 
\begin{array}{ccc}
a & b & 0 \\ 
0 & 0 & -1 \\ 
-t & s & 0\end{array}\right] $ is a simple extension of $A$. Similarly, if $b=c=0$ and $e,f\in R$
are such that $ae+df=1$, then $\left[ 
\begin{array}{ccc}
a & 0 & f \\ 
0 & d & -e \\ 
-1 & 1 & 0\end{array}\right] $ is a simple extension of $A$.

\textbf{(4)} If $a,b,c,d,f,q\in R$ are such that $aq+df=1$, then a simple
extension of $\left[ 
\begin{array}{cc}
a & ab \\ 
ac & d\end{array}\right] $ is $\left[ 
\begin{array}{ccc}
a & ab & f \\ 
ac & d & -q+cf(1-b) \\ 
-1 & 1-b & 0\end{array}\right] $.
\end{example}

\subsection{Applications to upper triangular matrices}

\label{S34}

In three of the last four examples, the $(2,3)$ entry of the simple extensions is $-1 $, i.e., we can choose $e=1$. Such extensions relate to
stable ranges 1 and 1.5, as the following result shows.

\begin{corollary}
\label{C9} The following properties hold:

\medskip \textbf{(1)} We have $sr(R)=1$ iff each upper triangular matrix $A\in Um(\mathbb{M}_{2}(R))$ has a simple extension $A^{+} $ whose $(2,3)$
entry is $-1$.

\smallskip \textbf{(2)} We have $fsr(R)=1.5$ iff each upper 
triangular matrix $A\in Um(\mathbb{M}_{2}(R))$ with nonzero $(1,1)$ entry has 
a simple extension $A^{+} $ whose $(2,3)$ entry is $-1$.

\smallskip \textbf{(3)} We have $asr(R)=1$ iff each upper 
triangular matrix $A\in Um(\mathbb{M}_{2}(R))$, whose $(1,1)$ entry does not belong to $J(R)$, has  a simple extension $A^{+} $ whose $(2,3)$ entry is $-1$.
\end{corollary}

\begin{proof}
Let $A=\left[ 
\begin{array}{cc}
a & b \\ 
0 & d\end{array}\right] \in Um(\mathbb{M}_{2}(R))$. If $a=0$, then $(b,d)\in Um(R^{2})$ and
from Equation (\ref{EQ8}) it follows that $A$ has a simple extension with
the $(2,3)$ entry $-1$ iff there exists $(f,t)\in R^{2}$ such that $bt+dft=1$ and hence iff there exists $f\in R$ such that $b+df\in
U(R)$. Thus all these matrices $A$ with $a=0$ have a simple extension with
the $(2,3)$ entry $-1$ iff $sr(R)=1$. 

Similarly, if $a\neq 0$ (resp.\ $a\notin J(R)$), then from Equation (\ref{EQ8}) it follows that $A$ has a simple extension with
the $(2,3)$ entry $-1$ iff there exists $(e,f,t)\in R^{3}$ such that $ae+bt+dft=1$ and hence iff there exists $f\in R$ such that $(b+df,a)\in
Um(R^2)$. From the definition of stable range 1.5 (resp.\ almost stable range 1) applied to $(b,c,a)\in Um(R^{3})$ it follows that all these matrices with $a\neq 0$ (resp.\ $a\notin J(R)$) have a simple extension with the $(2,3)$ entry $-1$ iff $fsr(R)=1.5$ (resp.\ $asr(R)=1$).
\end{proof}

\noindent \textbf{Proof of Corollary \ref{C3}(1).} We check that if $asr(R)=1$ then each matrix $A=\left[ 
\begin{array}{cc}
a & b \\ 
0 & c\end{array}\right] \in Um(\mathbb{M}_{2}(R))$ is simply extendable. Based on Example \ref{EX9}(3) we can assume that $a\notin J(R)$ and this case follows from Corollary \ref{C9}(3).

\begin{corollary}
\label{C10} A ring $R$ is an ${SE}^{\triangle}_{2}$ ring iff for all
homomorphisms $\phi :\mathbb{Z}[x,y,z]\rightarrow R$, the image of the
matrix $\mathcal{G}\in Um(\mathbb{M}_{2}(\mathbb{Z}[x,y,z]))$ (see
Subsection \ref{S26}) in $Um(\mathbb{M}_{2}(R))$ via $\phi $ is simply
extendable.
\end{corollary}

\begin{proof}
The `only if' part is obvious. To prove the `if' part, let $A=\left[ 
\begin{array}{cc}
a & b \\ 
0 & c\end{array}\right] \in Um(\mathbb{M}_{2}(R))$. Let $(a^{\prime },b^{\prime },c^{\prime
})\in R^3$ be such that $aa^{\prime }+bb^{\prime }+cc^{\prime }=1$. We consider the homomorphism $\phi :\mathbb{Z}[x,y,z]\rightarrow R$ that maps $x$, $y$ and $z$ to $aa^{\prime }$, $b$ and $b^{\prime }$ (respectively). The image of the matrix $\mathcal{G}$ via it, is the matrix $B=\left[ 
\begin{array}{cc}
aa^{\prime } & b \\ 
0 & cc^{\prime }\end{array}\right] \in Um(\mathbb{M}_{2}(R))$, hence $B$ is simply extendable. From
Theorem \ref{TH8} it follows that there exists $(e,f^{\prime})\in R^2$ such
that $(aa^{\prime }e,be+cc^{\prime }f^{\prime })\in Um(R^{2})$. Denoting $f:=c^{\prime }f^{\prime }\in R$, as $(aa^{\prime }e,be+cf)\in Um(R^{2})$, it
follows that $(ae,be+cf)\in Um(R^{2})$ and from Proposition \ref{TH8} it
follows that $A$ is simply extendable.
\end{proof}

\section{The proofs of the general results and applications}

\label{S4}

\subsection{Proof of Theorem \protect\ref{TH1}}

\label{S41}

\begin{proposition}
\label{PR3} Let $A\in Um(\mathbb{M}_{2}(R))$. Then the following properties
hold:

\medskip \textbf{(1)} If $\det (A)=0$, then $A$ is simply extendable iff it is non-full.

\smallskip \textbf{(2)} The matrix $A$ is extendable iff its reduction modulo $R\det(A)$ is
non-full.

\smallskip \textbf{(3)} If $A$ is extendable, then statement $\circled{\textup{10}}$ holds (hence, if $R$ is reduced, statement $\circled{\textup{9}}$ holds as well).
\end{proposition}

\begin{proof}
To prove the `if' part of (1), we write $A=\left[ 
\begin{array}{c}
l \\ 
m\end{array}\right] \left[ 
\begin{array}{cc}
n & q\end{array}\right] $. As $A\in Um(\mathbb{M}_{2}(R))$, it follows that $(l,m),(n,q)\in
Um(R^{2})$ and so there exist pairs $(e,f)$ and $(s,t)\in R^{2}$ such that $el+fm=sn+tq=1 $. Thus $\left[ 
\begin{array}{cc}
e & f\end{array}\right] A=\left[ 
\begin{array}{cc}
n & q\end{array}\right]$, and from Theorem \ref{TH8} it follows that $A$ is simply
extendable, a simple extension of it being $\left[ 
\begin{array}{ccc}
ln & lq & f \\ 
mn & mq & -e \\ 
-t & s & 0\end{array}\right] $.

To prove the `only if' part of (1), we assume that $\det (A)=0$ and $A$ is
simply extendable. From the second part of Corollary \ref{C7} it follows that there exists $(e,f)\in Um(R^{2})$ such that $(ae+cf,bf+df)\in Um(R^{2})$, hence there exist pairs $(p,r)$ and $(s,t)\in Um(R^{2})$ such that $ep+fr=(ae+cf)s+(be+df)t=1$. The
matrices $M:=\left[ 
\begin{array}{cc}
e & f \\ 
-p & r\end{array}\right] $ and $N:=\left[ 
\begin{array}{cc}
s & -be-df \\ 
t & ae+cf\end{array}\right] $ have determinant $1$ and a simple computation shows that, as $\det
(A)=0$, we have $MAN=\left[ 
\begin{array}{cc}
1 & 0 \\ 
w & 0\end{array}\right]=\left[ 
\begin{array}{c}
1 \\ 
w\end{array}\right] \left[ 
\begin{array}{cc}
1 & 0\end{array}\right] $ for some $w\in R$. Hence $A=(M^{-1}\left[ 
\begin{array}{c}
1 \\ 
w\end{array}\right] )(\left[ 
\begin{array}{cc}
1 & 0\end{array}\right] N^{-1})$ is non-full. Thus part (1) holds.

Part (2) follows from part (1) and Lemma \ref{L1}(1).

To prove part (3) we first note that part (2) implies that there exist $\bar{l},\bar{m},\bar{n},\bar{q}\in R/R\det (A)$ such that $A$ modulo $R\det (A)$
is $\left[ 
\begin{array}{c}
\bar{l} \\ 
\bar{m}\end{array}\right] \left[ 
\begin{array}{cc}
\bar{n} & \bar{q}\end{array}\right] $. If $l,m,n,q\in R$ lift $\bar{l},\bar{m},\bar{n},\bar{q}$
(respectively), then $B:=\left[ 
\begin{array}{c}
l \\ 
m\end{array}\right] \left[ 
\begin{array}{cc}
n & q\end{array}\right] $ is congruent to $A$ modulo $R\det (A)$ and $\det (B)=0$. Thus 
statement \circled{\textup{10}} holds, hence part (3) holds.
\end{proof}

We are now ready to prove Theorem \ref{TH1}. Recall from Subsection \ref{S25} that if $A\in Um(\mathbb M_2(R))$ has zero determinant, then $K_A$ and $Q_A$ are projective $R$-modules of rank $1$ dual to each other and $K_A\oplus Q_A\cong R^2$. 

From the previous paragraph it follows that $(4)\Rightarrow (3)$. The
implication $(3)\Rightarrow (2)$ is obvious. The equivalence $(1)\Leftrightarrow (2)$ follows from Proposition \ref{PR3}(1).

The implication $(1)\Rightarrow (4)$ follows from the fact that each projective $R $-module of rank $1$ generated by $2$ elements is isomorphic to $Q_{A}$
for some idempotent $A\in \mathbb{M}_{2}(R) $ of rank $1$ (and hence
unimodular of determinant $0$), and applying part (1) to $A$, it follows that $Q_A$
is isomorphic to $R$. Thus Theorem \ref{TH1} holds.

\begin{example}
\normalfont\label{EX10} For $A\in Um(\mathbb{M}_{2}(R))$ with $\det
(A)=0$, we consider the $R$-submodule $K_{A}^{\prime }:=R\left[ 
\begin{array}{c}
-b \\ 
a\end{array}\right] +R\left[ 
\begin{array}{c}
-d \\ 
c\end{array}\right] $ of $K_{A}$. For each $\mathfrak{m}\in {Max}R$, $K_{A}^{\prime }\not\subseteq\mathfrak{m}\mathbb{M}_{2}(R)$, hence $K_{A}^{\prime }\not\subseteq\mathfrak{m}K_{A}$. This implies that $K_{A}=K_{A}^{\prime }$. For $(e,f)\in R^{2}$, statement \circled{\textup{3}}
holds iff $K_{A}$ is free having $e\left[ 
\begin{array}{c}
-b \\ 
a\end{array}\right] +f\left[ 
\begin{array}{c}
-d \\ 
c\end{array}\right] $ as a generator.
\end{example}

\begin{example}
\normalfont\label{EX11} Let $x$ be an indeterminate and let $k\in \mathbb{N}$. Let $q:=4k+1$, $r:=2k+1$ and $\theta :=i\sqrt{q}\in \mathbb{C}$. We check
that $\mathbb{Z}[x]$ \emph{is a }${\Pi}_{2}$\emph{\ ring which is not an} $E_{2}$ \emph{ring}. As $\mathbb{Z}[x]$ is a UFD, it is also a Schreier domain and hence a ${\Pi}_{2}$
ring (see Section \ref{S1}). Thus it suffices to show that the matrix 
\begin{equation*}
A:=\left[ 
\begin{array}{cc}
r & 1-x \\ 
1+x & 2\end{array}\right] \in Um(\mathbb{M}_{2}(\mathbb{Z}[x]))
\end{equation*}is not extendable. As $\det (A)=x^{2}+q$, based on Lemma \ref{L1}(1), it
suffices to show that the image $B:=\left[ 
\begin{array}{cc}
r & 1-\theta \\ 
1+\theta & 2\end{array}\right] \in \mathbb{M}_{2}(\mathbb{Z}[\theta ])$ of $A$, via the composite
homomorphism 
$\mathbb{Z}[x]\rightarrow \mathbb{Z}[x]/(x^{2}+q)\rightarrow \mathbb{Z}[\theta ]$
that maps $x$ to $\theta $, is not simply extendable, equivalently, it is
full. A standard argument on norms shows that the element $2\in\mathbb{Z}[\theta ]$
is irreducible, hence its only decompositions into products of two elements are $2=u (2u^{-1})$ with $u\in U(\mathbb Z[\theta ])$. Thus, as $2u^{-1}$ divides
neither $1-\theta$ nor $1+\theta$, $B$ is full. We
also conclude that the integral domain $\mathbb{Z}[\theta ]$ is not a ${\Pi}_{2}$ ring. If $4k+1$ is square free, then $\mathbb{Z}[\theta ] $
is a Dedekind domain but not a \textsl{PID}.
\end{example}

\begin{remark}\label{rem2}
\textbf{(1)} The equivalent statements \circled{\textup{1}} to \circled{\textup{4}} are clearly stable
under similarity (inner automorphisms of the $R$-algebra $\mathbb{M}_2(R)$)
but in general they are not stable under arbitrary $R$-algebra automorphisms
of $\mathbb{M}_2(R)$. To check this, let $R$ be such that there
exists an $R$-module $P$ such that $P\oplus P=R^2$ but $P\ncong R$; so $[P]\in {Pic}(R)$ has order $2$. Then the idempotent $A$ of $\mathbb{M}_2(R)$ which is a projection of $R^2$ on the first
copy of $P$ along the second copy of $P$ satisfies $Q_A=P$. Hence $A$ is not
simply extendable but its image under the $R$-algebra automorphism $\mathbb{M}_2(R)={End}_R(P\oplus P)\cong \mathbb{M}_2(R)$ defined by $End_R(P)\cong R$ maps $A$ to $Diag(1,0)$. Concrete
example: if $R$ is a Dedekind domain with ${Pic}(R)\cong\mathbb{Z}/2\mathbb{Z}$ (such as $\mathbb{Z}[\sqrt{-5}]$), then we can take $P$ to be an
arbitrary maximal ideal of $R$ whose class in ${Pic}(R)$ is
nontrivial.

\smallskip \textbf{(2)} The fact that if ${Pic}(R)$ is trivial then $R$ is a $\Pi_2$ ring also follows easily from Theorem \ref{TH7}(3) and (4) via the
equivalence between $\mathbb{G}_m$-torsors and line bundles.
\end{remark}

\begin{corollary}
\label{C11} If $R$ is a $\Pi_2$ ring, then for each $A=\left[ 
\begin{array}{cc}
a & b \\ 
c & d\end{array}\right] \in Um(\mathbb{M}_{2}(R))$, statements $\circled{\textup{1}}$ to $\circled{\textup{8}}$ are equivalent.
\end{corollary}

\begin{proof}
Based on Theorem \ref{TH8}, it suffices to prove that $\circled{\textup{7}}\Rightarrow \circled{\textup{2}}$. We can assume that $\det(A)\neq 0$. Let $B=\left[ 
\begin{array}{cc}
a^{\prime} & b^{\prime} \\ 
c^{\prime} & d^{\prime}\end{array}\right]\in Um(\mathbb{M}_{2}(R))$ be such that $\det(B)=0$ and $B$ is congruent to $A$ modulo $R\det(A)$. As $R$ is a $\Pi_2$ ring, $B$ is simply extendable. From this and Theorem \ref{TH8} it follows that there exists $(e,f)\in Um(R^2)$ such that $(a^{\prime}e+c^{\prime}f,b^{\prime}e+d^{\prime}f)\in Um(R^2)$ and so $(a^{\prime}e+c^{\prime}f,b^{\prime}e+d^{\prime}f, ad-bc)\in Um(R^2)$. Hence, as $B$ is congruent to $A$ modulo $R\det(A)$, it follows that $(ae+cf,be+df, ad-bc)\in Um(R^2)$. Second part of Corollary \ref{C7} implies that $A$ is simply extendable.
\end{proof}

\begin{corollary}
\label{C12} If $R$ is a product of pre-Schreier domains of stable range at most $2$, then for each $A=\left[ 
\begin{array}{cc}
a & b \\ 
c & d\end{array}\right] \in Um(\mathbb{M}_{2}(R))$, statements \circled{\textup{1}} to 
$\circled{\textup{10}}$ are equivalent.
\end{corollary}

\begin{proof} We can assume that $R$ is a pre-Schreier domain and $sr(R)\le 2$. Thus $R$ is reduced and a $\Pi_2$ ring. Therefore, based on Corollary \ref{C11} and Theorem \ref{TH8}, it suffices to show that \circled{\textup{10}} implies \circled{\textup{2}}. Let $B\in\mathbb{M}_{2}(R)$ be such that $\det(B)=0$ and $B$ is congruent to $A$ modulo $R\det(A)$. As $B$ is a pre-Schreier domain, $B$ is non-full, hence $A$ modulo $R\det(A)$ is non-full. Thus $A$ modulo $R\det(A)$ is simply extendable (see Theorem \ref{TH1}) and we infer that $A$ is extendable (see Lemma \ref{L1}(1)). From this and the last part of Corollary \ref{C7} it follows that $A$ is simply extendable.
\end{proof}

\subsection{An application of Theorem \ref{TH1}}

\label{S44}

We recall the following well-known lemma.

\begin{lemma}
\label{L2} Let $A\in Um(\mathbb{M}_{2}(R))$, let $t\in R$ be such that $\det
(A)\in Rt$ and let $\hat{R}$ be the $t$-adic completion of $R$. Then there
exists $B\in Um(\mathbb{M}_{2}(\hat{R}))$ whose reduction modulo $\hat{R}t$ is the reduction of $A$ modulo $Rt$ and $\det(B)=0$.
\end{lemma}

\begin{proof}
Let $B_0:=A$. By induction on $n\in \mathbb{N}$, we show that there exists $B_{n}\in 
\mathbb{M}_{2}(R)$ congruent to $B_{n-1}$ modulo $Rt^{2^{n-1}}$ and $\det (B_{n})\in
Rt^{2^{n}}$. For $n=1$, let $s\in R$ be such that $\det (A)=st$. Writing $A=\left[ 
\begin{array}{cc}
a & b \\ 
c & d\end{array}\right] $, for $B_{1}:=A+t\left[ 
\begin{array}{cc}
x & y \\ 
z & w\end{array}\right] \in \mathbb{M}_{2}(R)$, $\det (B_{1})$ is congruent modulo $Rt^{2}$
to $st+(dx+cy+bz+aw)t$. As $A\in Um(\mathbb{M}_{2}(R))$, the linear
equation $dx+cy+bz+aw=-s$ has a solution $(x,y,z,w)\in
R^{4}$ and for such a solution we have $\det (B_{1})\in Rt^{2}$. The passage
from $n$ to $n+1$ follows from the case $n=1$ applied to $B_{n}$ and $Rt^{2^{n+1}}$ instead of to $A$ and $Rt^{2}$. This completes the induction.

As $t\in J(\hat{R})$, the limit $B\in \mathbb{M}_{2}(\hat{R})$ of the sequence $(B_n)_{n\ge 1}$ exists. We have $\det(B)=0$. As $t\in J(\hat{R})$, we have $B\in Um(\mathbb M_2(\hat{R}))$. 
\end{proof}

Lemma \ref{L2} also follows from the smoothness part of Theorem \ref{TH7}(2) via a standard lifting argument.

\begin{proposition}
\label{PR4} Let $A\in Um(\mathbb{M}_{2}(R))$. If the $\det (A)$-completion $\hat{R}$ of $R$ is a ${\Pi}_{2}$ ring, then $A$ is extendable.
\end{proposition}

\begin{proof}
As $\hat{R}$ is a ${\Pi}_{2}$ ring, each $B\in Um(\mathbb{M}_2(\hat{R}))$ as in Lemma \ref{L2} is simply extendable; hence its reduction modulo $\hat{R}\det(A)$, i.e., $A$ modulo $R\det(A)$, is also simply extendable. Thus the proposition follows from Lemma \ref{L1}(1).
\end{proof}

\subsection{Proof of Theorem \protect\ref{TH2}}

\label{S410}

The equivalence of statements (1), (2) and (3) follows from the equivalence of statements \circled{\textup{6}}, \circled{\textup{7}} and \circled{\textup{8}} (see Theorem \ref{TH8}). 

To check that $(4)\Rightarrow (1)$, let $A\in Um(\mathbb M_2(R))$. To show that there exists $B\in Um(\mathbb M_2(R))$ congruent to $A$ modulo $R\det(A)$ and with $\det(B)=0$, we can assume that $\det(A)\neq 0$; in this case the existence of $B$ follows from statement (4) applied to $a:=\det(A)$ and to the image of $A$ in $\{\bar{A}\in Um(\mathbb M_2(R/Ra))|\det(\bar{A})=0\}$. 

We are left to check that if $sr(R)\le 4$, then $(1)\Rightarrow (4)$. Let $a\in R$ be nonzero and let $\bar A\in Um(\mathbb M_2(R/Ra))$ with $\det(\bar{A})=0$. As $sr(R)\le 4$, there exists $A\in Um(\mathbb M_2(R))$ whose reduction modulo $Ra$ is $\bar A$; we have $\det(A)\in Ra$. As we are assuming $R$ is a ${Z}_2$ ring, there exists $B\in Um(\mathbb M_2(R))$ such that $\det(B)=0$ and $B$ is congruent to $A$ modulo $R\det(A)$, hence also modulo $Ra$. Thus statement (4) holds and so Theorem \ref{TH2} holds. 

\subsection{Proof of Theorem \protect\ref{TH3}}

\label{S42}

To prove part (1), let $A\in Um(\mathbb{M}_{2}(R))$. If $R$
is an $E_2$ ring, then firstly from Proposition \ref{PR3}(3) it follows that there exists $B\in\mathbb{M}_{2}(R)$ congruent to $A$ modulo $R\det(A)$ and with $\det(B)=0$; 
hence $R$ is a ${WZ}_2$ ring. Secondly, with the notation of the proof of Proposition \ref{PR3}(3), as $A$ unimodular it follows that $(\bar{l},\bar{m}),(\bar{n},\bar{q})\in Um(R/R\det (A))$. From this and Fact 
\ref{F1}, if $sr(R)\leq 2$, we can assume that $(l,m),(n,q)\in Um(R^{2})$, so
$B:=\left[ 
\begin{array}{c}
l \\ 
m\end{array}\right] \left[ 
\begin{array}{cc}
n & q\end{array}\right] \in Um(\mathbb{M}_{2}(R))$ is congruent to $A$ modulo $R\det(A)$ and has zero determinant; thus $R$ is a ${Z}_{2}$ ring.
If moreover $R$ is an ${SE}_{2}$ ring, then $A$ is simply
extendable and Theorem \ref{TH8} implies that statement $\circled{\textup{6}}$ holds; thus $R$ is a ${Z}_{2}$ ring. Hence part
(1) holds.

Based on part (1), to prove part (2), we can assume $R$ is a pre-Schreier
domain and it suffices to show that if $R$ is a ${WZ}_{2}$ ring then it is also a 
$E_2$ ring. As $R$ is a 
$\Pi _{2}$ domain, it suffices to show that if $A\in Um(\mathbb{M}_{2}(R))$ has
nonzero determinant, then $A$ is extendable, but the argument for this is similar to the one of the proof of Corollary \ref{C12} involving a matrix $B$ (as $R$ is a pre-Schreier domain, each $B\in\mathbb M_2(R)$ of zero determinant is non-full). Thus part (2) holds. 

Part (3) follows from Corollary \ref{C11}, as for each $A\in Um(\mathbb M_2(R))$, statements \circled{\textup{2}} and \circled{\textup{7}} are equivalent. 

For part (4), let $(a,b,c)\in Um(R^{2})$. In this paragraph we assume that $R$ is an ${SE}^{\triangle}_{2}$
ring; hence the matrix $\left[ 
\begin{array}{cc}
a & b \\ 
0 & c\end{array}\right] $ is simply extendable. From this and Theorem \ref{TH8}, it
follows that statement \circled{\textup{5}} holds, thus there exists $(x,y,z,w)\in Um(R^{4})$ such that $xw=yz$ and $1=ax+by+cw$. Hence the fact that $R$ is a $V_2$ ring follows from the computation 
\begin{equation*}
(1-ax)(1-cw)=1-ax-cw+acxw=by+acyz=y(b+acz).
\end{equation*}
For the remaining part of the proof of part (4) we will assume that $(a,b)\in Um(R^{2})$. For a fixed unit $\bar{u}\in U(R/Rc)$,
let $u\in R$ be such that $\bar{u}=u+(c)$ and consider the matrix $\left[ 
\begin{array}{cc}
ac & u \\ 
0 & bc\end{array}\right] \in Um(\mathbb{M}_{2}(R))$.

If $R$ is a $\Pi _{2}$ and a $WV_2$ ring, by applying the definition of a $WV_2$ ring to the triple $(ab,u,ac)\in Um(R^{3})$ it follows that there
exists a nonuple $(x,y,z_1,z_2,w,q,r,s,t)\in R^{9}$ such that (see Equation (\ref{EQ3})) 
\begin{equation*}
D:=\left[ 
\begin{array}{cc}
s-acx & qru+acz_1+bcz_2 \\ 
y & t-bcw\end{array}\right] \in Um(\mathbb{M}_{2}(R))
\end{equation*}has zero determinant and moreover $(ac,sq),(bc,rt)\in Um(R^2)$. From Theorem \ref{TH1}, as we are assuming $R$ is a $\Pi _{2}$ ring, it follows that $D$
is non-full and thus we can decompose $qru+acz_1+bcz_2=u_{a}u_{b}$, with $u_{a}\in R$ dividing $s-acx$ and $u_{b}\in R$ dividing $t-bcw$. Hence $(q+Rac)^{-1}\cdot (u_{a}+Rac)\in U(R/Rac)$ and $(r+Rbc)^{-1}\cdot (u_{b}+Rbc)\in U(R/Rbc)$ are such that the product of their images in $U(R/Rc)$ is $\bar{u}$. Thus $R$ is a $U_{2}$ ring.

We are now assuming that $R$ is an ${SE}_2^{\triangle}$ domain, and we want to prove that $R$ is a $U_2$ ring. For $(a,b)\in Um(R^2)$ and $c\in R$, to check that the natural product
homomorphism $U(R/Rac)\times U(R/Rbc)\rightarrow U(R/Rc)$ is surjective we can assume that $abc\neq 0$. As the matrix $\left[ 
\begin{array}{cc}
ac & u \\ 
0 & bc\end{array}\right] $ is simply extendable, from Theorem \ref{TH1} it follows that the
reduction modulo $(abc^{2})$ of $\left[ 
\begin{array}{cc}
ac & u \\ 
0 & bc\end{array}\right] $ is non-full and hence the system of congruences 
\begin{equation*}
xy\equiv ac\mod abc^{2},\;\,xw\equiv u\mod abc^{2},\;\,yz\equiv 0\mod abc^{2},\;\,zw\equiv bc\mod abc^{2}
\end{equation*}has a solution $(x,y,z,w)\in R^{4}$. From the second congruence it follows
that $(xw,c)\in Um(R^2)$. From this and the first and fourth
congruences, it follows that there exists $(y^{\prime },z^{\prime })\in R^2$ such
that $y=cy^{\prime }$ and $z=cz^{\prime }$. Thus the above first system of
congruences is equivalent to a second pne
\begin{equation*}
xy^{\prime }\equiv a\mod abc,\;\;xw\equiv u\mod abc^{2},\;\;y^{\prime
}z^{\prime }\equiv 0\mod ab,\;\;z^{\prime }w\equiv b\mod abc.
\end{equation*}As $(a,b)\in Um(R^{2})$, from the last two congruences it follows firstly that 
$(z^{\prime },a)\in Um(R^{2})$ and secondly that there exists $y^{\prime
\prime }\in R$ such that $y^{\prime }=ay^{\prime \prime }$. Thus the second
system of congruences is equivalent to a third one 
\begin{equation*}
xy^{\prime \prime }\equiv 1\mod bc,\;\;xw\equiv u\mod abc^{2},\;\;y^{\prime
\prime }z^{\prime }\equiv 0\mod b,\;\;z^{\prime }w\equiv b\mod abc.
\end{equation*}From the first and third congruences it follows firstly that $(y^{\prime
\prime },b)\in Um(R^{2})$ and secondly that there exists $z^{\prime \prime
}\in R$ such that $z^{\prime }=bz^{\prime \prime }$. Thus the third system
of congruences is equivalent to a fourth one
\begin{equation*}
xy^{\prime \prime }\equiv 1\mod bc,\;\;xw\equiv u\mod abc^{2},\;\;z^{\prime
\prime }w\equiv 1\mod ac
\end{equation*}that has only three congruences. Thus $x\in
U(R/Rbc)$ and $w+Rac\in U(R/Rac)$ are such that the product of their images in $R/Rc$ is $\bar{u}$. Hence $R$ is a $U_2$ ring and so part (4) holds.

Part (5) follows from Fact \ref{F1} applied to $n=4$ and definitions.

Part (6) follows from the fact that each $A\in Um(\mathbb{M}_2(R))$ is
extendable as its reduction modulo $R\det(A)$ is simply extendable (see
Lemma \ref{L1}(1) and the hypothesis that $R/R\det(A)$ is a $\Pi_2$ ring). Thus Theorem \ref{TH3} holds.

\subsection{Proof of Theorem \protect\ref{TH4}}

\label{S43}
Let $R$ be an integral domain of dimension 1.

To prove part (1), let $A\in Um(\mathbb{M}_{2}(R))$ with $\det (A)\neq 0$. 
The ring $\bar{R}:=R/\det(A)R$ has dimension $0$ and hence $Pic(\bar{R})$ is trivial. 
From this and Theorem \ref{TH1} it follows that $\bar{R}$ is a $\Pi_2$ ring and therefore the reduction $\bar{A}$ of $A$ modulo $R\det (A)$ is simply extendable. From Theorem \ref{TH1} it follows that $Q_{\bar{A}}\cong R/R\det (A)$. As $R
$ is an integral domain, we have two projective resolutions $0\rightarrow R\det(A)\rightarrow R\rightarrow Q_{\bar{A}}\rightarrow 0$ and (see Subsection \ref{S42}) $0\rightarrow \det(A)R^2\rightarrow Q_A\rightarrow Q_{\bar{A}}\rightarrow 0$. From
this and Subsection \ref{S24} applied to $b=\det(A)$ (recall that $sr(R)\le 2$), it follows that $A$ is equivalent to ${Diag}(1,\det (A)) $ and hence is simply extendable (see Theorem \ref{TH8}).

Each triangular matrix
in $Um(\mathbb{M}_{2}(R))$ has either two zero entries or a nonzero
determinant. From this, Example \ref{EX9}(3) and part (1)
it follows that $R$ is an ${SE}^{\triangle}_{2}$ ring. From this and Theorem \ref{TH3}(4) it follows that $R$ is a ${V}_{2}$ ring. To check that $R$ is a 
$Z_2$ ring we have to show that statement \circled{\textup{7}} holds for each 
$A\in Um(\mathbb{M}_{2}(R))$. We can assume that $\det(A)\neq 0$, in which case statement  \circled{\textup{7}} follows from 
part (1) and Theorem \ref{TH8}. Thus part (2) holds.

The first `iff' of part (3) follows from part (1) and definitions. The isomorphism classes of projective $R$-modules of rank 1 are the isomorphisms classes of nonzero ideals of $R$ which locally in the Zariski topology are principal; as for each $a\in R\setminus\{0\}$, $\dim(R/Ra)=0$ and hence $Pic(R/Ra)$ is trivial, all such nonzero ideals are generated by 2 elements. From this and Theorem \ref{TH1} it follows that $R$ is a $\Pi_2$ domain iff $Pic(R)$ is trivial. Hence part (3) holds. As PIDs are precisely Dedekind domains with trivial Picard groups, part (4) follows from the second `iff' of part (3). Thus Theorem \ref{TH3} holds.

\begin{remark}\label{rem3} The existence of an extension of a matrix in $Um(\mathbb{M}_{2}(R))$ does not depend only on \emph{the set} of its entries. This is so
as, referring to Example \ref{EX11}, the matrix $\left[ 
\begin{array}{cc}
1+\theta & 1-\theta \\ 
r & 2\end{array}\right] \in \mathbb{M}_{2}(\mathbb{Z}[\theta ])$ has the same entries as $B$, has nonzero determinant, and \emph{it is} extendable (see Theorem \ref{TH4}(1)).
\end{remark}

\subsection{Proof of Theorem \protect\ref{TH5}}

\label{S45}

The implications $(1)\Rightarrow (2)\Rightarrow (3)\Rightarrow (4)$ hold for all rings (see Example \ref{EX1} and Theorem \ref{TH3}(1) and (3)). 
As the identity $\mathbb{M}_{2}(R)=RUm(\mathbb{M}_{2}(R))$ holds for the Hermite ring $R$, the implication $(2)\Rightarrow (1)$ follows from Theorem \ref{TH8}. 
The implication $(4)\Rightarrow (2)$ is a particular case of Corollary \ref{C8}(3) as $sr(R)\le 2$. 
Thus
statements (1) to (4) are equivalent.

As each ${SE}_{2}$ ring is an ${SE}^{\triangle}_{2}$ ring and a $\Pi _{2}$
ring and each ${V}_{2}$ ring is a ${WV}_{2}$ ring, from Theorem \ref{TH3}(4) it follows that $(2)\Rightarrow (5)$ and $(6)\Rightarrow (7)$. Clearly, 
$(5)\Rightarrow (6)$. We have $(7)\Leftrightarrow (8)$, as $(8)$ just translates 
the definition of a $U_{2} $ ring for the unit $d+Rc\in U(R/Rc)$. Clearly, $(8)\Rightarrow (9)$. 

To check that $(9)\Rightarrow (8)$, let $(a,b),(c,d)\in Um(R^2)$. If $(e,f,s,t)\in R^4$ is such that $ae+bf=cs+dt=1$, then by applying (9) to $(ae,d)$ and $sc\in 1+Rd$ it follows that there exists $t_1\in R$ such that we can decompose $d+t_1sc=d_1d_2$ with $(ae,d_1),(1-ae,d_2)=(b,d_2)\in Um(R_2)$ and thus $(a,d_1),(b,d_2)\in Um(R_2)$ and statement (8) holds by taking by taking $t:=t_1s$.

The equivalence $(4)\Leftrightarrow (10)$ follows from Theorem \ref{TH3}(5) and (6).

We check that $(7)\Rightarrow (4)$. We have to show that,
if $R$ is a $U_{2}$ ring, then each $A\in Um(\mathbb{M}_{2}(R))$ is
extendable. Based on Lemma \ref{L1}(3), we can assume that $A$ is upper
triangular, and hence we can write $A=\left[ 
\begin{array}{cc}
ac & u \\ 
0 & bc\end{array}\right] $ where $a,b,c,u\in R$ are such that $\det (A)=abc^{2}$, $(a,b)\in
Um(R^{2})$ and $u\in R$ is an arbitrary representative of a fixed unit $u+Rc\in U(R/Rc)$. Hence, as we are assuming $R$ is a $U_{2}$ ring, we can
assume there exists a product decomposition $u=de$ where $d,e\in R$ are such
that $d+Rac\in U(R/Rac)$ and $e+Rbc\in U(R/Rbc)$. Let $d^{\prime },e^{\prime
}\in R$ be such that $d^{\prime }+Rac$ is the inverse of $d+Rac$ and $e^{\prime }+Rac$ is the inverse of $e+Rbc$. Let $f_{a},f_{b}\in R$ be such
that $dd^{\prime }=1+acf_{a}$ and $ee^{\prime }=1+bcf_{b}$. The matrix 
\begin{equation*}
B:=\left[ 
\begin{array}{cc}
ac(1+bcf_{b}) & u \\ 
abc^{2}d^{\prime }e^{\prime } & bc(1+acf_{a})\end{array}\right] =\left[ 
\begin{array}{cc}
acee^{\prime } & de \\ 
abc^{2}d^{\prime }e^{\prime } & bcdd^{\prime }\end{array}\right] \in \mathbb{M}_{2}(R)
\end{equation*}is congruent to $A$ modulo $R\det (A)$ and is non-full. Thus $A$ modulo $R\det (A)$ is non-full, hence simply extendable (see Theorem \ref{TH1}) and
from Lemma \ref{L1}(1) it follows that $A$ is extendable.

Thus the ten statements (1) to (10) of Theorem \ref{TH5} are equivalent.

We are left to check that $(11)\Rightarrow (7)$ when $R$ is a B\'{e}zout
domain. Given a unit $d+Rc\in U(R/Rc)$ and $(a,b)\in Um(R^{2})$, let $a^{\prime
}:=\gcd (a,c)$ and $b^{\prime }:=\gcd (b,c)$. As $\gcd (a,b)=1$, there
exists $c^{\prime }\in R$ such that $c=a^{\prime }b^{\prime }c^{\prime }$.
From the definition of a $W_2$ ring applied to the triple $(a^{\prime },b^{\prime },c^{\prime
})$ it follows that there exist units $e+Rc,f+Rc\in U(R/Rc)$ such that their
images in $U(R/Ra^{\prime })\times U(R/Rb^{\prime })$ are $(d+Ra^{\prime
},1+Rb^{\prime })$ and $(1+Ra^{\prime },d+Rb^{\prime })$ (respectively).
Then $(d+Ra,1+Rb)\in R/Rab\cong R/Ra\times R/Rb$ and $e+Rc\in R/Rc$ (resp.\ $(1+Ra,d+Rb)\in R/Rab\cong R/Ra\times R/Rb$ and $f+Rc\in R/Rc$) map to the
same element in $R/Ra^{\prime }b^{\prime }\cong R/Ra^{\prime }\times
R/Rb^{\prime }$ and hence, as $abc^{\prime }$ is the least common multiple
of $c$ and $ab$, there exists $(g,h)\in R^2$ such that $g+Rabc^{\prime }$ (resp.\ $h+Rabc^{\prime }$) reduces to both of them. As $abc^{\prime }+Rabc\in R/Rabc$
is nilpotent, the homomorphism $U(R/Rabc)\rightarrow U(R/Rabc^{\prime })$ is
onto and we conclude that $(e+Rc)(f+Rc)\in {Im}(U(R/Rac)\rightarrow U(R/Rc)){Im}(U(R/Rbc)\rightarrow U(R/Rc))$. Hence we can assume that $d-1\in
Ra^{\prime }b^{\prime }$, i.e., the images of $d+Rc$ and $1+Rab$ in $Ra^{\prime}b^{\prime}$ are equal, which implies that $d+Rc$ belongs to ${Im}(U(R/Rabc^{\prime
})\rightarrow U(R/Rc))$ and thus also to ${Im}(U(R/Rabc)\rightarrow
U(R/Rc))$. Hence Theorem \ref{TH5} holds.

\begin{corollary}
\label{C13} Let $R$ be a Hermite ring. Then $R$ is an \textsl{EDR}
iff for all homomorphisms $\phi :\mathbb{Z}[x,y,z]\rightarrow R$,
the image of $\mathcal{D}\in Um(\mathbb{M}_{2}(\mathbb{Z}[x,y,z]))$
(equivalently, $\mathcal{F}\in Um(\mathbb{M}_{2}(\mathbb{Z}[x,y,z]))$, see
Subsection \ref{S26}) in $Um(\mathbb{M}_{2}(R))$ via $\phi $ is extendable.
\end{corollary}

\begin{proof}
The `only if' part is obvious. To prove the `if' part, based on Theorem \ref{TH5}, it suffices to prove that $R$ is a $U_{2}$ ring, i.e., for all $a,c\in R$, the reduction product homomorphism $\rho _{a,c}:U(R/Rac)\times
U(R/R(1-a)c))\rightarrow U(R/Rc)$ is surjective (see Remark \ref{rem0}). Let $u\in R$ be such that $(u,c)\in Um(R^2)$. We want to prove that $u+Rc\in {Im}(\rho _{a,c})$. Let $s,t\in R$ be such that $cs+tu=1$. As $u+Rcs\in U(R/Rcs)$, it suffices to prove that $u+Rcs\in {Im}(\rho _{a,cs})$. Thus by replacing $c$ with $cs$ we can assume that $c=1-tu$. Hence $A=\left[ 
\begin{array}{cc}
a(1-tu) & u \\ 
0 & (1-a)(1-tu)\end{array}\right] \in Um(\mathbb{M}_{2}(R))$ is the image of $\mathcal{D}$ via the
homomorphism $\phi :\mathbb{Z}[x,y,z]\rightarrow R$ that maps $x$, $y$, and $z$ to $a$, $u$, and $t$ (respectively) and so $A$ is extendable. As $\mathcal{D}$
is equivalent to $\mathcal{E}$ (see Subsection \ref{S26}), it follows that $A $ is equivalent to the matrix $B:=\left[ 
\begin{array}{cc}
a & u \\ 
0 & (1-a)(1-uw)\end{array}\right] \in Um(\mathbb{M}_{2}(R))$, where $w:=2t-ut^{2}$. From Lemma \ref{L1}(3)
it follows that $B$ is extendable and hence its reduction modulo $Ra(1-a)(1-uw)=Ra(1-a)c^{2}$ is non-full (see Proposition \ref{PR3}(2)).
We will denote by $\bar{\ast}$ the reduction modulo $Ra(1-a)c^{2}$ of $\ast$, where $\ast $ is either $R$ or an element of $R$. We have a
product decomposition $\bar{u}=\bar{u_{1}}\bar{u}_2$ with $\bar{u}_{1}=u_{1}+Ra(1-a)c^{2}\in \bar{R}$ dividing $\bar{a}$ and $\bar{u}_{2}=u_{2}+Ra(1-a)c^{2}\in\bar R$ dividing $(1-\bar{a})(1-\bar{w}\bar{u})$ and hence
also $1-\bar{a}$. Thus there exists $(d,e,f)\in R^3$ such that $u_{1}u_{2}=u+a(1-a)c^{2}d$, $u_{1}$ divides $a+a(1-a)c^{2}e$ and $u_{2}$
divides $1-a+a(1-a)c^{2}f$. It follows that $\bigl(u_{1},c(1-a)\bigr),(u_{2},ca)\in
Um(R^{2})$, and therefore we have an identity $u+Rc=u_{1}u_{2}+Rc=\rho _{a,c}\bigl(u_{2}+Rac,u_{1}+R(1-a)c\bigr).$
\end{proof}

\begin{corollary}
\label{C14} Let $R$ be a B\'{e}zout domain. Then $R$ is an \textsl{EDD} iff for all $(a,u,t)\in R^3$ with $u\neq 0$ there exists $(s,l,z)\in R^3$ such that 
\begin{equation}\label{EQ11}
(1-su-la)^2+l-sul-l^2a-z(s+t-sut)=0.
\end{equation}
\end{corollary}

\begin{proof}
As in the proof of Corollary \ref{C13}, the Hermite ring $R$ is an \textsl{EDR} iff for all $(a,u,t)\in R^3$ the matrix $A=\left[ 
\begin{array}{cc}
a(1-tu) & u \\ 
0 & (1-a)(1-tu)\end{array}\right] \in Um(\mathbb{M}_{2}(R))$ is simply extendable and, due to Corollary \ref{C11}, iff there exists $(x,y,z,w)\in R^4$ such that 
$$1-a(1-tu)x-yu-w(1-a)(1-tu)=0=xw-yz.$$
If $u=0$, then $A$ is diagonal and hence simply extendable (see Example \ref{EX9}(3)).
Thus we can assume $u\neq 0$. The equation $1-a(1-tu)x-yu-w(1-a)(1-tu)=0$ can be rewritten as $(1-tu)(ax-aw+w)=1-yu$ and working modulo $Ru$ with $u\neq 0$ it follows that its general solution is $ax+w(1-a)=1-su$ and $y=s+t-sut$ with $s\in R$. As $R$ is an integral domain, the general solution of $ax+w(1-a)=1-su$ is $x=1-su+l(1-a)$ and $w=1-su-la$ and the equation $xw-yz=0$ becomes
$$[1-su+l(1-a)](1-su-la)-z(s+t-sut)=0,$$
from which the corollary follows.
\end{proof}

\subsection{Proof of Theorem \protect\ref{TH6} and applications}

\label{S46}

We will first prove Theorem \ref{TH6}. Let $\rho _{a,c}$ be as
above, the reduction product homomorphism. Similar to the proof of Corollary \ref{C13}, it suffices to prove that for all $(u,c)\in Um(R^2)$ and $a\in R$, we
have $u^{2}+Rc\in Im(\rho _{a,c})$. To check this we can assume that $c\neq
0 $. As in the mentioned proof, we can assume that there exists $t\in R$
such that $c=1-tu$. As we are assuming that $R$ is a $({WSU})_{2}$ ring, there
exists $M=\left[ 
\begin{array}{cc}
x & y \\ 
z & w\end{array}\right] \in {GL}_{2}(R)$ such that $M\left[ 
\begin{array}{cc}
ac & u \\ 
0 & (1-a)c\end{array}\right] $ is a symmetric matrix, i.e., the equation 
\begin{equation}
zac=xu+y(1-a)c  \label{EQ12}
\end{equation}holds. From this equation it follows that $c$ divides $xu$. Thus, as $(u,c)\in Um(R^{2})$, $c$ divides $x$. Let $s\in R$ be such that $x=cs$.

Let $\mu:=\det (M)=csw-yz\in U(R)$. Thus $(cs,yz)\in Um(R^{2})$. 

We first assume that $c\notin Z(R)$. Dividing Equation
(\ref{EQ12}) by $c$ it follows that 
\begin{equation}
za+y(a-1)=su.  \label{EQ13}
\end{equation}

We check that the assumption $(s,a)\notin Um(R^2)$ leads to a contradiction. This assumption implies that there exists $\mathfrak m\in {Max} R$ such that $Rs+Ra\subseteq\mathfrak m$. From this and Equation (\ref{EQ13}) it follows that $y(a-1)\in\mathfrak m$. As $a-1\not\in\mathfrak m$, it follows that $y\in\mathfrak m$, hence $(s,yz)\notin Um(R^2)$, a contradiction. 

Thus $(s,a)\in Um(R^2)$, i.e., $s+Ra\in U(R/Ra)$.

From Equation (\ref{EQ13}) it follows that the images of $(-y+Rc)(u+Rc)^{-1}\in U(R/Rc)$ and $s+Ra\in U(R/Ra)$ in $U(R/(Ra+Rc))$ are the same and 
hence there exists a unit of the quotient ring $R/(Ra\cap Rc)$ that maps to 
$(-y+Rc)(u+Rc)^{-1}\in U(R/Rc)$. 
As $(Ra\cap Rc)/Rac$ is a nilpotent ideal of $R/Rac$, the homomorphism $U(R/Rac)\rightarrow U(R/(Ra\cap Rc))$ is surjective and we conclude that $(-y+Rc)(u+Rc)^{-1}\in {Im}(\rho _{a,c})$.

Due to the symmetry of Equation (\ref{EQ13}) in $az$ and $y(a-1)$, similar
arguments give $s+R(1-a)\in U\bigl(R/R(1-a)\bigr)$ and $(z+Rc)(u+Rc)^{-1}\in {Im}(\rho _{a,c})$. Clearly $\mu+Rc\in {Im}(\rho _{a,c})$. As $\mu+Rc=(-y+Rc)(z+Rc)$, it follows that 
\begin{equation*}
(u+Rc)^{2}=[(-y+Rc)(u+Rc)^{-1}]^{-1}[(z+Rc)(u+Rc)^{-1}]^{-1}(\mu+Rc)\in {Im}(\rho _{a,c}).
\end{equation*}

We will use a trick with annihilators to show that an analogue of Equation (\ref{EQ13}) always holds (i.e., holds even if $c\in Z(R)$). As $(c,u),(c,yz)\in Um(R^2)$, it follows that $(c,uyz)\in Um(R^2)$. This implies
that the annihilator of $c$ in $R$ is contained in $Ruyz$. From this and the
identity $c[za+y(a-1)-su]=0$ it follows that there exists $q\in R$ such that 
$za+y(a-1)-su=quyz$. If $s_1:=s+qyz$, then $(s_1,yz)\in
Um(R^2)$ and the identity $za+y(a-1)=s_1u$ holds. Thus, using $s_1$ and the
last identity instead of $s$ and Equation (\ref{EQ13}), as above we
argue that $(u+Rc)^{2}\in {Im}(\rho _{a,c})$. Thus part (1) holds.

The first part of part (2) follows from part (1) and definitions. The second part of
part (2) follows from the first part of part (2) and Theorem \ref{TH5} (recall that $SU_2$ rings are Hermite rings). Thus
Theorem \ref{TH6} holds.

\begin{corollary}
\label{C15} For a B\'{e}zout domain the following statements are equivalent:

\medskip \textbf{(1)} The ring $R$ is an $(SU)_2$ (resp.\ $(SU^{\prime})_2$) ring.

\smallskip \textbf{(2)} For all $(a,b),(c,u)\in Um(R^2)$ with $abc\neq 0$,
if $(a^{\prime},b^{\prime})\in R^2$ such that $aa^{\prime}+bb^{\prime}=1$ is given, then there exists a triple $(s,l,w)\in R^3$ with the property that $csw+(a^{\prime}su-bl)(b^{\prime}su+al)\in U(R)$ (resp.\ $csw+(a^{\prime}su-bl)(b^{\prime}su+al)=1$).\end{corollary}

\begin{proof}
If $A_1,A_2\in\mathbb M_2(R)$ are equivalent, then there exists $M_1\in {GL}_2(R)$ such that $M_1A_1$ is symmetric iff there exists $M_2\in {GL}_2(R)$ such that $M_2A_2$ is symmetric. This is so, as for $M,N,M_1\in {GL}_2(R)$ such that $A_2=MA_1N$ and $M_1A_1$ is symmetric, by denoting $M_2:=N^TM_1M^{-1}\in {GL}_2(R)$, $M_2A_2=N^T(M_1A_1)N$ is symmetric. Also, if $M,N,M_1\in {SL}_2(R)$, then $M_2\in {SL}_2(R)$. 

Based on the previous paragraph, Subsection \ref{S26} and Example \ref{EX9}(3), the ring $R$ is an $(SU)_2$ (resp.\ $(SU^{\prime})_2$) ring iff for all $(a,b),(c,u)\in Um(R^2)$ with $abc\neq 0$ there exists $M=\left[ 
\begin{array}{cc}
x & y \\ 
z & w\end{array}\right]\in {GL}_2(R)$ (resp.\ $M=\left[ 
\begin{array}{cc}
x & y \\ 
z & w\end{array}\right]\in {SL}_2(R)$) such that the matrix $M\left[ 
\begin{array}{cc}
ac & u \\ 
0 & bc\end{array}\right] $ is symmetric, i.e., the equation $zac=xu+ybc$ holds. As in the
proof of Theorem \ref{TH6}, we write $x=cs$ with $s\in R$ and the identity $za=su+yb$ holds. The general solution of the equation $za=su+yb$
in the indeterminates $y,z$ is $z=a^{\prime}su-bl$ and $y=-b^{\prime}su-al$
with $l\in R$. Thus the corollary follows from the fact that $\det(M)=xw-yz=csw+(a^{\prime}su-bl)(b^{\prime}su+al)$.
\end{proof}

\begin{remark}\label{rem4}
\textbf{(1)} Referring to statement
(2) of Corollary \ref{C15} with $b=1-a$, we can take $a^{\prime}=1=b^{\prime}$, and hence one is searching for $(s,l,w)\in R^3$
such that $csw+(su+al-l)(su+al)\in U(R)$ (resp.\ and $csw+(su+al-l)(su+al)=1$,
 which for $c=1-tu$ with $t\in R$ is closely related  to Equation (\ref{EQ11}) via a substitution of the form $z:=-w$).

\smallskip \textbf{(2)} We recall from \cite{ZR}, Sect. 1 that $R$ is said to have 
\emph{square stable range 1}, and one writes $ssr(R)=1$, if for each $(a,b)\in Um(R^{2})$, there exists $r\in R$ such that $a^{2}+rb\in U(R)$. 
Clearly, if $sr(R)=1$, then $ssr(R)=1$. If $ssr(R)=1$, then for all $b\in R$, the cokernel of the reduction homomorphism 
$U(R)\rightarrow U(R/Rb)$ is an elementary abelian 2-group and hence $R$ is
a $WU_{2}$ ring. Moreover, if $(c,u)\in Um(R^{2})$, let $w\in R$ be such
that $u^{2}+cw\in U(R)$. Then $M:=\left[ 
\begin{array}{cc}
c & -u \\ 
u & w\end{array}\right] \in {GL}_{2}(R)$ and the product matrix $M\left[ 
\begin{array}{cc}
ac & u \\ 
0 & (1-a)c\end{array}\right] $ is symmetric. Hence, if $R$ is a Hermite ring with $ssr(R)=1$,
then each upper triangular matrix in $\mathbb{M}_{2}(R)$ has a companion
test matrix (see Subsection \ref{S26}) equivalent to a symmetric matrix; thus, from this and 
the proof of Corollary \ref{C13} it follows that $R$ is an \textsl{EDR} iff it is an ${SE}_2^{sym}$ ring.
\end{remark}

\begin{example}
\normalfont\label{EX12} Assume that $R$ is a Hermite ring and a $WU_{2}$
ring and let $A\in Um(\mathbb{M}_{2}(R))$.
Let $\bar{R}:=R/R\det (A)$ and let $\bar{A}\in Um(\mathbb{M}_{2}(\bar{R}))$
be the reduction of $A$ modulo $R\det(A)$. We know that $K_{\bar{A}}$ and $Q_{\bar{A}}$
are dual projective $\bar{R}$-modules of rank $1$ generated by two elements. We check that 
\begin{equation*}
\lbrack K_{\bar{A}}]=\lbrack Q_{\bar{A}}]\in {Pic}_{2}(\bar{R})[2]
\end{equation*}(equivalently, $K_{\bar{A}}^{2}$ or $Q_{\bar{A}}^{2}$ is a free $R$-module
of rank $2$). To simplify the notation, by replacing $(\bar{R},\bar{A})$
with $(R,A)$, we can assume that $\det (A)=0$. To check that $K_{A}^{\otimes
2}:=K_{A}\otimes _{R}K_{A}$ is isomorphic to $R$ we can assume that $R$ is
reduced. Based on Subsection \ref{S26}, we can assume that $A=\left[ 
\begin{array}{cc}
ac & u \\ 
0 & bc\end{array}\right] \in \mathbb{M}_{2}(R)$ with $a,b,c\in R$ such that $abc^{2}=0$ and $(a,b),(c,u)\in Um(R^{2})$; as $R$ is reduced, we have $abc=0$. Let $(s,t)\in
R^{2}$ be such that $sc+tu=1$. By Example \ref{EX10} we have $K_{A}=Rv_{1}+Rv_2$, where $v_{1}:=\left[ 
\begin{array}{c}
-u \\ 
ac\end{array}\right] $ and $v_2:=\left[ 
\begin{array}{c}
-bc \\ 
0\end{array}\right] $. As $(a,b)\in Um(R^{2})$, we have ${Spec}R={Spec}R_{a}\cup {Spec}R_{b}$. Moreover, $(K_{A})_{a}=R_{a}v_{1}$ and $(K_{A})_{b}=R_{b}(tv_{1}+\dfrac{s}{b}v_2)$. Over $R_{ab}$ we have $c=0$ and hence $v_{1}=u(tv_{1}+\dfrac{s}{b}v_2)$. Thus $(K_{A}^{\otimes
2})_{a}=R(v_{1}\otimes v_{1})$, $(K_{A}^{\otimes 2})_{b}=R\bigl((tv_{1}+\dfrac{s}{b}v_2)\otimes (tv_{1}+\dfrac{s}{b}v_2)\bigr)$, and over $R_{ab}$ we have an identity 
$$v_{1}\otimes v_{1}=u^{2}\bigl((tv_{1}+\dfrac{s}{b}v_2)\otimes (tv_{1}+\dfrac{s}{b} v_2)\bigr).$$ 
As $R$ is a $WU_{2}$ ring (see Theorem \ref{TH8}), there
exist units $u_{a}\in U(R/Rbc)$ and $u_{b}\in U(R/Rac)$ such that the product
of their images in $U(R/Rc)$ is $u^{2}$. As $R_{a}$ is a localization of $R/Rbc$ and $R_{b}$ is a localization of $R/Rac$, we will denote also by $u_{a}$ and $u_{b}$ their images in $R_{a}$ and $R_{b}$ (respectively). Thus $u_{a}^{-1}(v_{1}\otimes v_{1})$ and $u_{b}^{-1}\bigl((tv_{1}+\dfrac{s}{b}v_2)\otimes (tv_{1}+\dfrac{s}{b}v_2)\bigr)$ coincide over $R_{ab}$ and hence
there exists an element $v\in K_{A}^{\otimes 2}$ whose images in $(K_{A}^{\otimes 2})_{a}$ and $(K_{A}^{\otimes 2})_{b}$ are equal to $u_{a}^{-1}(v_{1}\otimes v_{1})$ and $u_{b}^{-1}\bigl((tv_{1}+\dfrac{s}{b}v_2)\otimes (tv_{1}+\dfrac{s}{b}v_2)\bigr)$ (respectively). Thus $Rv=K_{A}^{\otimes 2}$, hence $K_{A}^{\otimes 2}\cong R$.
\end{example}

\subsubsection{Proof of Corollary \protect\ref{C2}.} To prove Corollary \ref{C2}(1), we assume that $R$ is a Hermite ring such that ${Pic}_2(R/Ra)[2]$ is trivial for all nonzero $a\in R$. Each $U_2$ ring is a $WU_2$ ring. From Example \ref{EX12} it follows that if $R$ is a $WU_2$ ring, then for each $A\in Um(\mathbb M_2(R))$, the reduction $\bar{A}$ of $A$ modulo $R\det(A)$ is such that $K_{\bar A}\cong R/R\det(A)$, hence $A$ is extendable (see Theorem \ref{TH1} and Lemma \ref{L1}(1)). From the last two sentences and Theorem \ref{TH5} it follows that Corollary \ref{C2}(1) holds. Corollary \ref{C2}(2) follows from Corollary \ref{C2}(1) and Theorem \ref{TH6}, as each $(SU)_2$ ring is a $(WSU)_2$ ring and (see \cite{lor}, Prop. 3.1) a Hermite ring. Thus Corollary \ref{C2} holds. 

\subsection{Proof of Criterion \protect\ref{CR1}}\label{S47}

Based on Theorem \ref{TH8}, to prove part (1) it suffices to prove that
given $(a,b,c,e,f)\in R^{5}$ with $ac=b^{2}$ and $(a,c)\in Um(R^{2})$, we
have $(ae-bf,be-cf)\in Um(R^{2})$ if $ae^{2}-cf^{2}\in U(R)$, and the
converse holds if ${char}(R)=2$. The first part follows from the
fact that $ae^{2}-cf^{2}=e(ae-bf)+f(be-cf)$. For the second part we assume
that ${char}(R)=2$ and $(ae-bf,be-cf)\in Um(R^{2})$. To prove that $ae^{2}-cf^{2}\in U(R)$ it suffices to show that the assumption that there
exist $\mathfrak{m}\in {Max}R$ such that $ae^{2}-cf^{2}\in \mathfrak{m}$, leads to a contradiction. By replacing $R$ with $R/\mathfrak{m}$, we can
assume that $R$ is a field and we know that $a$ and $c$ are not both zero, $ae-bf$ and $be-cf$ are not both zero, $ac=b^{2}$ and $ae^{2}=cf^{2}$;
hence also $e$ and $f$ are not both $0$. If $acef=0$, then by the symmetry in 
$(a,c)$ and $(e,f)$ we can assume that $ae=0$; if $a=0$, then $b=0$, $c\neq 0
$, and hence $f=0$, thus $ae-bf$ and $be-cf$ are both zero, a contradiction,
and if $e=0$, then $f\neq 0$, hence $c=0$, and by the symmetry in the pair $(a,c)$, we similarly reach a contradiction. Thus we can assume that $acef\neq
0$, in which case $b^{2}=ac=e^{-2}c^2f^{2}=f^{-2}a^2e^2$; as ${char}(R)=2$ and $R$ is a field, it follows that $b=e^{-1}cf=f^{-1}ae$, hence 
$ae-bf$ and $be-cf$ are both zero, a contradiction. Thus part (1) holds.

To prove part (2), we note that, as $b\notin Z(R)$, for each $\mathfrak{m}\in {Max}R$, $b$
is a nonzero element of $R_{\mathfrak{m}}$. We recall that, as $R$ is a reduced Hermite ring, each $R_{\mathfrak{m}}$ is a valuation domain.

Let $d\in R
$ and $(a^{\prime},b^{\prime})\in Um(R^2)$ be such that $a=da^{\prime}$ 
and $b=db^{\prime}$. As $d$ divides $a$ and $(a,c)\in Um(R^2)$, 
we have $(d,c)\in Um(R^2)$. Thus also $(d^2,c)\in Um(R^2)$. From this, 
as $d^2$ divides $ac$, it follows that $d^2$ divides $a$, hence there exists 
$u\in R$ such that $a=d^2u$. As $b$ is a nonzero element of the valuation domain $R_{\mathfrak{m}}$, it
is easy to see that the image of $u$ in $R_{\mathfrak{m}}$ is a unit 
for all $\mathfrak{m}\in {Max}R$. Hence $u\in U(R)$.

By symmetry, there exist $g\in R$ and $v\in U(R)$
such that $c=g^2v$.

To complete the proof that the matrix $A$ is extendable, based on part (1)
it suffices to show that there exists $(e,f)\in R^{2}$ such that $ae^{2}-cf^{2}=ae^{2}+cf^{2}\in U(R)$. Replacing $A$ by $u^{-1}A$ (see
Lemma \ref{L1}(3)), we can assume that $u=1$. Thus $a=d^{2}$. From this and the
identities $ac=b^{2}$ and $c=g^{2}v$, a simple argument shows that 
there exists $w\in R$ such that $v=w^2$. By replacing $(v,g)$ with $(1,gw)$, 
we can assume that $v=1$ and $c=g^{2}$. As $(a,c)=(d^{2},g^{2})\in Um(R^{2})$, it follows that $(d,g)\in Um(R^{2})$,
hence there exists $(e,f)\in Um(R^{2})$ such that $de+gf=1$, and so $ae^{2}+cf^{2}=1\in U(R)$. Thus part (2) holds and so Criterion \ref{CR1}
holds.

\subsection{Proof of Criterion \protect\ref{CR2}}

\label{S48}

As $R$ is an $(SU)_{2}$ ring, it is a Hermite ring, and it is an $E_{2}$ ring iff it is an $E_{2}^{sym}$ ring. From this and Theorem \ref{TH5} it follows that $R$ is an 
\textsl{EDR} iff it is an $E_{2}^{sym}$ ring. Let $A=\left[ 
\begin{array}{cc}
a & b \\ 
b & c\end{array}\right] \in Um(\mathbb{M}_{2})$. Based on Lemma \ref{L1}(1) and Criterion \ref{CR1}(1), $A$ is extendable if there exists $(e,f)\in R^{2}$ such that $ae^{2}-cf^{2}+R(ac-b^{2})\in U(R/[R(ac-b^{2})])$, equivalently, if $(ae^{2}-cf^{2},ac-b^{2})\in Um(R^{2})$, and the converse holds if ${char}(R)=2$; hence Criterion \ref{CR2} holds.

\subsection{Proof of Criterion \protect\ref{CR3}}

\label{S49}

As $R$ is a Hermite ring, it is an \textsl{EDR} iff it is an ${SE}^{\triangle}_{2}$ ring and thus (see Corollary \ref{C10}) iff for
all $a,b,s\in R$ the matrix $A=\left[ 
\begin{array}{cc}
a & b \\ 
0 & 1-a-bs\end{array}\right] $ is simply extendable. As $sr)R)\le 2$, based on the last part of Corollary \ref{C7}, we
can replace simply extendable by extendable. As each B\'{e}zout domain 
is a Schreier domain, based on Corollary \ref{C12}, we can replace extendable 
by statement \circled{\textup{9}} holding for $\upsilon :=(a,b,0,1-a-bs)$, 
i.e., Equation
(\ref{EQ6}) gives for each such $\upsilon$ an equation 
\begin{equation}
\lbrack 1-(1-a-bs)w-by](1-ax)=ay[(1-a-bs)z+bx]  \label{EQ14}
\end{equation}in the indeterminates $x$, $y$, $z$ and $w$ which has a solution in $R^{4}$.
To solve this equation, we note that, as $(a,1-ax)\in Um(R^{2})$, $a$
divides $1-(1-a-bs)w-by$.

Thus, as $R$ is an integral domain, Equation (\ref{EQ14}) has a solution in 
$R^{4}$ iff the system of two equations 
\begin{equation}
1-(1-a-bs)w-by=ta\;\;{and}\;\;t(1-ax)=y[(1-a-bs)z+bx]  \label{EQ15}
\end{equation}in the indeterminates $t,x$, $y$, $z$ and $w$ has a solution in $R^{5}$
(this holds even if $a=0$!).

The first equation $1-(1-a-bs)w-by=ta$, rewritten as $w=1+a(w-t)+b(sw-y)$, has the general solution given by $w=1-aq-br$, $t=q+w=1+q-aq-br$ and $y=r+sw=r+s(1-aq-br)$, where $(q,r)\in R^2$ can be arbitrary.

The second equation can be rewritten as $t=x(at+by)+yz(1-a-bs)$ and has a
solution iff $t\in R(at+by)+Ry(1-a-bs)$. As $1=(1-a-bs)w+(at+by)$, we have $t\in R(at+by)+Ry(1-a-bs)$ iff $t\in R(at+by)+Ry=Rat+Ry$.

From the last two paragraphs it follows that System (\ref{EQ15}) has a
solution in $R^{5}$ iff there exists $(q,r)\in R^2$ as mentioned in
statement (2). Thus $(1)\Leftrightarrow (2)$.

If $(1-a,b)\in Um(R^2)$, then we can choose $(q,r)\in Um(R^2)$ such that $t=1+q(1-a)+rb=0$, hence $t=0\in Rat+Ry=Ry$, and for $e:=1$ and $f:=0$, we
have $(a,e),(be+af,1-bs-a)\in Um(R^2)$. Thus to prove that $(2)\Leftrightarrow (3)$ we can assume that $(1-a,b)\notin Um(R^2)$, and hence always $t\neq 0$. 

For $t_{1}:=\gcd (r-qs,t)$ we write $t=t_{1}t_{2}$ and $r-qs=\alpha t_{1}$ with $(\alpha ,t_{2})\in
Um(R^{2})$; thus $y=r+s(t-q)=st+t_{1}\alpha $. We would like to find 
$(q,r)\in R^{2}$
such that there exists $(g,h)\in R^{2}$ with the property that $gy+hat=t$,
equivalently, $tgs+t_{1}g\alpha +tah=t$, and thus we must have $g\in
Rt_{2}$. Writing $g=t_{2}\beta $ with $\beta \in R$, we want to find $(q,r)\in
R^{2}$ such that there exists $(\beta ,h)\in R^{2}$ satisfying $1=\beta
(st_{2}+\alpha)+ha$, equivalently, such that $(a,st_{2}+\alpha )\in
Um(R^{2})$. Replacing $r=qs+\alpha t_{1}$ in $t$, it follows that $t=1+q-qbs-t_{1}\alpha b-aq$, hence there exists $\gamma \in R$ such that $1+q(1-bs-a)=t_{1}\gamma $; we have $t_{2}=\gamma -\alpha b$ and thus $e:=st_{2}+\alpha =s\gamma +(1-bs)\alpha $. As 
$(\alpha ,t_{2})\in Um(R^{2})$, it follows that $(\alpha ,\gamma )\in Um(R^{2})$. There exists a unique $f\in R$ such that $\alpha =e-fs$ and $\gamma =be+f(1-bs)$; in fact we have $f=t_{2}=\gamma -\alpha b$ and thus $R=R\alpha +R\gamma =Re+Rf$, i.e., $(e,f)\in Um(R^{2})$. There exists $q\in R$ such that $be+f(1-bs)$ divides $1+q(1-bs-a)$ iff $\bigl(be+f(1-bs),1-bs-a\bigr)\in Um(R^{2})$ and hence iff $(be+af,1-bs-a)\in Um(R^{2})$. Thus $(2)\Leftrightarrow (3)$. 
So Criterion \ref{CR3} holds.

\section{Explicit computations for integral domains}

\label{S5}

Let $A:=\left[ 
\begin{array}{cc}
a & b \\ 
c & d\end{array}\right] \in Um(\mathbb{M}_{2}(R))$ be such that we can write $a=ga^{\prime }$, $c=gc^{\prime }$, $b=hb^{\prime }$, $d=hd^{\prime }$ with $a^{\prime
},b^{\prime },c^{\prime },d^{\prime },g,h\in R$ and $(a^{\prime
},c^{\prime }),(b^{\prime },d^{\prime })\in Um(R^{2})$. We have $(g,h)\in
Um(R^{2})$. Let $e^{\prime },f^{\prime }\in R$ be such that $a^{\prime
}e^{\prime }+c^{\prime }f^{\prime }=1$. Let $l:=b^{\prime }c^{\prime
}-a^{\prime }d^{\prime }\in R$ and $m:=b^{\prime }e^{\prime }+d^{\prime
}f^{\prime }\in R $; note that $\det (A)=-ghl$. As $[\begin{array}{cc}
l & m\end{array}]=[\begin{array}{cc}
b^{\prime } & d^{\prime }\end{array}]\left[ 
\begin{array}{cc}
c^{\prime } & e^{\prime } \\ 
-a^{\prime } & f^{\prime }\end{array}\right] $, the matrix $\left[ 
\begin{array}{cc}
c^{\prime } & e^{\prime } \\ 
-a^{\prime } & f^{\prime }\end{array}\right] $ has determinant 1, and 
$({b}^{\prime }{,d}^{\prime })\in Um(R^2)$, it follows that $(l,m)\in Um(R^2)$. 

Let $(e,f,w)\in R^3$ be such that $ae+cf=gw$. If $R$ is an
integral domain and $g\neq 0$, there exists $v\in R$ such that $(e,f)=(we^{\prime
}+c^{\prime }v,wf^{\prime }-a^{\prime }v)$, hence
\begin{equation*}
be+df=h(b^{\prime }e+d^{\prime }f)=h(b^{\prime }we^{\prime }+b^{\prime
}c^{\prime }v+d^{\prime }wf^{\prime }-a^{\prime }d^{\prime }v)=hw(b^{\prime}e^{\prime }+d^{\prime }f^{\prime })+hv(b^{\prime }c^{\prime }-a^{\prime}d^{\prime })
\end{equation*}is equal to $h(wm+vl)$. Thus $(ae+cf,be+df)=\bigl(gw,h(wm+vl)\bigr)$.

\begin{proposition}
\label{PR5} Let $R$ be an integral domain. Let $A=\left[ 
\begin{array}{cc}
a & b \\ 
c & d\end{array}\right] \in Um(\mathbb{M}_{2}(R))$ be such that the above notation $g,h,a^{\prime },b^{\prime },c^{\prime },d^{\prime }$ applies and let $(e^{\prime },f^{\prime },l,m)\in R^{4}$ be obtained as above. 
We assume $g\neq 0$.

\medskip \textbf{(1)} The matrix $A$ is simply extendable iff
there exists $(w,v)\in R^2$ such that $(g,wm+vl),(w,hvl)\in Um(R^{2})$, in which
case a simple extension of $A$ is 
\begin{equation*}
\left[ 
\begin{array}{ccc}
a & b & wf^{\prime }-a^{\prime }v \\ 
c & d & -we^{\prime }-c^{\prime }v \\ 
-t & s & 0\end{array}\right]
\end{equation*}where $s,t\in R$ are such that $gws+h(wm+vl)t=1$.

\smallskip \textbf{(2)} If the intersection $\{(g,l),(g,m),(h,l),(h,m)\}\cap
Um(R^{2})$ is nonempty (e.g., if $hlm=0$), then $A$ is simply extendable
and $w,v\in R$ are given by formulas.
\end{proposition}

\begin{proof}
There exists $(e,f)\in R^2$ such that $(ae+cf,be+df)\in Um(R^{2})$ iff
there exists $(w,v)\in R^2$ such that $\bigl(gw,h(wm+vl)\bigr)\in Um(R^{2})$ (see above).
We have $\bigl(gw,h(wm+vl)\bigr)\in Um(R^{2})$ iff $\bigl(g,h(wm+vl)\bigr),\bigl(w,h(wm+vl)\bigr)\in Um(R^{2})$. As $(g,h)\in Um(R^{2})$, we have 
$\bigl(g,h(wm+vl)\bigr)\in Um(R^{2})$ iff $(g,wm+vl)\in Um(R^{2})$;
moreover, $Rw+Rh(wm+vl)=Rw+Rhvl$. Thus $\bigl(gw,h(wm+vl)\bigr)\in Um(R^{2})$ 
iff $(g,wm+vl),(w,hvl)\in Um(R^{2})$. Based on the `iff'
statements of this paragraph and Theorem \ref{TH8}, it follows that part
(1) holds.

To check part (2), we first notice that if $hlm=0$, say $h=0$, then $g=1$
and hence $(g,l),(g,m)\in Um(R^{2})$. Based on part (1) it suffices to show
that in all four possible cases, we can choose $(w,v)\in R^2$ such that $(g,wm+vl),(w,hvl)\in Um(R^{2})$.

If $(g,l)\in Um(R^{2})$, for $(w,v):=(g,1)$ we have $Rg+R(wm+vl)=Rg+Rl=R$.
As $(g,h),(g,l)\in Um(R^{2})$, also $(w,hvl)=(g,hl)\in Um(R^{2})$.

If $(g,m)\in Um(R^{2})$, for $(w,v):=(1,0)$ we have $(g,wm+vl)=(g,m)\in
Um(R^{2})$ and $(w,hvl)=(1,0)\in Um(R^{2})$.

If $(h,m)\in Um(R^{2})$, then $(hl,m)\in Um(R^{2})$ and there exists $(w,v^{\prime })\in R^2$ such that $wm+hv^{\prime }l=1$; so 
$Rw+Rhv^{\prime}l=R $. For $v:=hv^{\prime }$ we have $wm+vl=1$ and 
so $(g,wm+vl)\in
Um(R^{2})$ and $(w,hvl)=(w,h^{2}v^{\prime }l)\in Um(R^{2})$ as $(w,hv^{\prime }l)\in Um(R^{2})$.

If $(h,l)\in Um(R^{2})$, let $(p,q)\in R^2$ be such that $1=pl+qm$. 
For $w:=hq+l$
and $v:=hp-m$ we compute $wm+vl=h(pl+qm)+ml-ml=h$, so $(g,wm+vl)=(g,h)\in Um(R^{2})$. Also, $Rw+Rh=R(hq+l)+Rh=Rl+Rh=R$ 
and $Rw+Rvl$ contains $wm+vl=h$ and hence it contains $Rw+Rh=R$. 
Thus $(w,vl)\in
Um(R^{2}) $. As $(w,h),(w,vl)\in Um(R^{2})$ it follows that $(w,hvl)\in
Um(R^{2})$.
\end{proof}

\begin{remark}\label{rem5}
We include a third proof of Corollary \ref{C3}(2) for B\'{e}zout domains. Assume $R$ is a B\'{e}zout domain with $asr(R)=1$. Based on the equivalence $\circled{\textup{1}}\Leftrightarrow \circled{\textup{2}}$, it suffices to
show that each matrix $A=\left[ 
\begin{array}{cc}
a & b \\ 
c & d\end{array}\right] \in Um(\mathbb{M}_{2}(R))$ is simply extendable. As $R$ is a 
Hermite domain, the notation of this section applies. As $\gcd (g,h)=1$, 
by the symmetry between the pairs $(a,c)$ and $(b,d)$, 
we can assume that $g\notin J(R)$. As $(m,l,g)\in
J_{3}(R)$ and $asr(R)=1$, there exists $v\in R$ such that $\gcd (m+lv,g)=1$.
We take $w:=1$. Hence $\gcd (w,hvl)=1$ and $\gcd (g,wm+lv)=\gcd (g,m+lv)=1$.
From Proposition \ref{PR5}(1) it follows that $A $ is simply extendable.
\end{remark}

\subsection{Numerical examples}

\label{S51}

For statements $\circled{\textup{2}}$ and $\circled{\textup{5}}$ we provide sample 
matrices $A=\left[ 
\begin{array}{cc}
a & b \\ 
c & d\end{array}\right]\in Um(\mathbb M_2(\mathbb Z))$ whose simple extensions $A^+_{(e,f,s,t)}=\left[ 
\begin{array}{ccc}
a & b & f \\ 
c & d & -e \\ 
-t & s & 0\end{array}\right]$ are parameterized by the set (see Equation (\ref{EQ8}))
$$\gamma_A:=\{(e,f,s,t)\in \mathbb{Z}^{4}|a(es)+b(et)+c(fs)+d(ft)=1\}$$
and were (initially) exemplified using a code written for $R=\mathbb Z$ by the second author. We have $\nu_{A^+_{(e,f,s,t)}}=\chi_{A^+}^{\prime}(0)=\det(A)+es+ft$.

(1) If $a=0$ and $d=1+b+c$, then we can take $A^{+}=\left[ 
\begin{array}{ccc}
0 & b & -1 \\ 
c & 1+b+c & -1 \\ 
1 & 1 & 0\end{array}\right].$

Concretely, suppose $(b,c)=(3,2)$; then $d=6$, $\det(A)=-6$, $\gamma_A=\{(e,f,s,t)\in \mathbb{Z}^{4}|3et+2fs+6ft=1\}$ and $\nu_{A}=\{-6+es+ft|(e,f,s,t)\in\gamma_A\}$.

\smallskip To solve the equation $3et+2fs+6ft=1$, let $w:=et+2ft$. We get 
$2fs+3w=1$ with general solution $fs=-1+3k$, $w=1-2k$, where 
$k\in\mathbb{Z}$. The general solution of the equation $et+2ft=1-2k$ is 
$et=2k-1-2l$ and $ft=1-2k+l$, where $l\in\mathbb{Z}$. Let $m:=l-2k+1$. 
It follows that 
\begin{equation*}
ft=m,\; et=1-2k-2m,\; fs=-1+3k,
\end{equation*}
and the only constraint is that $ft=m$ divides $etfs=(3k-1)(1-2k-2m)$, i.e., divides $(3k-1)(2k-1)$. As $es+ft=m+\dfrac{-6k^2+5k-1}{m}-6k+2$, it follows that
\begin{equation*}
\nu_A=\{-4+m-6k+\dfrac{-6k^2+5k-1}{m}|(m,k)\in\mathbb{Z}^2,\; m\;\textup{divides}\; -6k^2+5k-1\}.
\end{equation*}

(2) If $c=0$ and $d=1-a+b$, then we can take $A^{+}=\left[ 
\begin{array}{ccc}
a & b & -1 \\ 
0 & 1-a+b & -1 \\ 
1 & 1 & 0\end{array}\right] $ (cf. Corollary \ref{C9}(2): for simple extensions of upper triangular matrices with nonzero $(1,1)$ entries over rings $R$ with $fsr(R)=1.5$, we can choose $e=1$).

Concretely, suppose $(a,b)=(6,-10)$; hence $d=15$ and $\det(A)=-90$. Thus $\gamma_A=\{(e,f,s,t)\in \mathbb{Z}^{4}|6es-10et-15ft=1\}$.

To solve the equation $6es-10et-15ft=1$, let $w:=2et+3ft$. We get $6es-5w
=1$ with general solution $es=1+5k$, $w =1+6k$, where $k\in\mathbb{Z}$. Then 
$2et+3ft=1+6k$ has the general solution $et=-1-6k+3l$, $ft=1+6k-2l$, where $l\in\mathbb{Z}$. Let $m:=l-2k$. It follows that 
\begin{equation*}
es=1+5k,\; et=-1+3m,\; ft=1+2k-2m
\end{equation*}
are subject to the only constraint that $et=-1+3m$ divides $esft=(1+5k)(1+2k-2m)$, i.e., it divides $(1+5k)(2k+m)$. As $es+ft=2+7k-2m$, it follows that 
\begin{equation*}
\nu_A=\{-88+7k-2m|(m,k)\in\mathbb{Z}^2,\; -1+3m\;\textup{divides}\;
(1+5k)(2+2k+m)\}.
\end{equation*}
For $m=0$ (resp.\ $m=1$) it follows that $\nu_A\supseteq 3+7\mathbb Z$ (resp.\ $\nu_A\supseteq 1+14\mathbb{Z}$.)

If we replace $\mathbb{Z}$ by $R_0:=\mathbb{Z}[\dfrac{1}{21}]$, then $A_0:=A\in\mathbb{M}_2(R_0)$ is similar to $Diag(6,-15)$ and one checks that $\nu_{A_0}=-90+\dfrac{1}{6}+\dfrac{7}{2}+2R_0=-86-\dfrac{1}{3}+2R_0$.

(3) If $A=$ $\left[ 
\begin{array}{cc}
15 & 6 \\ 
10 & 14\end{array}\right] $, we can take $A^{+}=\left[ 
\begin{array}{ccc}
15 & 6 & -2 \\ 
10 & 14 & 1 \\ 
-1 & -1 & 0\end{array}\right] $; indeed we have $\det(A^+)=1\cdot 15-1\cdot 6+2\cdot 10-2\cdot 14=1$. The
entries of $A$ use double products of the primes $2,3,5,7$. We have 
$\gamma_A=\{(e,f,s,t)\in \mathbb{Z}^{4}|15es+6et+10fs+14ft=1\}$, $\det(A)=150$, $\nu_{A^{+}}=149$, and $\nu_{A}=\{150+es+ft|(e,f,s,t)\in\gamma_A\}$.

To solve the equation $15es+6et+10fs+14ft=1$, let $x:=5es+2et$ and $y:=5fs+7ft$, so we get $3x+2y =1$ with general solution $x=1+2k$, $y =-1-3k$, where $k\in\mathbb{Z}$. Then $5es+2et=1+2k$ has the general solution $es=1+2k+2l$, $et=-2(1+2k)-5l$, where $l\in\mathbb{Z}$, and $5fs+7ft=-1-3k$
has the general solution $fs=3(-1-3k)+7r$, $ft=-2(-1-3k)-5r$, where $r\in\mathbb{Z}$. Let $o:=k+l$, so $et=-2+4l-4o-5l=:q$. Thus $es=1+2o$, $et=q$, 
$l=-2-q-4o$, $k=2+q+5o$, therefore $fs=-3-18-9q-45o+7r$ and 
$ft=2+12+6q+30o-5r$. Let $p:=r-6o-q-3$. Thus $ft=-5p+q-1$ and 
$fs=-3o-2q+7p$. As $(es)(ft)=(et)(fs)$, we have an identity $
(1+2o)(q-1-5p)=q(7p-3o-2q)$,
which can be rewritten as 
\begin{equation*}
o(-2-10p+5q)=1+5p-q-2q^2+7pq,
\end{equation*}
and which for $q=2p$ becomes $-2o=6p^2+3p+1$, requiring $p$ to be odd. For $q=2p$, $es+ft=2o+q-5p$ becomes $2o-3p=-6p^2-6p-1$. Thus
\begin{equation*}
\nu_A=\{150+2o+q-5p|(o,q,p)\in\mathbb{Z}^3,\;o(-2-10p+5q)=1+5p-q-2q^2+7pq\}
\end{equation*}
contains the set $\{150-6p^2-6p-1|p-1\in 2\mathbb{Z}\}$.

(4) If $A=\left[ 
\begin{array}{cc}
30 & 42 \\ 
70 & 105\end{array}\right] $, we can take $A^{+}=\left[ 
\begin{array}{ccc}
30 & 42 & 1 \\ 
70 & 105 & 3 \\ 
1 & 1 & 0\end{array}\right] $; indeed we have $\det(A^+)=-3\cdot 30+3\cdot 42+1\cdot 70-1\cdot 105=1$. The
entries of $A$ use triple products of the primes $2,3,5,7$, hence they are not pairwise coprime.

\section{Universal rings, retractions and smooth homomorphisms}

\label{S6}

For $n\in \mathbb{N}$, let $\mathcal{P}_{n}:=\mathbb{Z}[x_{1},y_{1},\ldots
,x_{n},y_{n}]$ and consider the first type of universal `unimodular'
rings 
\begin{equation*}
\mathcal{U}_{n}:=\mathcal{P}_{n}/(\eta _{n})
\end{equation*}where $\eta _{n}:=1-\left( \sum_{i=1}^{n}x_{i}y_{i}\right)$. 
The image of an element $x$ of a ring via a quotient homomorphism 
will be denoted by $\bar {x}$; e.g., for $i\in \{1,\ldots ,n\}$ we 
denote by $\bar{x}_{i}$ and $\bar{y}_{i}$ the images of $x_{i}$ and $y_{i}$ 
(respectively) in quotients of $\mathcal{U}_{n}$.

\begin{example}
\normalfont\label{EX13} As 
$\sum_{i=1}^{4}\bar{x}_{i}\bar{y}_{i}=1$, $X:=\left[ 
\begin{array}{cc}
\bar{x}_{1} & \bar{x}_{2} \\ 
\bar{x}_{3} & \bar{x}_{4}\end{array}\right] \in Um(\mathbb{M}_{2}(\mathcal{U}_{4}))$. The `universal' unimodular $2\times 2$ matrix $X$ is not extendable (if it
were, then every $R$ would be an ${E}_2$ ring, contradicting
Example \ref{EX11}).
\end{example}

Denoting $\vartheta
_{4}:=1-x_{1}y_{1}y_{4}-x_{2}y_{1}y_{3}-x_{3}y_{2}y_{4}-x_{4}y_{2}y_{3}\in\mathcal P_4$, we also consider four extra types of universal rings (cf. Subsection \ref{S31})
\begin{equation*}
\mathcal E_2:=\mathcal{P}_4[w]/\bigl(\vartheta _{4}-w(x_{1}x_{4}-x_{2}x_{3})\bigr),
\end{equation*}\begin{equation*}
\mathcal{SE}_2:=\mathcal{P}_4/(\vartheta _{4})=\mathcal E_2/(\bar{w}),
\end{equation*}\begin{equation*}
\mathcal{WZ}_{2}:=\mathcal{P}_4/\bigl(\eta
_{4}+(x_{1}x_{4}-x_{2}x_{3})(y_{1}y_{4}-y_{2}y_{3})\bigr),
\end{equation*}
\begin{equation*}
\mathcal{Z}_2:=\mathcal{P}_{4}/(\eta_4,y_1y_4-y_2y_3)=\mathcal{U}_{4}/(\bar{y}_{1}\bar{y}_{4}-\bar{y}_{2}\bar{y}_{3})=\mathcal{WZ}_2/(\bar{y}_1\bar{y}_4-\bar{y}_2\bar{y}_3).
\end{equation*}

\begin{example}
\normalfont\label{EX14} The matrix $Y:=\left[ 
\begin{array}{cc}
\bar{x}_{1} & \bar{x}_{2} \\ 
\bar{x}_{3} & \bar{x}_{4}\end{array}\right] \in \mathbb{M}_{2}(\mathcal E_2)$ is the `universal' unimodular $2\times 2$ matrix that is extendable, the `standard' extension of it being 
$Y^+:=\left[ 
\begin{array}{ccc}
\bar{x}_{1} & \bar{x}_{2} & \bar{y}_{2}\\ 
\bar{x}_{3} & \bar{x}_{4} & -\bar{y}_{1}\\
-\bar{y}_{3} & \bar{y}_{4} & \bar{w}
\end{array}\right] \in {SL}_{3}(\mathcal E_2)$. To check that $Y$ is not simply
extendable it suffices to show that there exist $2\times 2$ matrices that are extendable but are not simply extendable. Let 
$R$ be an integral domain such that $sr(R)=3$ and each 
projective $R$-module $P$ with $P\oplus R\cong R^3$ is free. E.g., 
if $\kappa$ is a subfield of $\mathbb{R}$, then $sr(\kappa[x_{1},x_{2}])=3$ (see \cite{vas}, Thm. 8) and Seshadri proved that all 
finitely generated projective modules over it are free (see \cite{ses}, Thm.; 
see also \cite{lan}, Ch. XXI, Sect. 3, Thm. 3.5 for the Quillen--Suslin Theorem). Let 
$(a_{1},a_{2},b)\in Um(R^{3})$ be not reducible; thus $b\neq
0$. We have projective resolutions $0\rightarrow Rb\rightarrow R\rightarrow
R/Rb\rightarrow 0$ and $0\rightarrow P\xrightarrow{g}R^{2}\xrightarrow{f} R/Rb\rightarrow 0$, where the $R$-linear
map $f$ maps the elements of the standard basis of $R^{2}$ to $a_{1}+Rb$ and $a_{2}+Rb$, $P=\Ker(f)$ and the $R$-linear map $g$ is the inclusion. It follows that $P$ is of the type
mentioned (see Subsection \ref{S24}, paragraph after Fact \ref{F1}); thus 
we identify $P=R^{2}$. Let $A\in \mathbb{M}_{2}(R)$ be such that $L_A=g:R^{2}=P\rightarrow R^{2}$; 
we have $A\in Um(\mathbb M_2(R))$ and $R\det(A)=Rb$. The $R/R\det(A)$-module $E_A$ is isomorphic to 
$R/R\det(A)$ (see Subsection \ref{S25} for notation). From this, Subsection 
\ref{S25} and Theorem \ref{TH1} it follows that the reduction of $A$ modulo 
$R\det(A)$ is simply extendable. Thus $A$ is extendable (see Lemma \ref{L1}(1)). 
But $A$ is not simply extendable: if it were, then it would be equivalent to ${Diag}(1,\det(A))$ (see Theorem \ref{TH8}) and it would follow from Subsection 
\ref{S24} that $(a_{1},a_{2},b)\in Um(R^{3})$ is reducible, a contradiction.
If $R$ is a $\Pi_2$ domain (e.g., if $R=\kappa[x_1,x_2]$ with $\kappa$ as above), then 
statement \circled{\textup{9}} holds for $A$ (see Proposition \ref{PR3}(3)) 
but statement \circled{\textup{7}} does not hold for $A$ (see Corollary \ref{C11}).
\end{example}

\begin{theorem}
\label{TH9} The following properties hold:

\smallskip \textbf{(1)} The ring $\mathcal{U}_{n}$ is a regular UFD of
dimension $2n$ and the ring $\mathcal{Z}_{2}$ is a regular domain of
dimension $7$ which is not a UFD.

\smallskip \textbf{(2)} The rings $\mathcal{E}_2$, 
$(\mathcal{WZ}_2)_{\bigl(1-(\bar{x}_1\bar{x}_4-\bar{x}_2\bar{x}_3)(\bar{y}_1\bar{y}_4-\bar{y}_2\bar{y}_3)\bigr)}$ and $\mathcal{SE}_2$ 
are regular of dimensions $9$, $8$ and $8$ (respectively).
\end{theorem}

\begin{proof}
(1) As $(\bar{x}_{1},\ldots ,\bar{x}_{n})\in Um(\mathcal{U}_{n}^{n})$, for
the regularity part it suffices to show that for $f\in \{x_{1},\ldots
,x_{n}\}$, $(\mathcal{U}_{n})_{\bar{f}}$ and, when $n=4$, $(\mathcal{Z}_2)_{\bar{f}}$, are regular of dimensions $2n$ and $7$ (respectively). For $n=4$
this follows from Theorem \ref{TH7}(1) and (2) (respectively), applied to $R $ being the regular ring $\mathbb{Z}[x_{1},x_{2},x_{3},x_{4}][f^{-1}]$ of
dimension $5$ and $\upsilon $ being $(x_{1},x_{2},x_{3},x_{4})$. The
general case $n\in \mathbb{N}$ is argued as in the case $n=4$.

As $\mathcal{P}_n$ is a UFD, no prime integer divides $\eta _{n}$,  and all partial 
degrees of $\eta_{n}$ are $1$, it follows that $\eta _{n}$ is a prime element of $\mathcal{U}_{n}$, i.e.,  $\mathcal{U}_{n}=\mathcal{P}_{n}/(\eta _{n})$ is an integral domain. As $\mathcal{U}_{1}\cong \mathbb{Z}[x_{1}][x_{1}^{-1}]$ is a UFD, to prove that $\mathcal{U}_{n}$ is a UFD we
can assume $n\geq 2$. By the same arguments it follows that 
\begin{equation*}
\mathcal{U}_{n}/(\bar{x}_{1})\cong \mathbb{Z}[y_{1},x_{2},y_{2},\ldots
,x_{n},y_{n}]/(1-\sum_{i=2}^{n}x_{i}y_{i})\cong \mathcal{U}_{n-1}[z]
\end{equation*}is an integral domain (here $z$ is an indeterminate), i.e., $\bar{x}_{1}$ is
a prime element of $\mathcal{U}_{n}$. The localization $(\mathcal{U}_{n})_{\bar{x}_{1}}$ is isomorphic to $\mathbb{Z}[x_{1},x_{2},y_{2},\ldots
,x_{n},y_{n}][x_{1}^{-1}]$ and hence is a UFD. Based on the last two
sentences and Nagata's criterion for a UFD (see \cite{mat}, Thm. 20.2), we
conclude that $\mathcal{U}_{n}$ is a UFD if $n\geq 2$.

As $\mathcal{Z}_2/(\bar{y}_1-1,\bar{y_2}-1,\bar{y_3}-1,\bar{y}_4-1)\simeq\mathbb Z[x_1,x_2,x_3,x_4]/(1-x_1-x_2-x_3-x_4)\simeq\mathbb Z[x_1,x_2,x_3]$ is an integral domain and $\mathcal Z_2$ is regular, to check that $\mathcal{Z}_2$ is an integral
domain it suffices to check that for each $f\in\{y_1,y_2,y_3,y_4\}$, $(\mathcal{Z}_2)_{\bar f}$ is an integral domain, which is true as it is the
localization of a polynomial algebra over $\mathbb{Z}$ in $6$
indeterminates; e.g., if $f=y_1$, then $(\mathcal{Z}_2)_{\bar{f}}\cong\mathbb{Z}[\bar y_1,\bar x_2,\bar y_2,\bar x_3,\bar y_3,\bar x_4][\bar
y_1^{-1}]$.

Finally, we show that the assumption that $\mathcal{Z}_{2}$ is a UFD leads
to a contradiction. As $\mathcal{Z}_{2}$ is noetherian, it follows that $\mathcal{Z}_{2}$ is a Schreier domain (see \cite{coh}, Thm.
2.3). Thus for all $(a,b,c,d)\in Um(\mathcal{Z}_{2}^{4})$ with $ad=bc$, statements $\circled{\textup{1}}$ to $\circled{\textup{10}}$ hold (see Corollary \ref{C12}), and
from this, due to the universal construction of $\mathcal{Z}_{2}$, it
follows that for each ring $R$ and every $(a,b,c,d)\in Um(R^{4})$ with $ad=bc $, the mentioned ten statements hold. But, referring to Example \ref{EX11}, 
for $(2k+1,1-\theta ,1+\theta ,2)\in Um(\mathbb{Z}[\theta ])$, statement $\circled{\textup{2}}$ fails, a contradiction.

(2) Similar to the case of $\mathcal{Z}_2$ in part (1), the regularity of
part (2) follows from Theorem \ref{TH7}(1) to (4) applied to $U_{\upsilon} $ equal to $(\mathcal{U}_4)_{\bar{f}}$ for some $f\in \{\bar{x}_{1},\ldots,\bar{x}_{4}\} $.
\end{proof}

\begin{remark}\label{rem8}
There exists a universal algebra
`behind' (i.e., controlling via extensions of homomorphisms) every stable
range type of condition. For example, with $p,s,t,u$ as indeterminates, for
Kaplansky's condition introduced in \cite{kap}, Thm. 5.2, the universal
algebra is the $\mathcal{U}_{3}$-algebra $\mathcal{K}_{3}:=\mathcal{U}_{3}[p,q,s,t]/(1-ps\bar{x}_{1}-pt\bar{x}_{2}-qt\bar{x}_{3}).$
Similarly, for statement $\circled{\textup{3}}$, the universal algebra is
the $\mathcal{U}_{4}$-algebra 
\begin{equation*}
\mathcal{K}_{4}:=\mathcal{U}_{4}[p,q,s,t]/(1-ps\bar{x}_{1}-pt\bar{x}_{2}-qs\bar{x}_{3}-qt\bar{x}_{4}).
\end{equation*}With $n\in \{3,4\}$, as in the proof of Theorem \ref{TH7}(3), one checks
that $\mathcal{K}_{n}$ is a smooth $\mathcal{U}_{n}$-algebra of relative
dimension $3$. From Theorems \ref{TH5} and \ref{TH8}, if $n=4$,
and from \cite{kap}, Thm. 5.2, if $n=3$, it follows that a B\'{e}zout domain 
$R$ is an \textsl{EDD} iff each homomorphism $\mathcal{U}_{n}\rightarrow R$ extends to a homomorphism 
$\mathcal{K}_{n}\rightarrow R$.
\end{remark}

\section{Stability properties}

\label{S7}

In order to provide more examples of rings which are or are not ${E}_2$ 
or ${SE}_2$ rings, we include the following stability properties.

\begin{proposition}
\label{PR6} \textbf{(1)} Let $I$ be an ideal of $R$ contained in $J(R)$
(e.g., $R$ is complete in the $I$-adic topology). Then $R$ is $E_{2}$ (or $SE_{2}$ or $U_2$) ring iff so is $R/I$.

\smallskip \textbf{(2)} Let $n\in \mathbb{N}$. Then $R$ is $E_{2}$ (or $SE_{2}$) ring iff so is the ring of formal power series $R[[x_{1},\ldots ,x_{n}]]$.

\smallskip \textbf{(3)} Let $R\rightarrow S$ be an inclusion of rings that
has a retraction. If $R$ is not an ${E}_2$ (resp.\ ${SE}_2$) ring,
then $S$ is not an ${E}_2$ (resp.\ ${SE}_2$) ring.

\smallskip \textbf{(4)} If $I$ is a nil ideal of $R$, then $R$ is $Z_{2}$ ring 
iff so is $R/I$.

\smallskip \textbf{(5)} Let $I$ be an ideal of $R$ contained in $J(R)$. If $R/I$ is a $\Pi_2$ ring, then so is $R$, and the converse holds if $I$ is a
nil ideal.
\end{proposition}

\begin{proof}
(1) For $E_{2}$ and $SE_{2}$ this follows from the property (3) of
Subsection \ref{S21} and Theorem \ref{TH8}. For $U_2$ this follows
as for each ideal $J$ of $R$, $U(R/J)$ is the inverse image of $U\bigl(R/(J+I)\bigr)$ via the epimorphism $R/J\rightarrow R/(J+I)$.

(2) Follows from part (1), in the special case $I=(x_{1},\ldots ,x_{n})$.

(3) Let $A\in Um(\mathbb{M}_{2}(R))$ be not extendable (resp.\ not
simply extendable). We show that the assumption that $A$ has an extension
(resp.\ a simple extension) $A^{+}\in SL_3(S)$, leads to a
contradiction. But if $\tau :S\rightarrow R$ is the retraction, then the matrix $\tau
(A^{+})\in SL_3(R)$, obtained from $A^+$ by applying $\tau$ to its entries, is an extension (resp.\ a simple extension) of $A$, a contradiction.

(4) Based on the equivalence $\circled{\textup{5}}\Leftrightarrow\circled{\textup{7}}$ (see Theorem \ref{TH8}) it suffices to show that for each $\upsilon \in Um(R^{2})$, every $R$-algebra
homomorphism ${Z}_{\upsilon }\rightarrow R/I$ lifts to a retraction ${Z}_{\upsilon }\rightarrow R$. To check this, we can assume that $R$ is
noetherian, and the statement follows from the smoothness part in Theorem \ref{TH7}(2).

(5) Let $P$ be a projective $R$-module of rank $1$ generated by two
elements. If $R/I$ is a $\Pi _{2}$ ring, then from Theorem \ref{TH1} it
follows that $P/IP$ is free of rank $1$. As $R\rightarrow R/I$ and $P\rightarrow P/IP$ are projective covers (see \cite{bou}, Sect. 9, Prop. 10
and Cor. 2 after it), from the uniqueness up to isomorphisms of projective
covers (see \cite{bou}, Sect. 9, Cor. 1 after Prop. 8), it follows that
there exists an $R$-linear isomorphism $R\rightarrow P$. Thus from Theorem \ref{TH1} we infer that $R$ is a $\Pi _{2}$ ring. From \cite{bou}, Sect. 9,
Prop. 11 it follows that the converse holds if $I$ is a nilpotent ideal and actually
even if $I$ is only a nil ideal, as our projective modules are finitely
generated, hence finitely presented and thus are defined over finitely
generated $\mathbb{Z}$-subalgebras of $R$.
\end{proof}

\begin{example}
\normalfont\label{EX15} If $R$ is an \textsl{EDR}, then from Example \ref{EX1} and Proposition \ref{PR6}(2) it follows that for all $n\in \mathbb{N}$, $R[[x_{1},\ldots ,x_{n}]]$ is an $SE_{2}$ ring.
\end{example}

\begin{example}
\normalfont\label{EX16} From Example \ref{EX11} and Proposition \ref{PR6}(3) it follows that for all $n\in \mathbb{N}$ the polynomial ring $\mathbb{Z}[x_{1},\ldots ,x_{n}]$ is not an $E_2$ ring.
\end{example}

\begin{example}
\normalfont\label{EX17} Let $R$ be a Hermite ring which is a $\Pi _{2}$ ring
of Krull dimension at most 1. We check that $R$ is an \textsl{EDR}. Based on
Proposition \ref{PR6}(1) and Theorem \ref{TH5} we can assume that $R$ is
reduced. Based on Theorem \ref{TH5} and Corollary \ref{C11} it suffices
to show that for each matrix $A=\left[ 
\begin{array}{cc}
a & b \\ 
c & d\end{array}\right] \in \mathbb{M}_{2}(R)$, statement \circled{\textup{5}} holds for it.
As $R$ is a $\Pi _{2}$ ring, we can assume $\det (A)\neq 0$. The set ${Min}A\setminus V(R\det (A))$ is clopen in ${Min}A$ and hence there
exist radical ideals $\mathcal I_{1}$ and $\mathcal I_{2}$ such that ${Min}A\setminus
V(R\det (A))={Min}A\cap V(\mathcal I_{1})$ and ${Min}A\cap V(R\det (A))={Min}A\cap V(\mathcal I_{2})$; as $R$ is reduced, we have $\mathcal I_{1}\cap \mathcal I_{2}=0$. The
element $\det (A)+\mathcal I_{1}\in R/\mathcal I_{1}$ is not contained in any minimal prime ideal of 
$R/\mathcal I_{1}$ which is not maximal, and hence $S:=R/(\mathcal I_{1}+R\det (A))$ has Krull
dimension $0$. The Hermite ring $S/N(S)$ is a von Newman regular ring and
thus it is a semihereditary ring (equivalently, it is reduced with compact
minimum spectrum, see \cite{end}, Thm. 2 and \cite{que}, Prop. 10). It is
well-known that these imply that $S$ is an \textsl{EDR} (e.g., see \cite{LLS}, Thm. 2.4). From this and Corollary \ref{C11} it follows that there
exists an $R$-algebra homomorphism $h_{1}:Z_{\upsilon }\rightarrow R/\mathcal I_{1}$. As the image of $\det (A)$ in $R/\mathcal I_{2}$ is $0$, as in the part of
the proof of Theorem \ref{TH8} that pertains to the implication $\circled{\textup{8}}\Rightarrow \circled{\textup{5}}$, we conclude that
statement \circled{\textup{5}} holds for $A$.
\end{example}

\begin{example}
\normalfont\label{EX18} Let $R$ be a B\'{e}zout ring with the property that $\bar R:=R/N(R)$ is a semihereditary ring. We check that $R$ is a $\Pi_2$
ring. Each projective $\bar R$-module $\bar P$ of rank $1$ generated by $2$
elements is isomorphic to a direct sum $\bar R/\bar Ra\oplus \bar R/R\bar b$
where $a,b\in\bar R$ (see \cite{sho}, Prop.) and from the rank $1$ property
it follows first that $(a,b)\in Um(\bar R^{2})$ and secondly that $\bar
Ra\cap\bar Rb=0$. Thus $\bar P\cong \bar R$ and from Theorem \ref{TH1} it
follows that $\bar R$ is a $\Pi _{2}$ ring. From Proposition \ref{PR6}(5)
it follows that $R$ itself is a $\Pi_2$ ring. In \cite{LLS}, Thm. 2.4 it is
proved that $\bar R$ is a Hermite ring and in \cite{cou}, Thm. III. 3 (see
also \cite{zab2}, Thm. 2.2.1) it is proved that $R$ itself is a Hermite ring.

We assume that moreover $R$ has Krull dimension $2$ and for each $a\in \bar R\setminus Z(\bar R)$, $\bar R/\bar Ra$ is a 
$\Pi _{2}$ ring. As $Z(\bar R)$ is the union of minimal
prime ideals of $\bar R$, $\bar R/\bar Ra$ has Krull dimension at most $1$. 
As $\bar R/\bar R a$ is a Hermite ring, it follows from Example \ref{EX17} that it is 
an \textsl{EDR}. From this and \cite{sho}, Thm. it follows that $\bar R$ is 
an \textsl{EDR}. From this and Theorem \ref{TH5} and Proposition \ref{PR6}(1) it follows that $R$ itself is an \textsl{EDR}.
\end{example}

\begin{example}
\normalfont\label{EX19} Let $I$ be an ideal of $R$ such that the pair $(R,I)$
is henselian. Each finitely generated projective $R/I$-module lifts to
a projective $R$-module (see {\url{https://stacks.math.columbia.edu/tag/0D4A}}), and the same holds if the rank and the number of generators are
prescribed. Thus, if $R$ is a $\Pi_2$ ring, then $R/I$ is a $\Pi_2$
ring.
\end{example}

\section{Open problems}

\noindent
\textbf{OP1.} Does there exist $E_2$ rings $R$ with $sr(R)\ge 3$? If yes,
is it true that all of them are $SE_2$ rings?

\noindent
\textbf{OP2.} Does there exist a $J_{2,1}$ ring which is not a $\Pi_2$ ring?

\noindent
\textbf{OP3.} Does there exist a Hermite ring which is a $\Pi_2$ ring but not an \textsl{EDR}?

\noindent
\textbf{OP4.} For n>2, one can define similarly $E_n$ and $SE_n$ rings and it is easy to see that each $E_n$ ring is an $E_{n-1}$ ring. Is it true that each $SE_n$ ring is an $SE_{n-1}$ ring?

\noindent
\textbf{OP5.} Does there exist examples of $V_2$ rings which are not $\Pi_2$
rings?

\noindent
\textbf{OP6.} What is the exact `connection' between $WW_2$, $W_2$, $WU_2$
and $U_2$ rings?

\noindent
\textbf{OP7.} Classify all elements $t$ of a B\'{e}zout domain $R$ with the
property that the $t$-adic completion of $R$ is a B\'{e}zout ring.

\noindent
\textbf{OP8.} Can the $\Pi_2$ assumptions of Examples \ref{EX17} and \ref{EX18} be eliminated?

\noindent
\textbf{OP9.} If $sr(R)\le 2$, find conditions on $R$ to guarantee that
each projective $R$-module of rank $1$ is generated by $2$ elements. [Like
this is true if $\dim(R/J(R))\le 1$.]

\noindent
\textbf{OP10.} Classify all rings $R$ such that one of the following five
statements holds: (1) if $R$ is a $\Pi _{2}$ ring, then $R/J(R)$ is a $\Pi
_{2}$ ring, (2) each projective $R/J(R)$-module of rank 1 generated by 2
elements lifts to a projective $R$-module of rank 1; (3) the reduction
homomorphism ${Pic}(R)\rightarrow {Pic}(R/J(R))$ is
surjective; (4) each finitely generated projective module over $R/J(R)$
lifts to a finitely generated projective $R$-module; (5) each projective
module over $R/J(R)$ lifts to a projective $R$-module.

\noindent
\textbf{OP11.} Classify all rings with the property that one of the
following two statements implies the other: (1) the matrix $A\in Um(\mathbb{M}_2(R))$ is extendable; (2) the matrix $A\in Um(\mathbb{M}_2(R))$ is
equivalent to a symmetric matrix.

\noindent
\textbf{OP12.} If $R$ is a reduced Hermite ring of prime characteristic, does 
the assumption that its perfection is an \textsl{EDR} imply that $R$ itself is
 an \textsl{EDR}?
 
\label{S8}

\bigskip \noindent \textbf{Acknowledgement.} The third author would like to
thank SUNY Binghamton for good working conditions and Ofer Gabber for
sharing the argument of the footnote and for pointing out the link of 
Example \ref{EX19} and a glitch in one of the references. Competing interests: the authors declare none.

\hbox{} \hbox{Grigore C\u{a}lug\u{a}reanu\;\;\;E-mail: calu@math.ubbcluj.ro}
\hbox{Address: Department of Mathematics,} 
\hbox{Babe\c{s}-Bolyai
University, Cluj-Napoca, Romania.}

\hbox{} \hbox{Horia F. Pop\;\;\;E-mail: hfpop@cs.ubbcluj.ro} 
\hbox{Address:
Department of Computer Science,} 
\hbox{Babe\c{s}-Bolyai
University, Cluj-Napoca, Romania.}

\hbox{} \hbox{Adrian Vasiu,\;\;\;E-mail: avasiu@binghamton.edu} 
\hbox{Address:
Department of Mathematics and Statistics, Binghamton University,} 
\hbox{P. O. Box
6000, Binghamton, New York 13902-6000, U.S.A.}

\end{document}